\documentclass[11pt]{amsart}
\usepackage[utf8]{inputenc}
\usepackage[pdftex]{graphicx}
\usepackage{amssymb}
\usepackage{amsmath}
\usepackage{amsfonts}
\usepackage{amsthm}
\usepackage{amssymb}
\usepackage{bm}
\usepackage{mathrsfs}
\usepackage{appendix}
\usepackage{enumerate}
\usepackage[left=3.2cm, right=3.2cm, bottom=3.6cm]{geometry}
\usepackage[stretch=30,shrink=30]{microtype}
\usepackage[backend=bibtex,giveninits,sorting=nyvt,doi=false,isbn=false,url=false]{biblatex}
\makeatletter
\def\th@plain{%
	\thm@notefont{}
	\itshape 
}
\def\th@definition{%
	\thm@notefont{}
	\normalfont 
}
\makeatother

\numberwithin{equation}{section}

\newtheorem{theorem}{Theorem}[section]

\newtheorem{proposition}[theorem]{Proposition}
\newtheorem{corollary}[theorem]{Corollary}

\newtheorem{lemma}[theorem]{Lemma}

\theoremstyle{definition}

\theoremstyle{remark}
\newtheorem{remark}[theorem]{Remark}

\newcommand{\de}{\partial}
\newcommand{\N}{\mathbb{N}}
\newcommand{\Z}{\mathbb{Z}}
\newcommand{\R}{\mathbb{R}}
\newcommand{\C}{\mathbb{C}}

\newcommand{\ang}[1]{\langle #1 \rangle}

\DeclareMathOperator{\spt}{spt}
\DeclareMathOperator{\dvol}{dvol}
\DeclareMathOperator{\dist}{dist}

\DeclareMathOperator{\Maps}{Maps}
\DeclareMathOperator{\vol}{vol}

\begin{document}
\title[$\Gamma$-convergence for the self-dual $U(1)$-Yang--Mills--Higgs]{Convergence of the self-dual $\bm{U(1)}$-Yang--Mills--Higgs \\energies to the $\bm{(n-2)}$-area functional}
\author[D. Parise]{Davide Parise}
\address{University of Cambridge, Department of Pure Mathematics and Mathematical Statistics, Wilberforce Road, Cambridge CB3 0WA, United Kingdom.}
\email{dp588@cam.ac.uk}
\author[A. Pigati]{Alessandro Pigati}
\address{Courant Institute of Mathematical Sciences, New York University, 251 Mercer Street, New York, NY 10003, United States of America.}
\email{ap6968@nyu.edu}
\author[D. Stern]{Daniel Stern}
\address{University of Chicago, Department of Mathematics, 5734 S University Avenue, Chicago, IL 60637, United States of America.}
\email{dstern@uchicago.edu}
\date{\today}
\begin{abstract}
Given a hermitian line bundle $L\to M$ on a closed Riemannian manifold $(M^n,g)$, the self-dual Yang--Mills--Higgs energies are a natural family of functionals
\begin{align*}
	&E_\epsilon(u,\nabla):=\int_M\Big(|\nabla u|^2+\epsilon^2|F_\nabla|^2+\frac{(1-|u|^2)^2}{4\epsilon^2}\Big)
\end{align*}
defined for couples $(u,\nabla)$ consisting of a section $u\in\Gamma(L)$ and a hermitian connection $\nabla$ with curvature $F_\nabla$. While the critical points of these functionals have been well-studied in dimension two by the gauge theory community, it was shown in \cite{PigatiStern} that critical points in higher dimension converge as $\epsilon\to 0$ (in an appropriate sense) to minimal submanifolds of codimension two, with strong parallels to the correspondence between the Allen--Cahn equations and minimal hypersurfaces. 

In this paper, we complement this idea by showing the $\Gamma$-convergence of $E_\epsilon$ to ($2\pi$ times) the codimension two area: more precisely,
given a family of couples $(u_\epsilon,\nabla_\epsilon)$ with $\sup_\epsilon E_\epsilon(u_\epsilon,\nabla_\epsilon)<\infty$, we prove that a suitable gauge invariant Jacobian $J(u_\epsilon,\nabla_\epsilon)$ converges to an integral $(n-2)$-cycle $\Gamma$, in the homology class dual to the Euler class $c_1(L)$, with mass $2 \pi \mathbb{M}(\Gamma)\le\liminf_{\epsilon \rightarrow 0}E_\epsilon(u_\epsilon,\nabla_\epsilon)$.
We also obtain a recovery sequence, for any integral cycle in this homology class.

Finally, we apply these techniques to compare min-max values for the $(n-2)$-area from the Almgren--Pitts theory with those obtained from the Yang--Mills--Higgs framework, showing that the former values always provide a lower bound for the latter. As an ingredient, we also establish a Huisken-type monotonicity result along the gradient flow of $E_{\epsilon}$.
\end{abstract}

\maketitle


\section{Introduction}


\subsection{Background and motivation}
The discovery in the late 1970's of deep connections between minimal hypersurfaces and the Allen--Cahn equations opened up a rich line of investigation, shedding light onto the structure of solutions of semilinear elliptic equations and the existence theory for minimal hypersurfaces. Like minimal hypersurfaces, which arise as critical points of the $(n-1)$-area functional, solutions of the Allen--Cahn equations
\begin{equation}\label{AC}
\epsilon \Delta u=\frac{1}{\epsilon}W'(u)
\end{equation}
(where $\epsilon>0$ and $W: \mathbb{R}\to [0,\infty)$ is a double-well potential) arise naturally as critical points for the Allen--Cahn energies
$$F_{\epsilon}(u):=\int_{\Omega}\Big(\frac{\epsilon}{2}|du|^2+\frac{W(u)}{\epsilon}\Big)$$
on $W^{1,2}(\Omega,\mathbb{R})$. A recurring theme in the study of the correspondence between solutions of \eqref{AC} and minimal hypersurfaces is the convergence not only of critical points, but of the \emph{variational theory} for the functionals $F_{\epsilon}$ to that of the $(n-1)$-area on the space of $(n-1)$-boundaries as $\epsilon\to 0$. The earliest results in this direction were obtained by Modica and Mortola \cite{ModicaMortola} who established the asymptotic convergence of $F_{\epsilon}$ to (a constant multiple of) the perimeter functional for Caccioppoli sets, in the framework of \emph{$\Gamma$-convergence} introduced a few years earlier by De Giorgi \cite{DeGiorgiFranzoni}. De Giorgi's $\Gamma$-convergence provides a natural weak notion of convergence for variational problems involving a singular perturbation, well-suited to establishing convergence of minimizers to minimizers (see \cite{Braides} and \cite{DalMaso} for a contemporary treatment of $\Gamma$-convergence, and \cite{Alberti} for its application to the study of phase transitions). The work of Modica--Mortola was later generalized by Modica \cite{Modica} and Sternberg \cite{Sternberg}, in their resolution of some conjectures of Gurtin \cite{Gurtin}.

While the $\Gamma$-convergence results of \cite{Modica}, \cite{ModicaMortola}, and \cite{Sternberg} imply that $F_{\epsilon}$-minimizing solutions of \eqref{AC} (rather, their level sets and energy measures) converge to area-minimizing hypersurfaces, a series of results obtained over the last five years \cite{Dey, GasparGuaraco, Guaraco} show that the min-max theory for the Allen--Cahn functionals $F_{\epsilon}$ likewise converges to the min-max theory for the area functional on $(n-1)$-boundaries in the geometric measure theory framework developed by Almgren and Pitts \cite{Pitts}. Building on the analytic work of \cite{HutchinsonTonegawa} and \cite{TonWick}, these and related results have established the min-max theory for the Allen--Cahn functionals as a valuable regularization of the Almgren--Pitts min-max construction of minimal hypersurfaces, finding use, for instance, in Chodosh and Mantoulidis's work on the Multiplicity One conjecture in three-manifolds \cite{ChoMan}.

In view of these and other applications, it is natural to seek an analogous correspondence between certain geometric elliptic systems and minimal submanifolds of higher codimension. In \cite{PigatiStern}, the second- and third-named authors proposed a natural analog in codimension two, with the role of the Allen--Cahn equations taken on by a well-studied family of elliptic systems from gauge theory. Specifically, \cite{PigatiStern} considers the \emph{self-dual $U(1)$-Yang--Mills--Higgs} energies: the gauge-invariant functionals $E_{\epsilon}(u,\nabla)$ acting on a section $u\in \Gamma(L)$ and metric-compatible connection $\nabla$ on a hermitian line bundle $L\to M$ by
$$E_{\epsilon}(u,\nabla):=\int_M\Big(|\nabla u|^2+\epsilon^2|F_{\nabla}|^2+\frac{1}{4\epsilon^2}(1-|u|^2)^2\Big).$$
The functionals $E_{\epsilon}$ are distinguished from \emph{formally} similar functionals---such as $\int_M(|\nabla u|^2+\lambda |F_{\nabla}|^2+\frac{1}{4\epsilon^2}(1-|u|^2)^2)$ for $\lambda\neq \epsilon^2$---by their so-called \emph{self-duality}: namely, $E_{\epsilon}$ enjoys additional symmetry properties, such that minimizers of $E_{\epsilon}$ for bundles $L\to \Sigma^2$ over a Riemann surface $\Sigma^2$ satisfy a special first-order system known as the \emph{vortex equations}. 

The study of these functionals has a long history, which we do not attempt to survey here. In his thesis work \cite{Taubes1, Taubes2}, Taubes classified finite-energy critical points of $E_{\epsilon}$ for the trivial bundle $L\cong \mathbb{C}\times \mathbb{R}^2\to \mathbb{R}^2$, showing that all such critical points satisfy the first-order vortex equations, are determined---up to gauge equivalence---by the finite zero set $u^{-1}\{0\}\subset \mathbb{C}$ (counted with multiplicity), and have quantized energy $E_{\epsilon}(u,\nabla)=2\pi N\in 2\pi \mathbb{N}$ corresponding to the mass of the zero set $N=|u^{-1}\{0\}|$ (see \cite{Taubes1}, \cite{Taubes2}, and \cite{JaffeTaubes} for details). The asymptotic analysis as $\epsilon\to 0$ of the rescaled functionals $E_{\epsilon}$ was first taken up by Hong, Jost, and Struwe, who showed in \cite{HongJostStruwe} that for minimizers $(u_{\epsilon},\nabla_{\epsilon})$ of $E_{\epsilon}$ on line bundles $L\to \Sigma^2$ over a Riemann surface $\Sigma$, energy and curvature concentrate (subsequentially) as $\epsilon\to 0$ at a collection of $|\deg(L)|$ points in $\Sigma$, outside of which $u_{\epsilon}$ converges to a unit section $u_0$ and $\nabla_{\epsilon}$ to a flat connection $\nabla_0$ for which $\nabla_0u_0=0$. 

The results of \cite{PigatiStern} provide a far-reaching generalization of Hong--Jost--Struwe's analysis, characterizing the limiting behavior of arbitrary critical points on line bundles over a base manifold $M^n$ of general dimension. Namely, it is shown in \cite{PigatiStern} that for sequences $(u_{\epsilon},\nabla_{\epsilon})$ of critical points satisfying a uniform energy bound $E_{\epsilon}(u_{\epsilon},\nabla_{\epsilon})\leq C$, the energy densities 
$$e_{\epsilon}(u_{\epsilon},\nabla_{\epsilon}) := |\nabla_{\epsilon}u_{\epsilon}|^2+\epsilon^2|F_{\nabla_{\epsilon}}|^2+\frac{(1-|u_{\epsilon}|^2)^2}{4\epsilon^2}$$
converge subsequentially weakly in $(C^0)^*$ to (the weight measure of) a \emph{stationary integral $(n-2)$-varifold} $V$ in $M$---i.e., a (possibly singular) minimal variety of codimension two. In particular, this gives a codimension-two analog to the results of Hutchinson--Tonegawa \cite{HutchinsonTonegawa} for the Allen--Cahn equations, showing that critical points for $E_{\epsilon}$ converge cleanly to critical points of the $(n-2)$-area functional in the $\epsilon\to 0$ limit. We note, moreover, that the analysis in \cite{PigatiStern} depends strongly on the specific choice of coupling constants in the definition of $E_{\epsilon}$, suggesting that the \emph{self-dual} $U(1)$-Yang--Mills--Higgs energies provide more or less the \emph{unique} codimension-two analog for the Allen--Cahn energies, at least among similar functionals of Yang--Mills--Higgs type. 

\begin{remark}
In particular, the convergence behavior for critical points $(u_{\epsilon},\nabla_{\epsilon})$ of $E_{\epsilon}$ in the $O(1)$ energy regime is considerably simpler than its counterpart for the non-gauged Ginzburg--Landau energies
$$G_{\epsilon}: W^{1,2}(M,\mathbb{C})\to \mathbb{R},\quad G_{\epsilon}(u):=\int_M\Big(|du|^2+\frac{(1-|u|^2)^2}{4\epsilon^2}\Big)$$
in the $O(|\log\epsilon|)$ energy regime, whose critical points in general exhibit \emph{partial} energy concentration along a stationary, \emph{rectifiable} (not necessarily integral) $(n-2)$-varifold (cf.\ \cite{BethuelBrezisHelein, BethuelBrezisOrlandi,Cheng,ChaoYan,Chun,LinRiviere,LinRiviere2,Sandier1,Stern,Struwe} for details of the asymptotic analysis of the complex Ginzburg--Landau equations, as well as \cite{BethuelRiviere1994, BethuelRiviere1995, Riviere,Riviere1999,Riviere2002,SandierSerfaty} for related results for other functionals of Yang--Mills--Higgs type whose behavior resembles that of $G_{\epsilon}$). As remarked in \cite{PigatiStern}, the variational theory for the functionals $G_{\epsilon}$ is best understood as a relaxation of that for the Dirichlet energy on singular $S^1$-valued maps, and its relation to geometric measure theory and minimal submanifolds is subtle, and qualitatively quite different from that of the Allen--Cahn or self-dual Yang--Mills--Higgs energies. 
\end{remark}

Building on the ideas of \cite{PigatiStern}, the aim of the present paper is to understand to what extent the \emph{variational theory} for the functionals $E_{\epsilon}$ converges to that of the $(n-2)$-area, in the spirit of similar results for the Allen--Cahn functionals. Our chief analytic result provides a key step toward answering this question, establishing the \emph{$\Gamma$-convergence} of the functionals $E_{\epsilon}$ for pairs $(u,\nabla)$ on a hermitian line bundle $L\to M$ to the mass functional on the space of integral $(n-2)$-cycles dual to $c_1(L)$. This convergence result---whose precise formulation we give in the following subsection---may be thought of as a \emph{codimension-two analog of the classical results of Modica and Mortola}; and in spite of the very different setting, its proof bears a surprising resemblance to the original arguments in \cite{ModicaMortola}. In addition to implying the convergence of $E_{\epsilon}$-minimizing pairs $(u_{\epsilon},\nabla_{\epsilon})$ to area-minimizing $(n-2)$-cycles, the $\Gamma$-convergence framework---together with some topological arguments---allows us to compare the energy of min-max critical points for $E_{\epsilon}$ to the areas of corresponding min-max minimal varieties, along the lines of the comparison results for the Allen--Cahn and Almgren--Pitts min-max constructions obtained in \cite{GasparGuaraco,Guaraco}.

\subsection{Convergence results for the self-dual Yang--Mills--Higgs energies}
Let $L\to M^n$ be a hermitian line bundle over a closed, oriented Riemannian manifold $(M^n,g)$. Given a metric connection $\nabla$ on $L$, recall that the curvature $F_{\nabla}\in \Omega^2(M)\otimes \mathfrak{so}(L)$ is given by
\begin{equation}
F_{\nabla}(X,Y)u:=[\nabla_X,\nabla_Y]u-\nabla_{[X,Y]}u=-i\omega_{\nabla}(X,Y)u
\end{equation}
for some two-form $\omega_{\nabla}\in \Omega^2(M)$, which we will frequently identify with $F_{\nabla}$ when there is no confusion. Given a pair $(u,\nabla)$ consisting of a section $u\in \Gamma(L)$ and metric connection $\nabla$, we define as in \cite{PigatiStern} the two-form $\psi(u,\nabla)\in\Omega^2(M)$ by
$$\psi(u,\nabla)(X,Y):=2\langle i\nabla_Xu,\nabla_Yu\rangle,$$
which is easily seen to satisfy the pointwise bound $|\psi(u)|\leq |\nabla u|^2$ (cf.\ \cite[Section 2]{PigatiStern}). For the $\Gamma$-convergence results, we will be particularly interested in the two-forms
\begin{equation}\label{jdef}
J(u,\nabla):=\psi(u,\nabla)+(1-|u|^2)\omega_{\nabla}=d\langle \nabla u,iu\rangle+\omega_{\nabla},
\end{equation}
whose role should be compared to that of the one-forms $\sqrt{2W(v)}\cdot dv$ for real-valued functions $v:M\to \mathbb{R}$ in the work of Modica--Mortola \cite{ModicaMortola}. 

As with any $\Gamma$-convergence result, our main theorem consists of two parts. First, we show that for any family of pairs $(u_{\epsilon},\nabla_{\epsilon})$ with 
$$\sup_{\epsilon>0}E_{\epsilon}(u_{\epsilon},\nabla_{\epsilon})\leq \Lambda<\infty,$$
there exists a subsequence $(u_{\epsilon_j},\nabla_{\epsilon_j})$, with $\epsilon_j\to 0$, to which we can associate a limiting integral $(n-2)$-cycle $\Gamma$ with $2\pi \mathbb{M}(\Gamma)\leq \Lambda$. Second, we show that any integral $(n-2)$-cycle dual to $c_1(L)$ can be obtained in this way. More precisely, we have the following.

\begin{theorem}[$\bm{\Gamma}$-convergence]\label{GammaConvThm}
For a hermitian line bundle $L \rightarrow M$ as above, the following hold:
\begin{enumerate}[(i)]
\item \textup{Liminf inequality}. Given a family $(u_\epsilon, \nabla_\epsilon)$ of smooth sections with $|u_{\epsilon}|\leq 1$ and metric connections with uniformly bounded energies $E_\epsilon(u_\epsilon, \nabla_\epsilon) \leq \Lambda$, there exists an integral $(n - 2)$-cycle $\Gamma$ Poincar\'e dual to the Euler class $c_1(L) \in H^2(M; \mathbb{Z})$ such that, up to a subsequence, 
\begin{equation*}
J(u_\epsilon, \nabla_\epsilon) \rightharpoonup 2\pi\Gamma, \quad \text{as $\epsilon \rightarrow 0$}, 
\end{equation*}
as currents. Moreover, the following liminf inequality holds:
\begin{equation*}
2 \pi \mathbb{M}(\Gamma) \leq \liminf_{\epsilon \rightarrow 0} E_\epsilon(u_\epsilon, \nabla_\epsilon). 
\end{equation*}
\item \textup{Recovery sequence}. Given an integral $(n - 2)$-cycle $\Gamma$ whose homology class $[\Gamma]\in H_{n-2}(M;\Z)$ is dual to $c_1(L) \in H^2(M; \mathbb{Z})$, there exists a family $(u_\epsilon, \nabla_\epsilon)$ of smooth sections and connections on $L$ such that 
\begin{equation*}
J(u_\epsilon, \nabla_\epsilon) \rightharpoonup 2 \pi \Gamma, \quad \text{as $\epsilon \rightarrow 0,$}
\end{equation*}
as currents, and
\begin{equation*}
\lim_{\epsilon \rightarrow 0} E_\epsilon(u_\epsilon, \nabla_\epsilon) = 2 \pi \mathbb{M}(\Gamma). 
\end{equation*}
\end{enumerate}
\end{theorem}

\begin{remark}\label{curv.rk}
	Since the curvature forms $\omega_\epsilon:=i F_{\nabla_\epsilon}$ satisfy
	$J(u_\epsilon,\nabla_{\epsilon})=\omega_\epsilon+d\langle\nabla_{\epsilon}u_\epsilon,iu_\epsilon\rangle$,
	if $E_{\epsilon}(u_{\epsilon},\nabla_{\epsilon})=O(1)$, the boundedness of $\langle\nabla_{\epsilon}u_\epsilon,iu_\epsilon\rangle$ in $L^2(M)$ together with part (i) above implies that the curvatures $\omega_{\epsilon}$ also have a subsequential limit as currents.
	Simple examples show that this limit may fail to coincide with $2\pi\Gamma$ under our assumptions. However, assuming that $\nabla_\epsilon$ is critical for the energy
	$E_\epsilon(u_\epsilon,\cdot)$---hence, a minimizer by convexity of $E_\epsilon$ in $\nabla_{\epsilon}$---the corresponding Euler--Lagrange equation \eqref{EL} gives
	$\langle\nabla_{\epsilon}u_\epsilon,iu_\epsilon\rangle\rightharpoonup 0$, since $\epsilon^2\omega_\epsilon\to 0$ in $L^2(M)$. Thus, in this case
	\begin{equation*}
		2\pi\Gamma=\lim_{\epsilon\to 0}\omega_\epsilon
	\end{equation*}
	as currents. Together with Corollary \ref{min.cvg.cor} below, this implies that for a sequence of \emph{minimizers} $(u_\epsilon,\nabla_{\epsilon})$, the curvature forms $\frac{1}{2\pi}\omega_\epsilon$ converge subsequentially to an integral, area-minimizing cycle $\Gamma$ whose associated varifold agrees with the energy concentration varifold $V$ from \cite[Theorem~1.1]{PigatiStern} (up to a subsequence). This answers a question raised in \cite{PigatiStern}.
\end{remark}

Readers familiar with the $\Gamma$-convergence theory developed for the normalized Ginzburg--Landau functionals $\frac{G_{\epsilon}}{|\log\epsilon|}$ in recent decades (see in particular \cite{ABO2,ABO1, AlicandroPonsiglione, CanevariOrlandi2, JerrardSoner}) will notice some formal similarities between the above result and analogs for the functionals $\frac{G_{\epsilon}}{|\log\epsilon|}$. Namely, the results of \cite{ABO1} and \cite{JerrardSoner} show that for any complex-valued map $u:M\to \mathbb{C}$ with 
$$G_{\epsilon}(u)\leq 2\pi \Lambda \log(1/\epsilon)$$
and $0<\epsilon\ll 1$ sufficiently small, the Jacobian 2-form
$$Ju:=2du^1\wedge du^2$$
(which coincides with both $\psi(u,\nabla)$ and $J(u,\nabla)$ when $\nabla$ is the standard flat connection on the trivial bundle) is weakly close to ($2\pi$ times) an integral $(n-2)$-boundary $\Gamma$ of mass $\mathbb{M}(\Gamma)\leq \Lambda+o(1)$. The proof requires some delicate analysis: in particular, the mass $\|Ju\|_{L^1}$ of the Jacobians themselves is \emph{not} bounded in general by the energy $\frac{G_{\epsilon}(u)}{|\log\epsilon|}$ for small $\epsilon$, and the proof of the associated $\Gamma$-convergence result relies instead on a subtle application of the degree estimates of Sandier \cite{Sandier2} and Jerrard \cite{Jerrard} (see also \cite{SandierSerfaty2, SerfatyTice}).

In our setting, by contrast, the two-forms $J(u,\nabla)$ are easily seen to enjoy a pointwise bound
\begin{equation}\label{j.pt.bd}
|J(u,\nabla)|\leq |\nabla u|^2+(1-|u|^2)|F_{\nabla}|\leq |\nabla u|^2+\epsilon^2|F_{\nabla}|^2+\frac{1}{4\epsilon^2}(1-|u|^2)^2
\end{equation}
by the energy integrand $e_{\epsilon}(u,\nabla)$, so that $\|J(u,\nabla)\|_{L^1}\leq E_{\epsilon}(u,\nabla)$ automatically. As a consequence, to prove part (i) of Theorem \ref{GammaConvThm}, the only challenge lies in showing that the limiting $(n-2)$-cycle $\Gamma$ is integer rectifiable (and lies in the correct homology class). 

To achieve this, we establish a compactness result for sections $u_{\epsilon}\in \Gamma(L)$ with $E_{\epsilon}(u_{\epsilon},\nabla_{\epsilon})=O(1)$, showing that they converge subsequentially (after change of gauge) to a singular unit section, whose topological singular set $\Gamma$ coincides with the limit of $\frac{1}{2\pi}J(u_{\epsilon},\nabla_{\epsilon})$. These singular unit sections (modulo the action of the gauge group) provide a natural codimension-two analog of Caccioppoli sets, and it is not difficult to see that their topological singular sets are integral $(n-2)$-cycles (indeed, this is a consequence of results in \cite{ABO1} and \cite{JerrardSoner2}). Again, we note that the broad outlines of the argument are very much reminiscent of those in \cite{ModicaMortola} for the Allen--Cahn energies, with the bound \eqref{j.pt.bd} playing the role of the simple estimate $|\sqrt{2W(v)}\cdot dv|\leq \frac{\epsilon}{2}|dv|^2+\frac{W(v)}{\epsilon}$ for real-valued functions $v:M\to \mathbb{R}$.

\subsection{Applications to the study of minimizers and min-max constructions}
As an immediate corollary of Theorem \ref{GammaConvThm}, we see that minimizers of $E_{\epsilon}$ converge to homologically area-minimizing $(n-2)$-cycles, answering a question raised in \cite{PigatiStern}.

\begin{corollary}\label{min.cvg.cor} Let $L\to M$ be a nontrivial hermitian line bundle over a closed, oriented $n$-manifold $(M^n,g)$. If $(u_{\epsilon},\nabla_{\epsilon})$ minimize $E_{\epsilon}(u,\nabla)$ among all pairs $(u,\nabla)$ on $L$, then 
\begin{equation}\label{energ.lim}
\lim_{\epsilon\to 0}E_{\epsilon}(u_{\epsilon},\nabla_{\epsilon})=2\pi\min\{\mathbb{M}(\Gamma)\mid \Gamma\in \mathcal{Z}_{n-2}(M;\mathbb{Z})\text{ Poincar\'{e} dual to }c_1(L)\},
\end{equation}
and along a subsequence $\epsilon=\epsilon_j\to 0$, we have weak convergence
$$\lim_{\epsilon\to 0}J(u_{\epsilon},\nabla_{\epsilon})=\lim_{\epsilon\to 0}\omega_{\nabla_{\epsilon}}=2\pi \Gamma$$
of $J(u_{\epsilon},\nabla_{\epsilon})$ and the curvatures $\omega_{\nabla_{\epsilon}}$ to an $(n-2)$-cycle $\Gamma$ minimizing mass in the homology class dual to $c_1(L)$.
\end{corollary}

With Theorem \ref{GammaConvThm} in place, the proof of the corollary follows standard lines: by part (i) of the theorem, we know that the forms $J(u_{\epsilon},\nabla_{\epsilon})$ for a minimizing family $(u_{\epsilon},\nabla_{\epsilon})$ converge subsequentially to an integral $(n-2)$-cycle $\Gamma$, in the correct homology class, of mass $\mathbb{M}(\Gamma)\leq \frac{1}{2\pi}\liminf_{\epsilon\to 0} E_{\epsilon}(u_{\epsilon},\nabla_{\epsilon})$, providing one inequality in \eqref{energ.lim}. The opposite inequality follows from part (ii) of the theorem, which guarantees the existence of a recovery sequence $(u_{\epsilon},\nabla_{\epsilon})$ for a mass-minimizing cycle $\Gamma$. The convergence of the curvature two-forms $\omega_{\nabla_{\epsilon}}$ follows from the discussion in Remark \ref{curv.rk}.

For the min-max applications, we will focus on the trivial bundle $L=\mathbb{C}\times M\to M$ over a given closed, oriented $(M^n,g)$. We then consider a Banach space $X$ consisting of pairs $(u,\nabla=d-i\alpha)$, equipped with an appropriate norm, with respect to which $E_{\epsilon}$ is a smooth functional satisfying a variant of the Palais--Smale condition (as in Section 5 below or Section 7 of \cite{PigatiStern}). Removing from $X$ the degenerate set
$$X_0:=\{(u,\nabla)\in X : u\equiv 0\}$$
(on which $E_{\epsilon}\sim 1/\epsilon^2$ blows up as $\epsilon\to 0$), we see that the action of the gauge group of maps $\mathcal{G}=\Maps(M,S^1)$ given by
$$\phi \cdot (u,\nabla) := (\phi \cdot u, \nabla -i \phi^*(d\theta))$$
restricts to an action on the complement $X\setminus X_0$. 

For the purposes of intuition, we can view the gauge-invariant functionals $E_{\epsilon}$ as functions on the moduli space
$$\mathcal{M}:=(X\setminus X_0)/\mathcal{G},$$
whose topology may be compared with that of the space
$$Z:=\partial {\bf I}_{n-1}(M;\mathbb{Z}) \subseteq \mathcal{Z}_{n-2}(M;\mathbb{Z})$$
of integral $(n-2)$-boundaries in $M$, equipped with the flat metric. Indeed, we claim (see Section 5) that there are geometrically natural isomorphisms between the homotopy groups
$$\Phi: \pi_k(\mathcal{M},*)\to \pi_k(Z,0),$$
where $*\in \mathcal{M}$ denotes the collection of pairs $(u,\nabla)\in X$ with $|u|\equiv 1$ and $\nabla u=0$, and $0\in Z$ is the $0$-cycle. (Intuitively, one can think of this isomorphism as being induced by the zero locus map $(u,\nabla)\mapsto u^{-1}\{0\}$, but of course this will not define a continuous map into $Z$ in practice.) This isomorphism is nontrivial only when $k=1$ or $2$.

For $k=1,2$, to any class $\alpha\in \pi_k(Z,0)$, one can associate a min-max width
\begin{equation}\label{width.def.0}
{\bf W}(\alpha):=\inf_{\psi\in \alpha}\sup_{x\in S^k}\mathbb{M}(\psi(x))
\end{equation}
for the $(n-2)$-area functional. In practice, we work with the (essentially equivalent) discretized variant ${\bf W}^*(\alpha)$ of these min-max widths introduced by Almgren and Pitts (see \cite{Pitts}, or Section 5 below), which correspond to the masses of stationary $(n-2)$-varifolds. Likewise, for each nontrivial class $\beta\in \pi_k(\mathcal{M},*)$ and $\epsilon>0$, one can consider the min-max energies
$$\mathcal{E}_{\epsilon}(\beta):=\inf_{F\in \beta}\max_{x\in S^k}E_{\epsilon}(F(x)),$$
which are realized as critical values of the functionals $E_{\epsilon}$. (In practice, rather than working with families in $\pi_k(\mathcal{M},*)$, in Section 5 we work equivalently with the families $[0,1]\to X$ and $\bar D^2\to X$ giving their lifts in the Banach space $X$.) In rough terms, the results of Section 5 yield the following comparison.

\begin{theorem}[Min-max comparison]\label{WidthCompThm}
Let $\mathcal{M}$ be the moduli space of pairs $(u,\nabla)$ with $u\not\equiv 0$ and $Z$ the space of integral $(n-2)$-boundaries as above. With respect to the aforementioned isomorphism $\Phi: \pi_k(\mathcal{M},*)\to \pi_k(Z,0)$, the min-max energies satisfy
\begin{equation}\label{width.comp}
\liminf_{\epsilon\to 0}\mathcal{E}_{\epsilon}(\beta)\geq {\bf W}^*(\Phi(\beta))
\end{equation}
for any $\beta\in \pi_k(\mathcal{M},*).$ In particular, the mass of the stationary integral $(n-2)$-varifold $V_{YMH}$ associated to the critical points $(u_{\epsilon},\nabla_{\epsilon})$ by the results of \cite{PigatiStern} is bounded below by the mass of the corresponding min-max $(n-2)$-varifold $V_{GMT}$ produced by Almgren's min-max construction.
\end{theorem} 

While we have restricted our attention here to the comparison of one- and two-parameter min-max constructions associated to the homotopy groups of $\mathcal{M}$ and $Z$, we believe that the techniques used in the proof of Theorem \ref{WidthCompThm} should apply to all natural min-max constructions for the energies $E_{\epsilon}$, with appropriate modifications to the topological part of the argument. In particular, while Theorem \ref{WidthCompThm} can be compared to \cite[Proposition 8.19]{Guaraco} in the Allen--Cahn setting, we expect that the same ideas can be used to prove an analog of \cite[Theorem 6.1]{GasparGuaraco} treating higher-parameter families detecting cohomology classes in $H^*(\mathcal{M};\mathbb{Z})$ of higher degree.

Moreover, let us point out that in the Allen--Cahn setting, Akashdeep Dey has recently succeeded in proving a bound in the opposite direction \cite{Dey}, concluding that the min-max energies for the Allen--Cahn functionals in fact \emph{coincide} with the corresponding Almgren--Pitts widths in the $\epsilon \to 0$ limit. Though establishing a codimension-two analog of Dey's bound for the self-dual Yang--Mills--Higgs functionals lies beyond the scope of the present paper, we optimistically conjecture that such an estimate should hold, giving equality in \eqref{width.comp}. 

A key element in the proof of the min-max comparison theorem is the  $L^2$ gradient flow associated to the Yang--Mills--Higgs energies: i.e., the following system of coupled nonlinear parabolic equations
\begin{align}
\left\{
\begin{aligned} 
\partial_t u_t&=-\nabla_t^*\nabla_tu_t+{\textstyle\frac{1}{2\epsilon^2}}(1-|u_t|^2)u_t, \\
\partial_t\alpha_t&=-d^*d\alpha_t+\epsilon^{-2}\langle iu_t,\nabla_tu_t\rangle, 
\end{aligned}
\right.
\end{align}
subject to some initial data $(u_0, \nabla_0=d-i\alpha_0)$. 
The necessity of its introduction is due to some technical difficulties emerging in the proof of Theorem \ref{WidthCompThm} when passing from maps continuous in the flat norm, which are given by the $\Gamma$-convergence theory, to maps continuous in the mass norm, the relevant ones in the Almgren--Pitts setting. Indeed, the former can exhibit a phenomenon called \textit{concentration of mass} whereby the energy density accumulates at small scales, preventing a direct application of the so-called \textit{interpolation theory} developed by Marques, Neves and collaborators, which would give a corresponding continuous map in the mass norm.

Since we expect the gradient flow of $E_\epsilon$ to approximate a (weak) mean curvature flow of codimension two, a Huisken-type monotonicity formula should be expected to hold, thus providing the desired $(n - 2)$-energy density bounds at all scales after running the flow for a fixed amount of time (uniformly in $\epsilon$). 
This provides us with a canonical regularization preventing concentration of mass, without increasing the total energy.
At the end of the paper, we check that the flow satisfies long-time existence, uniqueness and continuous dependence on the initial data. 



\subsection*{Acknowledgements} The authors thank Y. Liokumovich for answering questions about the interpolation results of \cite{LMN}. DS acknowledges the support of the National Science Foundation under grant DMS-2002055. DP acknowledges the support of the UK Engineering and Physical Sciences Research Council (EPSRC) grant EP/L016516/1 and would like to thank N. Wickramasekera for constant encouragement and interest in this work. 

\section{Notation and preliminaries} \label{SectionPreliminaries}

Let $(M^n,g)$ be a closed, oriented Riemannian manifold and let $L \rightarrow M$ be a complex line bundle over $M$, endowed with a hermitian structure $\langle \cdot, \cdot \rangle$. We will denote by $W \colon L \rightarrow \mathbb{R}$ the nonlinear potential 
\begin{equation*}
    W(u) := \frac{1}{4}(1-|u|^2)^2, 
\end{equation*}
and for a hermitian connection $\nabla$ on $L$, a section $u \in \Gamma(L)$, and a parameter $\epsilon \in (0, 1)$, we denote by $E_\epsilon(u, \nabla)$ the scaled Yang--Mills--Higgs energy 
\begin{equation} \label{e.bd}
    E_\epsilon(u, \nabla) := \int_M ( \vert \nabla u \vert^2 + \epsilon^2 \vert F_\nabla \vert^2 + \epsilon^{-2} W(u) ) \, \dvol_g = \int_M e_\epsilon(u, \nabla)\,\dvol_g, 
\end{equation}
where $\dvol_g$ denotes the volume form on $M$, $e_\epsilon(u, \nabla)$ is the energy density and $F_\nabla$ is the curvature of $\nabla$. As discussed in the introduction, working with $U(1)$-connections allows us to identify $F_\nabla$ with the real, closed, two-form $\omega=\omega_{\nabla}$ via  
\begin{equation} \label{defOmega}
    F_\nabla(X, Y)u = [\nabla_X, \nabla_Y]u - \nabla_{[X, Y]}u = - i \omega_{\nabla}(X, Y)u. 
\end{equation}
The Euler--Lagrange equations for critical points of \eqref{e.bd} are given by 
\begin{align} \label{EL}
	\left\{
    \begin{aligned}
    \nabla^* \nabla u & = {\textstyle\frac{1}{2\epsilon^2}} (1 - \vert u \vert^2) u, \\
    \epsilon^2 d^* \omega_\nabla & = \langle \nabla u, i u \rangle.
    \end{aligned}
    \right.
\end{align}
Here $\nabla^*$ denotes the formal adjoint of $\nabla$ and $d^*$ the formal adjoint of $d$. We refer to \cite[Section~2]{PigatiStern} for further details and to the appendix of the same paper for the regularity of solutions to these equations.  

A key feature of the energies $E_{\epsilon}$ is their \emph{gauge-invariance}: that is, for any $\phi\in \mathcal{G} = \Maps(M,S^1)$, the energy $E_{\epsilon}(u,\nabla)$ is invariant under the change of gauge
$$\phi\cdot (u,\nabla)=(\phi u, \nabla-i\phi^*(d\theta)),$$
corresponding to a fiberwise rotation of $L$. As discussed in the introduction, an important first step in understanding the $\Gamma$-convergence theory for $E_{\epsilon}$ is identifying an appropriate gauge-invariant analog of the Jacobian two-form $2du^1\wedge du^2$ for complex-valued maps. To this end, for a pair $(u,\nabla)$, we consider the two-forms $\psi(u,\nabla)$ given by
\begin{equation*}
\psi(u,\nabla)(X, Y) := 2 \langle i \nabla_X u, \nabla_Y u \rangle, 
\end{equation*}
for vector fields $X$ and $Y$, and define the \emph{gauge-invariant Jacobians}
\begin{equation*}
J(u, \nabla) := \psi(u,\nabla) + (1 - \vert u \vert^2) \omega_\nabla.
\end{equation*}
A straightforward computation shows that
\begin{equation}\label{bdry.char}
d\langle \nabla u, iu\rangle=\psi(u)-|u|^2\omega_\nabla=J(u,\nabla)-\omega_\nabla,
\end{equation}
from which we deduce that $J(u,\nabla)$ is closed and cohomologous to $\omega$. Moreover, as mentioned in the introduction, it is easy to check that $\psi(u,\nabla)$ satisfies the pointwise estimate $|\psi(u,\nabla)|\leq |\nabla u|^2$, which together with Young's inequality implies
\begin{equation} \label{JboundedbyE}
\vert J(u,\nabla) \vert \leq |\nabla u|^2+\epsilon^2|\omega_\nabla|^2+\frac{1}{4\epsilon^2}(1-|u|^2)^2= e_{\epsilon}(u,\nabla),
\end{equation}
so that $J(u,\nabla)$ has $L^1$ norm bounded above by $E_{\epsilon}(u,\nabla)$. Throughout the paper, we identify $J(u,\nabla)$ with an $(n-2)$-current, with the assignment
\begin{equation*}
	\langle J(u,\nabla),\eta \rangle := \int_M J(u,\nabla)\wedge\eta
\end{equation*}
for all $\eta\in\Omega^{n-2}(M)$; under this identification, note that the mass of $J(u,\nabla)$ is precisely
$$\mathbb{M}(J(u,\nabla))=\|J(u,\nabla)\|_{L^1(M)}\leq E_{\epsilon}(u,\nabla).$$

Finally, given a smooth reference connection $\nabla_0$ on $L$ with associated curvature two-form $\omega_0$, it will be useful to note that, by \eqref{bdry.char}, we can write
\begin{equation}\label{Jbeta}
J(u,\nabla)=d(\beta(u,\nabla))+\omega_0
\end{equation}
where
\begin{equation}\label{beta.def}
\beta(u,\nabla=\nabla_0-i\alpha):=\langle \nabla u,iu\rangle+\alpha=\langle \nabla_0u,iu\rangle+(1-|u|^2)\alpha,
\end{equation}
implicitly using the fact that $\nabla$ can be written as $\nabla_0 - i \alpha$, for $\alpha \in \Omega^1(M)$, so that $\omega_{\nabla}=\omega_0+d\alpha$.

\subsection{Notions from geometric measure theory}
For the convenience of the reader, we collect here some terminology and notation from geometric measure theory used throughout the paper. We follow \cite{Simon} and we refer the reader to it for further details. 

We denote by $\mathcal{I}_k(M; \mathbb{Z})$ the space of \emph{integer rectifiable $k$-currents with finite mass}. Recall that an \emph{integral $k$-current} is an integer rectifiable $k$-current whose \emph{boundary has finite mass} (and, as a consequence, is itself an integer rectifiable $(k-1)$-current). We denote by ${\bf I}_{k}(M; \mathbb{Z})$ the space of $k$-dimensional integral currents in $M$ and by $\mathcal{Z}_k(M; \mathbb{Z})$ the subset of those currents $T \in {\bf I}_{k}(M; \mathbb{Z})$ satisfying $\partial T = 0$.

Given $T \in {\bf I}_{k}(M; \mathbb{Z})$ we denote by $\vert T \vert$ the associated integral varifold and by $\Vert T \Vert$ the induced Radon measure on $M$. The definition of mass used in this paper is
\begin{align*}
	&\mathbb{M}(T) := \sup \{T(\phi)\mid\phi\in\Omega^k(M),\,\|\phi\|_{C^0(M)}\le 1\},
\end{align*}
where the last norm is understood with respect to the Euclidean norm on $k$-covectors. Setting $\mathbb{M}(S, T) := \mathbb{M}(S - T)$ for $S, T \in \mathcal{I}_{k}(M; \mathbb{Z})$ we obtain a metric on $\mathcal{I}_{k}(M; \mathbb{Z})$ known as the \emph{mass metric}. We can topologize the space $\mathcal{I}_{k}(M; \mathbb{Z})$ differently via the so-called \emph{flat distance}
\begin{equation*}
\mathcal{F}(S,T) := \inf \{\mathbb{M}(P) + \mathbb{M}(Q) \mid S - T = P + \partial Q,\, P \in \mathcal{I}_{k}(M; \mathbb{Z}),\, Q \in \mathcal{I}_{k + 1}(M; \mathbb{Z})  \}, 
\end{equation*}
for $S,T \in \mathcal{I}_{k}(M; \mathbb{Z})$. Writing $\mathcal{F}(T) = \mathcal{F}(T, 0)$, note that we trivially have 
\begin{equation*}
\mathcal{F}(T) \leq \mathbb{M}(T) \quad \text{for all } T \in \mathcal{I}_{k}(M; \mathbb{Z}). 
\end{equation*}
Some further concepts from geometric measure theory relevant to the min-max comparison are introduced in Section \ref{min-max.sec} below.


\section{The liminf inequality} \label{SecLimInf}

\subsection{The distributional gauge-invariant Jacobian and singular unit sections}\hspace{30mm}
In the classical $\Gamma$-convergence theory for the Allen--Cahn energies, it is important to identify the space of $(n-1)$-boundaries in $M$ with the distributional derivatives of functions in $BV(M,\{-1,1\})$, which arise as limits of the functions $\Phi(v_{\epsilon})$ for real-valued functions $v_{\epsilon}:M\to \mathbb{R}$ with $F_{\epsilon}(v_{\epsilon})=O(1)$, where $\Phi(s):=\int_0^s\sqrt{2W(t)}\,dt\,/\int_0^1\sqrt{2W(t)}\,dt$. Similarly, the study of $\Gamma$-convergence for functionals of Ginzburg--Landau type is closely related to the theory of distributional Jacobians for circle-valued (and, more generally, sphere-valued \cite{ABO2, ABO1, JerrardSoner2}) maps, but the structure theory of these Jacobians does not play a direct role in the $\Gamma$-convergence proofs, since these results are not typically accompanied by compactness results for the given sequence of complex-valued maps.

For our results, it will likewise be useful to identify the space $\mathcal{Z}_{n-2}(M;\mathbb{Z})$ of integral $(n-2)$-cycles in $M$ with the topological singularities (distributional Jacobians) of certain singular unit sections of hermitian line bundles on $M$, arising as a limit of the two-forms $J(u,\nabla)$ for smooth pairs $(u,\nabla)$. To this end, we seek to extend the definition of the $(n-2)$-current $J(u,\nabla)$ to a larger class of pairs $(u,\nabla)$ of lower regularity, generalizing the distributional Jacobian for complex-valued maps.

First, we need to understand the continuity of $J(u,\nabla)$ as a map into the space of $(n-2)$-currents $\mathcal{D}_{n-2}(M)$ with the $(C^1)^*$ metric. Given $p\in (1,\infty)$ and a fixed reference connection $\nabla_0$ on $L\to M$, we introduce the norm
\begin{equation*}
\|(u,\nabla)\|_p:=\|u\|_{L^p(M)}+\|\nabla_0u\|_{L^p(M)}+\|\nabla-\nabla_0\|_{L^p(M)}
\end{equation*}
on the space of smooth pairs $u\in \Gamma(L)$ and $\nabla=\nabla_0-i\alpha$, and denote by $X_p(L)$ the metric space obtained as the completion of the space of smooth pairs
$$(u,\nabla)=(u,\nabla_0-i\alpha),\text{ where }|u|\leq 1$$
with respect to the norm $\|\cdot\|_p$. Note that, in a local trivialization, elements of $X_p(L)$ can be identified with pairs $(u,\alpha)$ where $\alpha$ is a one-form in $L^p$ and $u$ is a $W^{1,p}$ map to the unit disk $\bar D\subset \mathbb{C}$. The precise definition of the norm $\Vert \cdot \Vert_p$ is somewhat arbitrary, and other equivalent norms would work just as well. With respect to this norm, it is not difficult to check that the assignment $(u,\nabla)\mapsto J(u,\nabla)$ satisfies the desired continuity properties, summarized in the following proposition.

\begin{proposition} 
For a fixed reference connection $\nabla_0$ on $L\to M$ and $p\in (1,2)$, given pairs $(u,\nabla)$ and $(v,\nabla')$ satisfying $|u|\leq 1$, $|v|\leq 1$, and $\|(u,\nabla)-(v,\nabla')\|_p\leq 1$, we see that the one-forms $\beta(u,\nabla)$ and $\beta(v,\nabla')$ given by \eqref{beta.def} satisfy
\begin{equation}\label{beta.diff}
\|\beta(u,\nabla)-\beta(v,\nabla')\|_{L^1(M)}\leq C(p)(1+\|(u,\nabla)\|_p)\|(u,\nabla)-(v,\nabla')\|_p^{p-1}.
\end{equation}
Consequently, the assignment $(u,\nabla)\mapsto J(u,\nabla)$ extends continuously to a map
$$(X_p(L),\|\cdot\|_p)\to (\mathcal{D}_{n-2}(M),(C^1)^*)$$
where $(\mathcal{D}_{n-2}(M),(C^1)^*)$ denotes the space of $(n-2)$-currents equipped with the $(C^1(M))^*$ norm.
\end{proposition}

\begin{proof} Writing $\nabla=\nabla_0-i\alpha$ and $\nabla'=\nabla_0-i\eta$ for $\alpha,\eta\in \Omega^1(M)$, it follows from \eqref{beta.def} that
\begin{align*}
\beta(u,\nabla)-\beta(v,\nabla')&=\langle \nabla_0u,iu\rangle-\langle \nabla_0v,iv\rangle+(1-|u|^2)\alpha-(1-|v|^2)\eta\\
&=\langle \nabla_0(u-v),iu\rangle+\langle \nabla_0v, i(u-v)\rangle\\
&\quad+(1-|u|^2)(\alpha-\eta)+(|v|^2-|u|^2)\eta.
\end{align*}
In particular, since $|u|\leq 1$ and $|v|\leq 1$, letting $p'$ denote the H\"{o}lder conjugate of $p$, we deduce that
\begin{align*}
\int_M|\beta(u,\nabla)-\beta(v,\nabla')|&\le \int_M (|\nabla_0(u-v)|+|\nabla_0v||u-v|+|\alpha-\eta|+2|\eta||u-v|)\\
&\le \|\nabla_0(u-v)\|_{L^1(M)}+\|\nabla_0v\|_{L^p(M)}\|u-v\|_{L^{p'}(M)}\\
&\quad+\|\nabla-\nabla'\|_{L^1(M)}+2\|\nabla'-\nabla_0\|_{L^p(M)}\|u-v\|_{L^{p'}(M)}\\
&\le C[\|(u,\nabla)-(v,\nabla')\|_p+(\|(u,\nabla)\|_p+\|(v,\nabla')\|_p)\|u-v\|_{L^{p'}(M)}]\\
&\le C[\|(u,\nabla)-(v,\nabla')\|_p+(\|(u,\nabla)\|_p+\|(v,\nabla')\|_p)\|u-v\|_{L^p(M)}^{p-1}]
\end{align*}
for a constant $C=C(p,M)$, where we used the fact that $\Vert u - v \Vert_{L^\infty(M)} \leq 2$ in the last inequality. 
Assuming that $\|(u,\nabla)-(v,\nabla')\|_p\leq 1$, the estimate \eqref{beta.diff} easily follows.

Now, by the characterization \eqref{Jbeta} of $J(u,\nabla)$, for any $\zeta\in \Omega^{n-2}(M)$, we have
\begin{align*}
|\langle J(u,\nabla)-J(v,\nabla'),\zeta\rangle|&=\Big|\int_Md(\beta(u,\nabla)-\beta(v,\nabla'))\wedge \zeta\Big|\\
&=\Big|\int_M (\beta(u,\nabla)-\beta(v,\nabla'))\wedge d\zeta\Big|\\
&\le \|\beta(u,\nabla)-\beta(v,\nabla)\|_{L^1(M)}\|d\zeta\|_{C^1(M)}.
\end{align*}
The second equality follows from Stokes' theorem. 
Together with the estimate \eqref{beta.diff}, this implies that
$$\|J(u,\nabla)-J(v,\nabla')\|_{(C^1(M))^*}\leq C(p,M) (1+\|(u,\nabla)\|_p)\|(u,\nabla)-(v,\nabla')\|_p$$
when $\|(u,\nabla)-(v,\nabla')\|_p\leq 1$. In particular, the assignment $(u,\nabla)\mapsto J(u,\nabla)$ is continuous with respect to the norms $\|\cdot \|_p$ and $(C^1(M))^*$, and therefore admits the desired extension
\begin{align*}
	&(X_p(L),\|\cdot\|_p)\to (\mathcal{D}_{n-2}(M), (C^1)^*). \qedhere
\end{align*}
\end{proof}

Consider now the subset of $X_p(L)$ given by
$$\mathcal{V}_p(L):=\{(u,\nabla)\in X_{p}(L) : |u|\equiv 1\text{ almost everywhere}\},$$
i.e., the set of pairs $(u,\nabla)\in X_p(L)$ where $u$ belongs to the space
$$\mathcal{U}_p(L):=\{u\in W^{1,p}(M,L) : |u|\equiv 1\text{ almost everywhere}\}$$
of $W^{1,p}$ unit sections. Note that for any $(u,\nabla)\in \mathcal{V}_p(L)$ we have
$$\beta(u,\nabla)=\beta(u,\nabla_0),$$
so we can view both $\beta$ and $J=d\beta+\omega_0$ as functions on $\mathcal{U}_p(L)$, \emph{independent of the connection} $\nabla$.
Notice that the definition of $\beta(u)$ still depends on the initial choice of reference connection $\nabla_0$, but of course the assignment $\mathcal{U}_p\ni u\mapsto J(u)$ remains gauge-invariant and independent of $\nabla_0$. In particular, in any local trivialization---in which $u$ becomes identified with a $W^{1,p}$ map to $S^1$ and $\nabla_0=d-i\alpha_0$---we have $\beta(u)=\langle du,iu\rangle-\alpha_0$, and $J(u)=d\langle du,iu\rangle$ coincides with \emph{the standard distributional Jacobian}.

The remainder of the subsection is devoted to recording some key properties of the operator $J: \mathcal{U}_p(L)\to \mathcal{D}_{n-2}(M)$. At the local level, note that this reduces to the study of topological singularities for maps in $W^{1,p}(M,S^1)$, and the arguments that follow are largely drawn from \cite{ABO1} and \cite{JerrardSoner2}. 

\begin{proposition}\label{j.props} For any $u,v\in \mathcal{U}_p(L)$, there exists an integer rectifiable $(n-1)$-current $S\in \mathcal{I}_{n-1}(M;\mathbb{Z})$ of mass
$$\mathbb{M}(S)\leq \frac{1}{2\pi}\int_M |\nabla_0(u+v)||u-v|$$
such that
\begin{equation*}
J(u)-J(v)=2\pi \partial S, 
\end{equation*}
as currents. Moreover, $J(u)=J(v)$ if and only if $u=\phi e^{i\psi}v$ for some $\phi:M\to S^1$ harmonic and $\psi \in W^{1,p}(M,\mathbb{R})$---i.e., if $u$ and $v$ differ by a change of gauge.
\end{proposition}

\begin{proof} To prove the first statement, we introduce the map
$$\Phi: \mathcal{U}_p(L)\times \mathcal{U}_p(L)\to W^{1,p}(M,S^1)$$
given by setting
$$\Phi(u,v):=e^{-i\langle u,iv\rangle}u\bar{v}$$
in any local trivialization; indeed, note that the complex-valued map $u\bar{v}$ is invariant under change of gauge. By direct computation, one can check that the map $w:=\Phi(u,v)$ satisfies the identity
\begin{equation*}
\langle dw,iw\rangle =\beta(u)-\beta(v)-d\langle u,iv\rangle=\langle \nabla_0(u+v),i(u-v)\rangle. 
\end{equation*}
Hence
$$\int_M |dw| \leq \int_M |\nabla_0(u+v)||u-v|,$$
and the distributional Jacobian $Jw=d\langle dw,iw\rangle$ satisfies
$$Jw=d[\beta(u)-\beta(v)]=J(u)-J(v).$$
By \cite[Theorem 3.8]{ABO2}, we can appeal to the coarea formula for maps in $W^{1,1}(M,S^1)$ to deduce the existence of an integer rectifiable current $S\in \mathcal{I}_{n-1}(M;\mathbb{Z})$ of mass
$$\mathbb{M}(S)\leq \frac{1}{2\pi}\int_M|dw|\leq \int_M|\nabla_0(u+v)||u-v|$$
such that
$$2\pi \partial S=Jw=J(u)-J(v),$$
proving the first part of the proposition.

For the latter statement, note that $J(u)-J(v)=0$ if and only if the map $w=\Phi(u,v)\in W^{1,p}(M,S^1)$ satisfies $Jw=0$. But it is easy to check (cf.\ \cite{Demengel}) that a map $w\in W^{1,p}(M,S^1)$ satisfies $Jw=0$ if and only if $w=\phi e^{i\psi}$ for some $\phi:M\to S^1$ harmonic and $\psi\in W^{1,p}(M,\mathbb{R})$.
Indeed, if $Jw=0$ then the one-form $\langle dw,iw \rangle$ is closed, and thus decomposes as $h+d\psi$ with $h$ harmonic, so that $\phi=e^{-i\psi}w$ is a harmonic map. The reverse direction is immediate.
\end{proof}

\begin{corollary}\label{same.class} If $u\in \mathcal{U}_p(L)$ is such that $J(u)$ has finite mass, then $\frac{1}{2\pi}J(u)$ is an integral $(n-2)$-cycle in the homology class dual to $c_1(L)\in H^2(M;\mathbb{Z})$.
\end{corollary}

\begin{proof} 
By Proposition \ref{densityProp} below, note that there exists at least one $u_0\in \mathcal{U}_p(L)$ such that $\frac{1}{2\pi}J(u_0)$ is given by a prescribed integral (in fact, polyhedral) cycle $P\in \mathcal{Z}_{n-2}(M;\mathbb{Z})$ dual to $c_1(L)$. As a consequence, for any $u\in \mathcal{U}_p(L)$, it follows from Proposition \ref{j.props} that 
$$\frac{1}{2\pi}(J(u)-J(u_0))=\partial S$$ 
for an integer rectifiable $S\in \mathcal{I}_{n-1}(M;\mathbb{Z})$ of finite mass.

In particular, if $\mathbb{M}(J(u))<\infty$, then it follows that $\mathbb{M}(S)+\mathbb{M}(\partial S)<\infty$, and we can deduce from \cite[Theorem~30.3]{Simon} that $\partial S$ is itself an integral $(n-2)$-cycle. In particular, 
$$\frac{1}{2\pi}J(u)=\frac{1}{2\pi}J(u_0)+\partial S=P+\partial S$$
is then an integral $(n-2)$-cycle homologous to $P$, proving the claim.
\end{proof}

\subsection{Proof of Theorem \ref{GammaConvThm}, part (i)}
To complete the proof of the $\liminf$ part of the $\Gamma$-convergence theorem, it remains to establish a compactness result for sections $u_{\epsilon}$ coming from couples $(u_\epsilon,\nabla_\epsilon)$ in $X_p$ (modulo gauge transformations) under the assumption of a uniform energy bound $E_{\epsilon}(u_{\epsilon},\nabla_{\epsilon})\leq \Lambda$. As in the previous section, we will continue to work with a fixed smooth reference connection $\nabla_0$ on the line bundle $L\to M$.

\begin{lemma}\label{w1p.est} Let $(u,\nabla)$ satisfy $|u|\leq 1$ and $E_{\epsilon}(u,\nabla)\leq \Lambda$. Then there is a gauge-equivalent pair $(u',\nabla')$ for which
$$\|\nabla'-\nabla_0\|_{L^p(M)}+\|\nabla_0u'\|_{L^p(M)}\leq C(p,M,L,\Lambda)$$
for all $p\in(1,\frac{n}{n-1})$.
\end{lemma}

\begin{proof} 
Writing the initial connection as
$$\nabla=\nabla_0-i\eta$$
for a one-form $\eta\in \Omega^1(M)$, consider the Hodge decomposition
$$\eta=d^*\xi+d\psi+h(\eta),$$
where $\xi\in\Omega^2(M)$, $\psi\in C^{\infty}(M)$, and $h(\eta)$ is harmonic. Since the gradients of $S^1$-valued harmonic maps form a lattice in the space $\mathcal{H}^1(M)$ of harmonic one-forms, note that we can find a harmonic map $f:M\to S^1$ such that
$$\|f^*(d\theta)-h(\eta)\|_{L^{\infty}(M)}\leq C(M).$$
Now, letting
$$\phi:=e^{i\psi}f:M\to S^1,$$
we see that
$$\alpha:=\eta-\phi^*(d\theta)=\eta-f^*(d\theta)-d\psi=d^*\xi+[h(\eta)-f^*(d\theta)].$$
Thus, making the change of gauge
$$(u',\nabla'):=(\phi^{-1}\cdot u, \nabla+i\phi^*(d\theta)),$$
we see that the new connection $\nabla'$ is given by
$$\nabla'=\nabla_0-i\alpha,$$
where $\alpha$ is co-closed, and the harmonic component $h(\alpha)=h(\eta)-f^*(d\theta)$ of the Hodge decomposition $\alpha=d^*\xi+h(\alpha)$ satisfies $\|h(\alpha)\|_{L^{\infty}(M)}\leq C.$

To obtain the desired bound for $\|\nabla'-\nabla_0\|_{L^p(M)}=\|\alpha\|_{L^p(M)}$, it remains to estimate the co-exact component $d^*\xi$. To this end, note that $\xi$ can be assumed exact and is given by
$$\xi=\Delta_H^{-1}(d\eta),$$
by definition of the Hodge decomposition. By the $L^p$ regularity theory for the Hodge Laplacian and a standard duality argument, we have an automatic bound of the form
\begin{equation}\label{coex.lp}
\|d^*\xi\|_{L^p(M)}\leq C(p,M)\|d\eta\|_{W^{-1,p}(M)}=C(p,M)\|d\eta\|_{(W^{1,p'}(M))^*}
\end{equation}
for any $p\in (1,\infty)$.

Now, by definition \eqref{beta.def} of the one-form $\beta(u,\nabla)$, we have
$$\eta=\beta(u,\nabla)-\langle \nabla u,iu\rangle,$$
while \eqref{Jbeta} gives
$$J(u,\nabla)=d(\beta(u,\nabla))+\omega_0.$$
We therefore see that
$$d\eta=J(u,\nabla)-\omega_0-d\langle \nabla u,iu\rangle,$$
and for any $\zeta\in \Omega^2(M)$, it follows that
\begin{align*}
\int_M\langle d\eta, \zeta\rangle &=\int_M\langle J(u,\nabla)-\omega_0,\zeta\rangle-\int_M\langle d\langle \nabla u,iu\rangle,\zeta\rangle\\
&=\int_M\langle J(u,\nabla)-\omega_0,\zeta\rangle-\int_M\langle\langle \nabla u,iu\rangle, d^*\zeta\rangle\\
&\leq \|J(u,\nabla)\|_{L^1(M)}\|\zeta\|_{C^0(M)}+\|F_{\nabla_0}\|_{L^1(M)}\|\zeta\|_{C^0(M)}\\
&\quad+C\|\langle \nabla u,iu\rangle\|_{L^2(M)}\|\zeta\|_{W^{1,2}(M)}.
\end{align*}
We know already that $\|J(u,\nabla)\|_{L^1(M)}\leq E_{\epsilon}(u,\nabla)\leq \Lambda$, and since $\nabla_0$ is a fixed reference connection, we automatically have $\|F_{\nabla_0}\|_{L^1(M)}\leq C(M,L)$ independent of $(u,\nabla)$. Moreover, since $|u|\leq 1$, we also see that
$$\|\langle \nabla u, iu\rangle\|_{L^2(M)}\leq \|\nabla u\|_{L^2(M)}\leq E_{\epsilon}(u,\nabla)^{1/2}.$$
Combining the preceding estimates, it follows that
$$\int_M\langle d\eta,\zeta\rangle \leq C(M,L,\Lambda)(\|\zeta\|_{C^0(M)}+\|\zeta\|_{W^{1,2}(M)}),$$
and by the Sobolev embedding $W^{1,q}(M) \hookrightarrow C^0(M)$ for $q>n$ (as well as the obvious embedding $W^{1,q}(M) \hookrightarrow W^{1,2}(M)$ for $q>n\geq 2$), we deduce in particular that
$$\|d\eta\|_{(W^{1,q}(M))^*}\leq C(q,M,L,\Lambda)$$
for any $q>n$. Together with \eqref{coex.lp}, this implies that
$$\|d^*\xi\|_{L^p(M)}\leq C(p,M,L,\Lambda)$$
for all $1<p<\frac{n}{n-1}$, and consequently
\begin{equation}
\|\nabla'-\nabla_0\|_{L^p(M)}=\|\alpha\|_{L^p(M)}\leq \|d^*\xi\|_{L^p(M)}+\|h(\alpha)\|_{L^p(M)}\leq C(p,M,L,\Lambda)
\end{equation}
for $p\in (1,\frac{n}{n-1}),$ giving the desired estimate for $\nabla'-\nabla_0$. 

In particular, since $\nabla'u'=\nabla_0u'-i\alpha u'$, for $1<p<\frac{n}{n-1}$, it also follows that
\begin{align*}
	&\|\nabla_0 u'\|_{L^p(M)}
	\leq \|\nabla' u'\|_{L^p(M)}+\|\alpha\|_{L^p(M)}
	\le \|\nabla u\|_{L^p(M)}+\|\alpha\|_{L^p(M)}
	\leq C(p,M,L,\Lambda),
\end{align*}
as claimed.
\end{proof}

With the preceding lemma in place, we can now finish the proof of the liminf part of the $\Gamma$-convergence statement.

\begin{proof}[Proof of Theorem \ref{GammaConvThm}(i)]
Given a family $(u_{\epsilon},\nabla_{\epsilon}=\nabla_0-i\alpha_{\epsilon})$ with $|u_{\epsilon}|\leq 1$ and uniformly bounded energy $E_{\epsilon}(u_{\epsilon},\nabla_{\epsilon})\leq \Lambda$, we may assume without loss of generality that the change of gauge given in the preceding lemma has already been applied to $(u_{\epsilon},\nabla_{\epsilon})$, so that
$$\|\alpha_{\epsilon}\|_{L^p(M)}+\|\nabla_0 u_{\epsilon}\|_{L^p(M)}\leq C(p,M,L,\Lambda)$$
for $1<p<\frac{n}{n-1}$. In this case, it follows that the sections $u_{\epsilon}$ are uniformly bounded in $W^{1,p}$ norm
$$\|u_{\epsilon}\|_{W^{1,p}(M)}=\|u_{\epsilon}\|_{L^p(M)}+\|\nabla_0u_{\epsilon}\|_{L^p(M)},$$
so by the Rellich--Kondrachov theorem, we can pass to a subsequence such that $u_{\epsilon}$ converges strongly in $L^p(M,L)$ to a limiting section $u\in W^{1,p}(M,L)$. Moreover, since the sections $u_{\epsilon}$ satisfy the pointwise bound $|u_{\epsilon}|\leq 1$, we see that the convergence $u_{\epsilon}\to u$ must be strong in $L^q(M,L)$ for every $q\in [1,\infty)$, and therefore the limiting section $u$ must satisfy
\begin{equation*}
\int_M (1-|u|^2)^2=\lim_{\epsilon\to 0}\int_M(1-|u_{\epsilon}|^2)^2\leq \lim_{\epsilon\to 0}\epsilon^2E_{\epsilon}(u_{\epsilon},\nabla_{\epsilon})=0;
\end{equation*}
i.e., $|u|\equiv 1$ almost everywhere, so $u\in \mathcal{U}_p(L)$.

By \eqref{beta.def} and a straightforward calculation, one can check that
\begin{align*}
\beta(u_{\epsilon},\nabla_{\epsilon})-\beta(u)
& = \langle \nabla_0u_{\epsilon},iu_{\epsilon}\rangle-\langle \nabla_0u,iu\rangle+(1-|u_{\epsilon}|^2)\alpha_\epsilon \\
& = \langle \nabla_0(u_{\epsilon}+u),i(u_{\epsilon}-u)\rangle+(1-|u_{\epsilon}|^2)\alpha_\epsilon+d\langle u_{\epsilon},iu\rangle,
\end{align*}
so that the difference $J(u_{\epsilon},\nabla_{\epsilon})-J(u)=d[\beta(u_{\epsilon},\nabla_{\epsilon})-\beta(u)]$ satisfies
\begin{align*}
&\|J(u_{\epsilon},\nabla_{\epsilon}) - J(u)\|_{(C^1(M))^*} \\
& \leq C\|\langle\nabla_0(u_{\epsilon}+u),i(u_{\epsilon}-u)\rangle+(1-|u_{\epsilon}|^2)\alpha_\epsilon \|_{L^1(M)}\\
& \leq  C(\|\nabla_0u_{\epsilon}\|_{L^p(M)}+\|\nabla_0 u\|_{L^p(M)})\|u_{\epsilon}-u\|_{L^{p'}(M)} + C\|\alpha_\epsilon\|_{L^p(M)}\|1-|u_{\epsilon}|^2\|_{L^{p'}(M)}\\
&\leq C(p,M,L,\Lambda)(\|u_{\epsilon}-u\|_{L^{p'}(M)}+\|1-|u_{\epsilon}|^2\|_{L^{p'}(M)}).
\end{align*}
Since $u_{\epsilon}\to u$ strongly in $L^{p'}$ for $p>1$, taking the limit as $\epsilon \to 0$, we have that the right-hand side goes to 0, establishing the desired convergence $J(u_{\epsilon},\nabla_{\epsilon})\to J(u)$ in $(C^1)^*$. Finally, lower semicontinuity of the mass gives the obvious bound
\begin{equation*}
\mathbb{M}(J(u)) \leq \liminf_{\epsilon \rightarrow 0} \mathbb{M}(J(u_{\epsilon},\nabla_{\epsilon}))\leq \liminf_{\epsilon \rightarrow 0} E_{\epsilon}(u_{\epsilon},\nabla_{\epsilon}) \leq \Lambda,
\end{equation*} 
and by Corollary \ref{same.class}, it follows that $\frac{1}{2\pi}J(u)$ defines an integral $(n-2)$-cycle in the correct homology class. 
\end{proof}

\begin{remark}
Alternatively, one can also give another proof of the liminf inequality via techniques similar to those used in Alberti, Baldo and Orlandi \cite{ABO2,ABO1} for functionals of Ginzburg--Landau type. Though this method would be slightly more involved than the proof given here, the automatic mass bound $\|J(u_{\epsilon},\nabla_{\epsilon})\|_{L^1(M)}\leq E_{\epsilon}(u_{\epsilon},\nabla_{\epsilon})$ again simplifies several steps, reducing the problem to establishing the integrality of the limiting cycle. 
\end{remark}

\section{Recovery sequence}   
In this section we prove existence of a recovery sequence, thus establishing the other half of the $\Gamma$-convergence and finishing the proof of Theorem \ref{GammaConvThm}. The proof is constructive in nature and exploits in a crucial way the two-dimensional solutions of the vortex equations appearing in Theorem \ref{taubes}.  We start by recalling a few basic facts from algebraic topology. 

\begin{proposition}
Any cohomology class $\alpha\in H^2(M;\Z)$ is the Euler class $c_1(L)$ of some complex line bundle $L \to M$. Also, the Euler class classifies the line bundle up to isomorphism.
\end{proposition}

Indeed, it is well known that any complex line bundle arises as the pullback of the canonical line bundle on $\mathbb{CP}^\infty$ by means of a continuous map $f:M\to\mathbb{CP}^\infty$, with a correspondence between the homotopy class $[f]$ and the isomorphism class of the line bundle. For a specific choice of the generator $\lambda$ of $H^2(\mathbb{CP}^\infty; \mathbb{Z})$, we then have
$c_1(L)=f^*\lambda$. On the other hand, $\mathbb{CP}^\infty$ is an Eilenberg--MacLane space $K(\Z,2)$: hence, any $\alpha\in H^2(M; \Z)$ equals $f^*\lambda$ for a unique homotopy class $[f]$; see, e.g., \cite[Theorem~4.57]{Hatcher}. For a more elementary proof using the exponential sheaf sequence, see for instance \cite[pp.\ 139--140]{GriffithsHarris}.

We know from Section \ref{SecLimInf} that the homology class of a limit cycle $\Gamma$ is dual to the Euler class of the bundle. Conversely, given a hermitian line bundle $L\to M$ and a cycle $\Gamma$ whose homology class $[\Gamma]$ is dual to $c_1(L)$,
we now show how to realize $\Gamma$ as the limit of $\frac{1}{2\pi}J(u_\epsilon,\nabla_\epsilon)$,
for appropriate pairs of sections and connections on $L$, as in part (ii) of Theorem \ref{GammaConvThm}.

The next proposition provides a useful variant of Federer's polyhedral approximation theorem (cf.\ \cite[Lemma~4.2.19]{Federer}) for our setting, providing a polyhedral approximation of a given cycle $\Gamma$, which can be realized as the distributional Jacobian $J(v)$ of an appropriate singular unit section. 
Locally, this is a simpler version of the main result from \cite{ABO2}, with appropriate modifications for the manifold setting.

\begin{proposition} \label{densityProp}
Given an integral $(n-2)$-cycle $\Gamma\in \mathcal{Z}_{n-2}(M;\mathbb{Z})$, there exists a triangulation of $M$ and an integer valued function $k$ on the collection $\{\Delta\}$ of $(n-2)$-simplices of the triangulation, each with a fixed orientation, such that the integral current
\begin{align*}
		P:={\textstyle\sum_\Delta} k(\Delta)\Delta
\end{align*}
is a cycle arbitrarily close to $\Gamma$ in the flat topology,
	with $\mathbb{M}(P)$ arbitrarily close to $\mathbb{M}(\Gamma)$.
	Also, there exists a section $v\in\mathcal{U}_p(L)\cap C^\infty(M\setminus\mathcal{S}_{n-2})$, for $p\in(1,2)$, such that
	\begin{align*}
		&J(v)=2\pi P
	\end{align*}
	and, with respect to a reference connection $\nabla_0$,
	\begin{align}\label{est.v}
		&|\nabla_0 v|\le C\operatorname{dist}(\cdot,\mathcal{S}_{n-2})^{-1}, 
	\end{align}
	where $\mathcal{S}_{n-2}=\bigcup \Delta$ is the $(n-2)$-skeleton of the triangulation, and $C$ depends on $v$. 
\end{proposition}
\begin{proof}
	In order to approximate $\Gamma$, we modify Federer's classic approximation result \cite[Lemma~4.2.19]{Federer} as follows. 
	Given $\delta>0$, using the same proof we can find a finite collection of disjoint $C^1$ embeddings
	$F_j:\bar B^{n-2}\to M$ and multiplicities $a_j\in\Z$ such that 
	\begin{align*}
		&\mathbb{M}(T)<\delta,\quad\text{where } T:=\Gamma-{\textstyle\sum_j} a_jF_j(B^{n-2}).
	\end{align*}
	Moreover, we can find a triangulation of $M$ such that each piece $F_j(\bar B^{n-2})$ is a subcomplex, for instance triangulating first a tubular neighborhood of each and then extending to a triangulation of the complement, using \cite[Theorem~10.6]{Munkres}.
	We can also refine the triangulation in such a way that each simplex has diameter less than a given $\rho>0$ and
	admits a diffeomorphism $f$ to (a scaled copy of) the standard simplex with $\operatorname{Lip}(f)+\operatorname{Lip}(f^{-1})\le C$, for a universal constant $C$.
	
	We now argue as in the deformation theorem (see \cite[Theorem~4.2.9]{Federer} or \cite[Theorem~29.1]{Simon}),
	using our triangulation in place of the Euclidean grid.
	Since we are in a manifold, we cannot easily average over translations; but, recalling that the simplices are identified with the standard one, we can average instead over the center of the retraction.


	Namely, given the standard $k$-dimensional simplex $\Delta^k$, denote $\frac{1}{2}\Delta^k$ the rescaled simplex with the same center.
	Since $\frac{1}{2}\Delta^k$ has positive distance from the boundary $\de\Delta^k$, for any point $p\in\frac{1}{2}\Delta^k$ the radial retraction $r_p:\Delta^k\setminus\{p\}\to\de\Delta^k$ is locally Lipschitz outside $\{p\}$ and satisfies $|dr_p(x)|\le C(k)|x-p|^{-1}$.
	Then, for $0\le m<k$, given a normal rectifiable 
	$m$-current $W$ on $\Delta^k$, with $C=C(k)$ we have
	\begin{align*}
		&\int_{\frac{1}{2}\Delta^k}\int_{\Delta^k} |dr_p(x)|^m\,d|W|(x)\,d\mathcal{L}^k(p)
		\le C\int_{\Delta^k}\int_{\frac{1}{2}\Delta^k}|x-p|^{-m}\,d\mathcal{L}^k(p)\,d|W|(x)
		\le C\mathbb{M}(W).
	\end{align*}
	Hence, there exists $p$ such that the inner integral on the left-hand side is bounded by $C(k)\mathbb{M}(W)$ (and $\|W\|(\{p\})=0$ if $m=0$). A standard cut-off argument shows that the pushforward $(r_p)_*W$ is a well-defined current whose mass is bounded by the same quantity.
	If $W$ has no boundary in the interior of $\Delta^k$, as in the proof of the deformation theorem it is easy to check that the difference $W-(r_p)_*W=\de V$ for some $(m+1)$-current $V$ with $\mathbb{M}(V)\le C(k)\mathbb{M}(W)$.
	Scaling by a factor $\rho$ gives the same result for a current $W$ supported on the scaled simplex, with the bounds $\mathbb{M}((r_p)_*W)\le C(k)\mathbb{M}(W)$ and $\mathbb{M}(V)\le C(k)\rho\mathbb{M}(W)$.
	
	The same argument applies to an $m$-current supported on the $k$-skeleton of our triangulation, assuming that $0\le m<k$ and that the boundary of the current is supported on the $(k-1)$-skeleton, since the retractions on each $k$-simplex paste together.
	In particular, this holds for the $(n-2)$-current $T$, with $k=n$, since
	\begin{align*}
		&\de T=-{\textstyle\sum_j} a_j\de(F_j(B^{n-2}))
	\end{align*}
	is supported on the $(n-3)$-skeleton. We can thus construct a retraction $r$ to the $(n-1)$-skeleton such that $T':=r_*T$
	satisfies $T=T'+\de R'$, with
	\begin{align*}
		&\mathbb{M}(T')\le C\mathbb{M}(T)\quad \text{and} \quad \mathbb{M}(R')\le C\rho\mathbb{M}(T)
	\end{align*}
	where $T'$ is an integral current supported on the $(n-1)$-skeleton.
	We can repeat the same on the $(n-1)$-skeleton and retract $T'$ to a current $T''$ supported on the $(n-2)$-skeleton, such that $T'=T''+\partial R''$, with
\begin{align*}
\mathbb{M}(T'')\le C\mathbb{M}(T')\quad\text{and}\quad\mathbb{M}(R'')\leq C\rho\mathbb{M}(T').
\end{align*}	
Since $\de T''=\de T$ vanishes on the interior of each $(n-2)$-simplex, by the constancy theorem $T''$ is an algebraic sum of the $(n-2)$-simplices. Thus, defining
	\begin{align*}
		&P:=T''+{\textstyle\sum_j} a_jF_j(B^{n-2}),
	\end{align*}
	we have $\Gamma-P=\de(R'+R'')$ and
	\begin{align*}
		&|\mathbb{M}(P)-\mathbb{M}(\Gamma)|\le\mathbb{M}(P-\Gamma)\le\mathbb{M}(T)+\mathbb{M}(T'')\le C\delta,
\end{align*}
together with 
\begin{align*}
\quad \mathbb{M}(R')+\mathbb{M}(R'') \le C\delta,
\end{align*}
for $\rho$ small enough.
	
Let us now fix a smooth section $w_0:M\to L$ which is transverse to the zero section, the existence of which is guaranteed, for instance, by \cite[Theorem~IV.2.1]{Kosinski}. The implicit function theorem implies then that $S_0:=w_0^{-1}\{0\}$ is a smooth $(n-2)$-submanifold. Moreover, it comes equipped with the canonical orientation such that a positive basis $\{v_3,\dots,v_n\}$ of $T_p S_0$, extended with $\{v_1,v_2\}$ such that $\{dw_0[v_1],dw_0[v_2]\}$ gives a positive basis of $L_p$, gives a positive basis $\{v_1,\dots,v_n\}$ of $T_pM$.
With this orientation, letting $v_0:=\frac{w_0}{|w_0|}$, we have $J(v_0)=2\pi S_0$ and $[M]\frown c_1(L)=[S_0]$.
		
We can then find another triangulation of $M$ such that $S_0$ is a union of $(n-2)$-simplices. Using \cite[Theorem~10.4]{Munkres}, viewing the two triangulations as embeddings of simplicial complexes, we can find a subdivision of each complex and a perturbation $(F_t)_{t\in[0,1]}$ of the second embedding in such a way that $F_1$ agrees with the first embedding, for a suitable identification of the two domain complexes.
The perturbation can be chosen to be smooth on each domain simplex.
	
Note that the perturbed $S_1=F_1(S_0)\subseteq \mathcal{S}_{n-2}$ is still the zero set of a piecewise smooth section, which we denote $w_1$, obtained for instance from $w_0$ by parallel transport along the curves $t\mapsto F_t\circ F_0^{-1}(x)$ for $x\in M$ with respect to some fixed connection $\nabla_0$ on $L$. 
Clearly, for this transported section $w_1$, the singular unit section $v_1:=\frac{w_1}{|w_1|}$ satisfies $J(v_1)=2\pi S_1$ outside $\mathcal{S}_{n-3}$, the $(n-3)$-skeleton of the triangulation.

The perturbations $J(v_1)$ and $S_1$ differ from $J(v_0)$ and $S_0$ by boundaries of rectifiable currents, by Proposition \ref{j.props}.
Hence, from $J(v_0)=2\pi S_0$ we deduce that $J(v_1)-2\pi S_1$ is the boundary of a rectifiable $(n-1)$-current vanishing outside $\mathcal{S}_{n-3}$. The retraction of the latter to the $(n-1)$-skeleton must then be a linear combination of $(n-1)$-simplices by the constancy theorem, with the same boundary. Hence, $J(v_1)-2\pi S_1=0$. Since $[P]=[\Gamma]=[S_1]$, the difference $[P-S_1]$ is trivial in $H_{n-2}(M;\Z)$. Hence,
\begin{align*}
		 P-S_1=\de\Big({\textstyle\sum_j} k_j R_j\Big)
\end{align*}
for a collection $\{R_j\}$ of $(n-1)$-simplices in the triangulation. We have the following elementary fact.

\begin{lemma}
There exists a map $v'\in C^{\infty}(M\setminus\bigcup_j \operatorname{spt}(\de R_j),S^1)$ with $J(v')=2\pi\de({\textstyle\sum_j} k_jR_j)$ and $|dv'|\le C \dist(\cdot,\bigcup_j \spt(\de R_j))^{-1}$.
\end{lemma}
The proof is a straightforward application of the techniques in \cite[Section~4]{ABO1}. Indeed, 
for a geodesic ball $\bar U_j\subset M$ covering $R_j$, the arguments of \cite{ABO1} can be applied to obtain a map $v_j':\bar U_j\to S^1$, locally Lipschitz outside $\spt(\de R_j)$, satisfying $J(v_j')=2\pi\de R_j$ and $|dv_j'|\le C\dist(\cdot,\spt(\de R_j))^{-1}$; up to regularization, we can assume $v_j'$ smooth outside $\spt(\de R_j)$.
The map $v_j'$ restricts to a contractible map from $\de U_j$ to $S^1$, even when $n=2$ (in which case the latter has degree zero). Hence, it can be extended smoothly in the complement of $\bar U_j$.
The product $v':=\prod_j(v_j')^{k_j}$ is the desired map.
	
We can now conclude the proof of the proposition. Since $v_1=\frac{w_1}{|w_1|}$ has Jacobian $J(v_1)=2 \pi S_1$, the product $v:=v'v_1$ then has
	\begin{align*}
		&J(v)
		=J(v')+J(v_1)
		=2\pi(P-S_1)+2\pi S_1
		=2\pi P.
	\end{align*}
	Thus, the cycle $P$ and the map $v$ have all the desired properties and the proof of Proposition \ref{densityProp} is complete. 
\end{proof}

We now show how to obtain a recovery sequence $(u_{\epsilon},\nabla_{\epsilon})$ for any polyhedral approximation $P$ of $\Gamma$.
Once this is done, the result follows for any integral $(n-2)$-cycle $\Gamma$ by the preceding proposition and a diagonal argument. 

Fix a triangulation of $M$ as in the conclusion of Proposition \ref{densityProp}. For an $(n-2)$-simplex $\Delta$, fix a diffeomorphism $\bar\Delta\to\Delta$ from the standard simplex $\bar\Delta$. For $\delta>0$ small, we denote by $\Delta_\delta$ the image of the set of points in $\bar\Delta$ with distance at least $\delta$ from the boundary. Given $p\in\Delta\setminus\de\Delta$, we denote by $B_r^\perp(p)$ the ball of radius $r$ in the normal bundle to $\Delta$ at $p$;
for a set $S$ of such points, we then set $B_r^\perp(S):=\bigcup_{p\in S}B_r^\perp(p)$.
Note that there exists $c'>0$ independent of $\delta$ such that
the exponential map is a diffeomorphism from $B_{c'\delta}^\perp(\Delta_\delta)$ to its image and such that, setting
\begin{align*}
	&V_\delta(\Delta):=\exp(B_{c'\delta}^\perp(\Delta_\delta)),
\end{align*}
we have $V_\delta(\Delta)\cap V_\delta(\Delta')=\emptyset$.
We can also require that the closest point to $\exp_p(v)$ in the $(n-2)$-skeleton $\bigcup\Delta$ is $p$, whenever $v\in B_{c'\delta}^\perp(p)$ and $p\in\Delta_\delta$. With these preparations in place, we come now to the main result of this section.

%

\begin{proposition}\label{glueing}
For $\epsilon > 0$ small enough, there exists a family of smooth couples $(u_\epsilon,\nabla_\epsilon)$ such that
\begin{align*}
	J(u_\epsilon,\nabla_\epsilon) \rightharpoonup 2\pi P, \quad \text{as }\epsilon \rightarrow 0, 
\end{align*}
as currents, and
\begin{align*}
	\lim_{\epsilon \rightarrow 0} E_\epsilon(u_\epsilon,\nabla_\epsilon) = 2\pi\mathbb{M}(P).
\end{align*}
\end{proposition}

Throughout the proof, we will use the following key fact, for a proof of which we refer the reader to \cite[Theorem~III.2.3]{JaffeTaubes}.

\begin{theorem}\label{taubes}
For the trivial line bundle $L\to\C$, given any integer $k_0\in\Z$ there exists a smooth couple $(u_\epsilon,\nabla_\epsilon)$ which is (locally) critical for the energy $E_\epsilon$, has
$u_\epsilon^{-1}\{0\}=\{0\}$ and
\begin{align*}
		E_\epsilon(u_\epsilon,\nabla_\epsilon)=2\pi|k_0|.
\end{align*}
Moreover, $|u_\epsilon|\le 1$ and, writing $\nabla_\epsilon=d-i\alpha_\epsilon$, we have the decay for gauge invariant quantities
\begin{align}\label{decay}
		|\nabla_\epsilon u_\epsilon|+\frac{1-|u_\epsilon|^2}{\epsilon}+\epsilon|d\alpha_\epsilon| \le \frac{C(k_0)}{\epsilon}e^{-c(k_0)|z|/\epsilon}.
\end{align}
Finally, we can require that $u_\epsilon=|u_\epsilon|e^{ik_0\theta}$ for $|z|\ge\epsilon$, which gives
\begin{align}\label{est.non.gauge}
		|u_\epsilon^*(d\theta)|\le C(k_0)|z|^{-1},\quad|du_\epsilon|+|\alpha_\epsilon|\le C(k_0)(\epsilon^{-1}\wedge|z|^{-1}).
\end{align}
\end{theorem}

Note that the pairs $(u_\epsilon,\nabla_\epsilon)$ can be obtained from $(u_1,\nabla_1)$ by scaling.
The exponential decay is proved in \cite[Theorem~III.8.1]{JaffeTaubes}; see also the proof of \cite[Corollary~5.4]{PigatiStern}.
As for the last part, by a change of gauge we can assume $u_1/|u_1|=e^{ik_0\theta}$ for $|z|\ge 1$.
Observing that $$\langle \nabla_1 u_1,iu_1 \rangle=|u_1|^2(u_1^*(d\theta)-\alpha_1),$$ we deduce \eqref{est.non.gauge} from the smoothness of the pair and the decay for $|\nabla_1 u_1|$; the conclusion for arbitrary $\epsilon$ then follows.

We proceed now to the proof of Proposition \ref{glueing}, from which the final part of the $\Gamma$-convergence result stated in Theorem \ref{GammaConvThm} will follow.
\begin{proof}[Proof of Proposition \ref{glueing}]
Let $P$ be a polyhedral cycle and $v\in \mathcal{U}_p(L)$ a singular unit section with $J(v)=2\pi P$ as in the conclusion of Proposition \ref{densityProp}. Fix an $(n-2)$-simplex $\Delta$, a small parameter $\delta>0$, and set $\lambda:=\frac{c'}{3}\delta$. Let $k_0=k(\Delta)$ be the constant multiplicity with which $\Delta$ appears in the polyhedral cycle $P$. 
In the sequel, we will identify $V_\delta(\Delta)$ with $\Delta_\delta\times B_{3\lambda}^2$,
with respect to a fixed trivialization of the normal bundle to $\Delta$.
Also, the vector bundle $L$ is trivial near $\Delta$; hence, we can identify the section $v$ with a smooth $S^1$-valued map on $V_\delta(\Delta)\setminus \Delta$.

We fix a couple $(u_\epsilon',d-i\alpha_\epsilon')$ as in Theorem \ref{taubes}, with degree $k_0$.
Note that, for any $p\in P$, $v$ has degree $k_0$ on the loop $\theta\mapsto (p,\lambda e^{i\theta})$, since $J(v)=2\pi P$. Hence, identifying $u_\epsilon'$ and $\alpha_\epsilon'$ with their pullback under the projection $V_\delta(\Delta)=\Delta_\delta\times B_{3\lambda}^2\to B_{3\lambda}^2\subset\C$, we can write
\begin{align}\label{def.f}
	\frac{u_\epsilon'}{|u_\epsilon'|} = e^{if} v
\end{align}
with $f:\C\setminus\{0\}\to\R$ smooth and depending on $\epsilon$. We then define the new sections 
\begin{align*}
	\tilde u_\epsilon := [1-\chi(1-|u_\epsilon'|)]e^{i\chi f}v,
\end{align*}
and one-forms
\begin{align*}
	\tilde \alpha_\epsilon := \chi\alpha_\epsilon'+(1-\chi)(u_\epsilon')^*(d\theta)+d((\chi-1)f),
\end{align*}
where $\chi:\C\to\R$ is a smooth cut-off function such that $0\le\chi\le 1$, $|d\chi|\le 2/\lambda$ and
\begin{equation*}
\chi(z) = \begin{cases}
1 & \text{for $\vert z \vert \leq \lambda$,} \\
0 & \text{for $\vert z \vert \geq 2\lambda$}. 
\end{cases}
\end{equation*}
Note that the newly defined couples of sections and connections reduce to
\begin{equation*}
(\tilde u_\epsilon,\tilde \alpha_\epsilon) = \begin{cases}
(u_\epsilon',\alpha_\epsilon') & \text{for $|z|<\lambda$}, \\
(v,v^*(d\theta))  & \text{for $|z|>2\lambda$}. 
\end{cases}
\end{equation*}
In particular, the energy density $e_{\epsilon}(\tilde{u}_{\epsilon},d-i\tilde\alpha_{\epsilon})$ of this couple vanishes for $|z|>2\lambda$.
Also, $1-|\tilde u_\epsilon|=\chi(1-|u_\epsilon'|)$, so that the inequality
\begin{align}\label{w.est}
	&(1-|\tilde u_\epsilon|^2)^2\le (1-|u_\epsilon'|^2)^2
\end{align}
holds. Moreover, on the region $\Omega_{\lambda}:=\{\lambda<|z|<2\lambda\}$, using that $(u_\epsilon')^*(d\theta)$ is closed we compute
\begin{align*}
	&d\tilde \alpha_\epsilon
	= \chi d\alpha_\epsilon'+d\chi\wedge(\alpha_\epsilon'-(u_\epsilon')^*(d\theta)).
\end{align*}
Since $\langle \nabla_\epsilon'u_\epsilon',iu_\epsilon' \rangle
= |u_\epsilon'|^2 ((u_\epsilon')^*(d\theta)-\alpha_\epsilon')$,
in view of \eqref{decay} we can conclude that
\begin{align} \label{expDecayForm}
	&\epsilon|d\tilde \alpha_\epsilon|
	\le \epsilon |d\alpha_{\epsilon}'|+\frac{2\epsilon}{\lambda}|u_{\epsilon}'|^{-1}| \nabla_{\epsilon}'u_{\epsilon}'| \le C\frac{1+\epsilon/\lambda}{\epsilon}e^{-c\lambda/\epsilon}
\end{align}
on $\Omega_{\lambda}$, provided that $\lambda/\epsilon$ is big enough.
Also,
\begin{align*}
	&d\tilde u_\epsilon
	= O(|d\chi|(1-|u_\epsilon'|))+O(|d|u_\epsilon'||)
	+ i\tilde u_\epsilon(d(\chi f)+v^*(d\theta)),
\end{align*}
and recalling that $v^*(d\theta)=(u_\epsilon')^*(d\theta)-df$, we conclude that
\begin{align*}
	&(d-i\tilde \alpha_\epsilon)\tilde u_\epsilon
	= O(|d\chi|(1-|u_\epsilon'|))+O(|d|u_\epsilon'||)
	+i\chi\tilde u_\epsilon((u_\epsilon')^*(d\theta)-\alpha_\epsilon').
\end{align*}
Denoting $\tilde \nabla_\epsilon:=d-i\tilde \alpha_\epsilon$ and using that $|d|u_\epsilon'||\le|\nabla_\epsilon' u_\epsilon'|$, we obtain the decay 
\begin{align} \label{expDecayGrad}
	&|\tilde \nabla_\epsilon \tilde u_\epsilon|\le C\frac{1+\epsilon/\lambda}{\epsilon}e^{-c\lambda/\epsilon}
\end{align}
on $\Omega_{\lambda}$.

Choose now $\delta=\delta(\epsilon):=\epsilon^{3/4}$, so that $\lambda(\epsilon)/\epsilon\to\infty$ as $\epsilon\to 0$.
Since the slices $\exp(B_{c'\delta}^\perp(p))$ are orthogonal to $\Delta$ and have area comparable with $\lambda^2$, we deduce from \eqref{w.est}, \eqref{expDecayForm} and \eqref{expDecayGrad} that the energy of the couple $(\tilde u_\epsilon,\tilde \nabla_\epsilon)$ is bounded as follows
\begin{align*}
E_\epsilon(\tilde u_\epsilon,\tilde \nabla_\epsilon) &=2\pi |k_0|\mathcal{H}^{n-2}(\Delta)(1+o(1))+O(\delta^2\epsilon^{-2}e^{-c\delta/\epsilon}) \\
&=2\pi |k_0|\mathcal{H}^{n-2}(\Delta)+o(1),
\end{align*}
with $o(1)$ an infinitesimal term as $\epsilon\to 0$.

Denote by $K:=\bigcup\spt{\de\Delta}$ the $(n-3)$-skeleton of the triangulation. Let us choose $C'>1$ such that $q\in B_{C'\delta}(K)$ whenever $\dist(q,\mathcal{S}_{n-2})\le c'\delta$ and $q\not\in\bigcup_\Delta V_\delta(\Delta)$.
Note that the pairs glue together to give a pair $(\tilde u_\epsilon,\tilde \nabla_\epsilon)$ on the set $M\setminus \bar B_{C'\delta}(K)$ by declaring that $(\tilde u_\epsilon,\tilde \nabla_\epsilon)$ is given by $(v,\nabla_v)$ on the complement of $\bigcup_\Delta V_\delta(\Delta)$, with $\nabla_v$ the connection making $v$ a parallel section. In order to have a pair defined on all of $M$, we pick a smooth cut-off function $\rho_\delta$ defined by 
\begin{equation}
\rho_\delta = 
\begin{cases}
0 & \text{on $B_{2C'\delta}(K)$,} \\
1 & \text{on $M\setminus B_{4C'\delta}(K)$}, 
\end{cases}
\end{equation}
satisfying the additional bound $|d\rho_\delta|\le \delta^{-1}$. With $\nabla_0$ a fixed reference connection, we claim that the couple
\begin{align*}
	&(u_\epsilon,\nabla_\epsilon)
	:=(\rho_\delta \tilde u_\epsilon,(1-\rho_\delta)\nabla_0+\rho_\delta\tilde \nabla_\epsilon)
\end{align*}
has the desired properties. As a first trivial observation, note that near $K$ the pair $(u_\epsilon, \nabla_\epsilon)$ is given by $(0, \nabla_0)$, i.e., the trivial section with the reference connection. 

Next, since $\operatorname{vol}(B_r(K))=O(r^3)$, we have the estimate 
\begin{align}\label{dummy.bound}
	\lim_{\epsilon \rightarrow 0} \int_{B_{4C'\delta}(K)}(|d\rho_\delta|^2+\epsilon^{-2})\le \lim_{\epsilon\rightarrow 0} (\delta(\epsilon)^{-2}+\epsilon^{-2})\cdot C \delta(\epsilon)^3 = 0,
\end{align}
since $\delta(\epsilon)=\epsilon^{3/4}$. Fixing again a simplex $\Delta$, we write $\nabla_0=d-i\alpha_\Delta$
with respect to the chosen trivialization near $\Delta$. Thus, 
\begin{align*}
	&\nabla_\epsilon=d-i(1-\rho_\delta)\alpha_\Delta-i\rho_\delta\tilde \alpha_\epsilon.
\end{align*}
Note that $\nabla_{\epsilon}u_{\epsilon}=\tilde u_\epsilon d\rho_\delta+\rho_\delta\nabla_{\epsilon}\tilde u_\epsilon$ and that
the trivialization can be chosen to guarantee $|\alpha_\Delta|+|d\alpha_\Delta|\le C(M,L)$. In view of \eqref{dummy.bound}, in order to show that the energy of the couple $(u_\epsilon,\nabla_\epsilon)$ on $B_{4C'\delta}(K)$ is infinitesimal, we just have to show that the two quantities
\begin{align*}
	&\int_{B_{4C'\delta}(K)\setminus\bigcup_\Delta V_\delta(\Delta)}(|\rho_\delta\nabla_0 v|^2
	+\epsilon^2|F_{(1-\rho_\delta)\nabla_0+\rho_\delta\nabla_v}|^2)
\end{align*}
and
\begin{align*}
	&\int_{B_{4C'\delta}(K)\cap V_\delta(\Delta)}(|\rho_\delta\nabla_0\tilde u_\epsilon|^2
	+\epsilon^2|F_{(1-\rho_\delta)\nabla_0+\rho_\delta\tilde\nabla_\epsilon}|^2)
\end{align*}
converge to zero (since the contribution of $\tilde\nabla_{\epsilon}\tilde u_\epsilon$ is infinitesimal on $B_{4C'\delta}(K)$). The first assertion follows from \eqref{est.v} and the fact that the integrand is supported on $\{\operatorname{dist}(\cdot,\mathcal{S}_{n-2})\ge c'\delta\}$, which implies that
\begin{align*}
	&\int_{B_{4C'\delta}(K)\setminus\bigcup_\Delta V_\delta(\Delta)}|\rho_\delta\nabla_0 v|^2=O(\delta^3\cdot\delta^{-2})=O(\delta)
\end{align*}
and similarly
\begin{align*}
	&\int_{B_{4C'\delta}(K)\setminus\bigcup_\Delta V_\delta(\Delta)}\epsilon^2|F_{(1-\rho_\delta)\nabla_0+\rho_\delta\nabla_v}|^2
	=O(\delta^3\cdot\epsilon^2\delta^{-4})=O(\delta).
\end{align*}
As for the second one, by \eqref{dummy.bound} it is enough to prove that, for $p\in\Delta_\delta$,
\begin{align*}
	&\int_{\{p\}\times B_{3\lambda}^2}\Big(|d\tilde u_\epsilon|^2+\frac{\epsilon^2}{\delta^2}|\tilde\alpha_\epsilon|^2\Big)\le C\log(\epsilon^{-1}).
\end{align*}
However, by \eqref{est.v} and \eqref{est.non.gauge}, $dv$, $d\chi$, $\alpha_\epsilon'$ and $(u_\epsilon')^*(d\theta)$ at the point $(p,z)$ are all bounded by $C\delta^{-1}$ on the region $\{|z|>\lambda=\frac{c'}{3}\delta\}$, which implies $|df|\le C\delta^{-1}$ and $|f|\le C$ by \eqref{def.f}. Since this region has area $O(\delta^2)$, its contribution is bounded.
On the other hand, $(\tilde u_\epsilon,\tilde\alpha_\epsilon)=(u_\epsilon',\alpha_\epsilon')$ on $\{|z|\le\lambda\}$;
using again \eqref{est.non.gauge}, the claim follows.

Finally, note that $J(u_\epsilon,\nabla_\epsilon) \rightharpoonup 2\pi P$ as currents. 
Indeed, with the same computations as above, we obtain that $\nabla_0 u_\epsilon$ is bounded in $L^p$ independently of $\epsilon$, for any $p<2$. But $u_\epsilon\to v$ almost everywhere,
hence weakly in $W^{1,p}(M,L)$, which gives
\begin{align*}
	J(u_\epsilon,\nabla_\epsilon) \rightharpoonup J(v) = 2\pi P, 
\end{align*}
again as currents as $\epsilon$ goes to 0. 
\end{proof}

\section{Comparison of the min-max constructions}\label{min-max.sec}

With the $\Gamma$-convergence result established, we turn now to the proof of the min-max comparison described in Theorem \ref{WidthCompThm}. The outline of the proof is broadly similar to that of the analogous result of Guaraco \cite[Proposition 8.19]{Guaraco} in the Allen--Cahn setting. First, we employ Theorem \ref{GammaConvThm} to extract from continuous families of pairs $(u,\nabla)$ discretized families of $(n-2)$-boundaries with mass bounded above by $E_{\epsilon}(u_{\epsilon},\nabla_{\epsilon})+o(1)$. To complete the proof of Theorem \ref{WidthCompThm}, we then have to show that the homotopy class of this associated family of cycles is determined by that of the family of pairs $(u,\nabla)$ in the desired way. 

The details of the proof are somewhat more involved than their codimension-one analog, since the assignment from pairs $(u,\nabla)$ to the space of $(n-2)$-boundaries is less explicit, and the homotopy groups of the space of $(n-2)$-boundaries are slightly more complicated. In the next subsection, we recall the relevant definitions from Almgren's min-max methods, and define carefully the min-max values to which Theorem \ref{WidthCompThm} applies.

\subsection{Natural min-max constructions for $\bm{E_{\epsilon}}$} Throughout this section, let $L=\mathbb{C}\times M\to M$ be the \emph{trivial} line bundle over a closed, oriented $n$-manifold $(M^n,g)$ of dimension $n\geq 3$. Fixing a trivialization of $L$, the space of pairs $(u,\nabla)$ consisting of sections $u\in \Gamma(L)$ and hermitian connections $\nabla$ can then be identified with pairs $(u,\alpha)$, where $u:M\to \mathbb{C}$ is a complex-valued map and $\alpha\in \Omega^1(M)$ is a one-form such that $\nabla=d-i\alpha$. 

For a fixed $p>n$, we will view $E_{\epsilon}$ as a functional on the Banach space $\widehat{X}$ consisting of pairs $(u,\nabla)$ where $u\in [W^{1,2}\cap L^p](M)$ and $\nabla=d-i\alpha$ for $\alpha\in W^{1,2}(M)$ (with topology induced by the norm $\|du\|_{L^2(M)}+\|u\|_{L^p(M)}+\|\alpha\|_{W^{1,2}(M)}$), equipped with the Finsler structure
\begin{equation}\label{fins.def}
\|(v,\beta)\|_{(u,\nabla)}:=\|v\|_{L^p(M)}+\|\nabla v\|_{L^2(M)}+\|\beta\|_{L^2(M)}+\|D\beta\|_{L^2(M)},
\end{equation}
where $D$ is the (Levi-Civita) covariant derivative of the one-form $\beta$. It is straightforward to check (cf.\ \cite[Section~7]{PigatiStern}) that the energies $E_{\epsilon}$ define $C^1$ functionals on $\widehat{X}$, and an adaptation of the proof of \cite[Proposition 7.6]{PigatiStern} shows that they satisfy a variant (modulo gauge transformations) of the Palais--Smale condition with respect to the Finsler structure \eqref{fins.def}, making $\widehat{X}$ an appropriate setting for the min-max construction of critical points (provided the nonlinear potential $W$ is modified as described in \cite[Section~7]{PigatiStern}).

\begin{remark} The Palais--Smale result stated in \cite[Proposition 7.6]{PigatiStern} for $E_{\epsilon}$ in $\widehat{X}$ is not quite correct as written when the base manifold $M$ has $H^1(M;\mathbb{Q})\neq 0$. This is due to the fact that a sequence $(u_j,\nabla_j)$ for $E_{\epsilon}$ which is Palais--Smale with respect to the natural Banach norm on $\widehat{X}$ may fail to yield another Palais--Smale sequence under the change of gauge $(\phi_j u_j,\nabla_j-\phi_j^*(d\theta))$ for a sequence of harmonic map $\phi_j:M\to S^1$. However, it is easy to check that the Palais--Smale property with respect to the Finsler structure \eqref{fins.def} is preserved under harmonic change of gauge, and \cite[Proposition 7.6]{PigatiStern} holds with the Banach norm replaced by this Finsler structure.
\end{remark}

Though the space $\widehat{X}$ itself is topologically trivial, the functionals $E_{\epsilon}$ have a rich min-max theory in the $\epsilon\to 0$ limit, owing to the topology of the moduli space
$$\mathcal{M}:=(\widehat{X}\setminus X_0)/\mathcal{G},$$
where $X_0:=\{(u,\alpha)\in \widehat{X} : u\equiv 0\}$ and $\mathcal{G}:=W^{2,2}(M,S^1)$ is the gauge group. Indeed, writing
$$Y:=\{(u,\alpha)\in \widehat{X} : d^*\alpha=0\},$$
note that there is a natural retraction $\rho_C:\widehat{X}\to Y$ given by passing to the Coulomb gauge
$$\rho_C(u,\alpha):=(e^{-i\varphi_{\alpha}}u,\alpha-d\varphi_{\alpha}),$$
where $\varphi_{\alpha}\in W^{2,2}(M,\mathbb{R})$ is the unique solution of 
$$d^*d\varphi_{\alpha}=d^*\alpha \quad \text{and} \quad \int_M\varphi_{\alpha}=0.$$
It is clear that the quotient map $Y\setminus X_0\to\mathcal{M}$ is surjective. The elements of $\mathcal{G}$ sending a given couple in $Y\setminus X_0$ to a couple in the same space are precisely the harmonic maps $\mathcal{H}=\operatorname{Harm}(M,S^1)$, so we can identify $\mathcal{M}$ (homeomorphically) with the quotient
$$\mathcal{M}=(Y\setminus X_0)/\mathcal{H}.$$
Moreover, note that the harmonic $S^1$-valued maps $\mathcal{H}$ contain $S^1$ as a subgroup (by identification with the constant maps), and the quotient $\mathcal{H}/S^1$ has a natural identification
$$\mathcal{H}/S^1\cong [M: S^1]\cong H^1(M;\mathbb{Z}),$$
since each homotopy class in $[M:S^1]$ is uniquely represented in $\mathcal{H}$ up to rotations. We can then view $\mathcal{M}$ as the quotient
$$\mathcal{M}=[(Y\setminus X_0)/S^1]/H^1(M;\mathbb{Z}),$$
of the quotient space $(Y\setminus X_0)/S^1$ by the free and properly discontinuous action of $H^1(M;\mathbb{Z})$.
Moreover, we have the following facts, allowing to extract the algebraic topology invariants of $\mathcal{M}$.

\begin{proposition}\label{tech.topo}
The projection $Y\setminus X_0\to (Y\setminus X_0)/S^1$ is a fiber bundle and, hence, a weak fibration. The former space has trivial homotopy groups, while the latter is weakly homotopy equivalent to $\mathbb{CP}^\infty$, and is the universal cover of $\mathcal{M}$.
\end{proposition}

\begin{proof}
	Let $Q:=(Y\setminus X_0)/S^1$ and denote $\pi:Y\setminus X_0\to Q$ the projection.
	Given $(u,\alpha)\in Y\setminus X_0$, we can find a measurable set $E\subseteq M$ such that $\int_E u\neq 0$.
	In particular, there exists $\delta>0$ such that $\int_E v\neq 0$ for all couples $(v,\beta)$ with distance less than $\delta$ from the $S^1$-orbit of $(u,\nabla)$---namely, such that $\|(v,\beta)-e^{i\theta}\cdot(u,\alpha)\|_{\hat X}<\delta$ for some $e^{i\theta}\in S^1$. These couples form an open set $\pi^{-1}(U)$, for $U$ open in the quotient $Q$. It is then easy to check that the map
	$$\pi^{-1}(U)\to S^1\times U,\quad (v,\beta)\mapsto\Big(\textstyle{\int_E v\,/\,|\int_E v|},\pi((v,\beta))\Big)$$
	gives a local trivialization over $U$. Hence, $\pi$ is a fiber bundle and thus a weak fibration (see \cite[Proposition~4.48]{Hatcher}).
	
	To check the second statement, note that $Q$ (deformation) retracts onto $\hat S/S^1$, where $\hat S$ is the unit sphere of the Banach space $[W^{1,2}\cap L^p](M,\mathbb{C})$, viewed as a subset of $\hat X$ with trivial connection component.
	Given a dense, linearly independent set $\{u_k\}_{k=1}^\infty$ in this Banach space, we denote by $H^\ell$ the linear span of $\{u_1,\dots,u_\ell\}$ and by $\pi_\ell:[W^{1,2}\cap L^p](M,\mathbb{C})\to H^\ell$ the nearest point projection, which is well-defined and continuous since $H^\ell$ is finite-dimensional and the Banach space is strictly convex.
	
	Letting $\hat S^\ell:=\hat S\cap H^\ell$, note that the union $P:=\bigcup_\ell(\hat S^\ell/S^1)$, endowed with the topology induced by the subspaces $\hat S^\ell/S^1$, is homeomorphic to $\mathbb{CP}^\infty$, and the identity map $i:P\to\hat S/S^1$ is continuous. We claim that, for any compact set $K\subset\hat S/S^1$, the inclusion $K\hookrightarrow\hat S/S^1$ can be deformed to a map $K\to \hat S_\ell/S^1$ for some $\ell$ (within maps into $\hat S/S^1$). This implies that $i$ induces isomorphisms $i_*$ on homotopy groups, because then any map $S^k\to\hat S/S^1$ can be deformed to a map with values in $\hat S^\ell/S^1$ for some $\ell$ (hence $i_*$ is surjective),
	and a homotopy in $\hat S/S^1$ between two maps $S^k\to\hat S^\ell/S^1$ can be deformed to a homotopy in $\hat S^{\ell'}/S^1$ with $\ell'\ge\ell$ (hence $i_*$ is injective).
	
	To prove the claim, note that for any $[u]\in\hat S/S^1$ there exists $\ell$ such that the distance from $u$ to $H^\ell$ is less than $1$, and the same holds on a neighborhood of $[u]$. By compactness of $K$, we can find $\ell$ such that this is true for all the elements of $K$. The map $([u],t)\mapsto \frac{(1-t)u+t\pi_\ell(u)}{\|(1-t)u+t\pi_\ell(u)\|_{\hat X}}$ gives the desired deformation.
	
	The fact that $\hat S$, and hence $Y/X_0$, have trivial homotopy groups is proved in the same way.
	The last conclusion follows from the fact that $\mathbb{CP}^\infty$ is simply connected.
\end{proof}

We therefore conclude that the path-connected space $\mathcal{M}$ has $\pi_1(\mathcal{M})\cong H^1(M;\mathbb{Z})$, $\pi_2(\mathcal{M})\cong \mathbb{Z}$, and $\pi_k(\mathcal{M})=0$ for $k\geq 3$; or equivalently, for $k>0$,
$$\pi_k(\mathcal{M})\cong H_{n-2+k}(M;\mathbb{Z}).$$

The results of this section concern the min-max energies associated to the generator of $\pi_2(\mathcal{M})$, and to each class $\lambda\in H_{n-1}(M;\mathbb{Z})\cong \pi_1(\mathcal{M})$ (with basepoint the trivial pair $(u_0\equiv 1,\nabla_0\equiv d) \mod \mathcal{G}$). In practice, we work with their lifts to maps $\bar D^2\to \widehat{X}$ and $[0,1]\to \widehat{X}$.

As in \cite{PigatiStern}, consider the collection 
$$\mathcal{C}_2\subset C^0(\bar D^2, \widehat{X})$$
of continuous families
$$\bar D^2\ni y\mapsto (u_y,\nabla_y)\in \widehat{X}$$
parametrized by the closed unit disk $\bar D^2\subset \mathbb{C}$, subject to the boundary condition
$$u_y\equiv y \quad  \text{and} \quad \nabla_y\equiv d \quad \text{for } y\in \partial D^2=S^1.$$
By the long exact sequence for homotopy groups in weak fibrations, families in $\mathcal{C}_2$ (avoiding $X_0$) descend to the generators of $\pi_2(\mathcal{M})$. It was shown in \cite[Section 7]{PigatiStern} that the associated min-max energies
\begin{equation}
\mathcal{E}_{\epsilon}(\mathcal{C}_2):=\inf_{F\in \mathcal{C}_2}\max_{y\in \bar D^2}E_{\epsilon}(F_y)
\end{equation}
are uniformly bounded from above and below as $\epsilon\to 0$, arise as the energies $E_{\epsilon}(u_{\epsilon},\nabla_{\epsilon})$ of nontrivial critical points $(u_{\epsilon},\nabla_{\epsilon})$ for $E_{\epsilon}$, and converge subsequentially to the mass of a (nontrivial) stationary integral $(n-2)$-varifold, up to a factor of $2\pi$. Likewise, for each nontrivial $\lambda\in H_{n-1}(M;\mathbb{Z})$, we can consider the collection 
$$\mathcal{C}_{\lambda}\subset C^0([0,1],\widehat{X})$$
of continuous families $[0,1]\ni t\mapsto (u_t,\nabla_t)\in \widehat{X}$ satisfying
$$(u_0,\nabla_0)\equiv (1,d), \quad (u_1,\nabla_1)\equiv (\phi, d-i\phi^*(d\theta)),$$
where $\phi\in C^{\infty}(M,S^1)$ is a map in the homotopy class dual to $\lambda$ (i.e., generic fibers of $\phi$ are homologous to $\lambda$). Families in $\mathcal{C}_{\lambda}$ (avoiding $X_0$) descend to loops in $\mathcal{M}$, whose class in $\pi_1(\mathcal{M})$ is determined by $\lambda$, and we will likewise consider their min-max energies
$$\mathcal{E}_{\epsilon}(\lambda):=\inf_{F\in \mathcal{C}_{\lambda}}\max_{t\in [0,1]}E_{\epsilon}(F_t).$$

\begin{remark}
	Note that a family as above, with energy bounded by a given $\Lambda$ (fixed), must avoid the degenerate set of couples $X_0$ for $\epsilon$ small enough.
	Using Proposition \ref{tech.topo}, one can check that the min-max values defined above coincide with the corresponding ones for the homotopy groups of $\mathcal{M}$.
\end{remark}

\subsection{Natural min-max constructions for the $\bm{(n-2)}$-mass functional}
By Almgren's thesis \cite{PhDAlmgren}, we know that the space $Z\subseteq\mathcal{Z}_{n-2}(M;\mathbb{Z})$ of integral $(n-2)$-boundaries in $M$, equipped with the flat topology, has homotopy groups identical to those of $\mathcal{M}$; namely,
$$\pi_k(Z,0)\cong H_{n-2+k}(M;\mathbb{Z})$$
for $k>0$, while $\pi_0(Z)=0$.
In \cite{Alm65} (see also \cite{Pitts}), Almgren associates to each class in $\pi_k(\mathcal{Z}_m(M;\mathbb{Z}))$ a stationary integral $k$-varifold by means of a discretized min-max construction, which replaces continuous families of cycles in the flat topology with discrete families satisfying an approximate continuity condition with respect to the stronger mass topology. For our comparison results, it is convenient to work with discrete families which are fine \emph{in flat norm} and exhibit \emph{no concentration of mass}; by the interpolation arguments of \cite[Section 13]{WillmoreConj} and \cite[Theorem 2.10]{LMN}, the associated min-max masses coincide with the masses of the stationary varifolds produced by Almgren. 
\begin{remark}
While Theorems 2.10 and 2.11 of \cite{LMN} are stated for cycles with $\mathbb{Z}/2\mathbb{Z}$ coefficients, the coefficient group plays no role in these arguments.
\end{remark} 

Following the notation of \cite[Section 2]{LMN}, for $m=1$ or $2$, denote by $I^m$ the $m$-cube $I^m=[0,1]^m$, and for $j\in \mathbb{N}$, denote by $I(1,j)$ the cube complex on $I^1$ with $1$-cells (or edges)
$$[0,3^{-j}], [3^{-j}, 2\cdot 3^{-j}],\ldots, [1-3^{-j},1]$$
and $0$-cells (or vertices) $[0],[3^{-j}],\ldots,[1-3^{-j}], [1]$. Likewise, denote by $I(2,j)$ the cell complex 
$$I(2,j)=I(1,j)\otimes I(1,j)$$
on $I^2$ given by subdividing $I^2$ into $3^{2j}$ squares of area $3^{-2j}$, and denote by $I(m,j)_k$ the collection of $k$-cells of $I(m,j)$.
Given an assignment $\phi: I(m,j)_0\to \mathcal{Z}_{n-2}(M;\mathbb{Z})$, we will say that it has (flat) fineness ${\bf f}(\phi)<\delta$ if 
$$\mathcal{F}(\phi(x),\phi(y))<\delta\text{ for all adjacent vertices }x,y\in I(m,j)_0.$$

If $\phi: I(m,j)_0\to \mathcal{Z}_{n-2}(M;\mathbb{Z})$ satisfies $\phi(x)=0$ for $x\in \partial I^m$ and ${\bf f}(\phi)<\delta$ for $\delta<\delta_M$ sufficiently small, then Almgren's construction \cite{PhDAlmgren} assigns to $\phi$ a homology class $\Psi(\phi)\in H_{n-2+m}(M;\mathbb{Z})$, as follows. For each (oriented) one-cell $e=[x,y]\in I(m,j)_1$, provided $\delta>0$ is sufficiently small, we can find an integral $(n-1)$-current $S_e\in {\bf I}_{n-1}(M;\mathbb{Z})$ such that
$$\partial S_e=\phi(y)-\phi(x) \quad \text{and} \quad \mathbb{M}(S_e)\leq \epsilon_M$$
for a given small constant $\epsilon_M>0$. If $m=1$, then summing over all one-cells $e\in I(1,j)_1$ gives an $(n-1)$-cycle
$$ S={\textstyle\sum_{e\in I(1,j)_1}S_e} \in \mathcal{Z}_{n - 1}(M; \mathbb{Z}) $$
whose homology class $\Psi(\phi):=[S]\in H_{n-1}(M;\mathbb{Z})$ does not depend on the choice of small-mass fill-ins $S_e$. If $m=2$, then for each $2$-cell $\Box\in I(2,j)_2$ we denote by $S_{\Box}\in \mathcal{Z}_{n-1}(M;\mathbb{Z})$ the $(n-1)$-cycle $S_{\Box}=\sum_{e\in \partial \Box}S_e$ given by summing the fill-ins $S_e$ over all oriented edges $e$ of $\partial\Box$, and consider the (unique) $n$-current $Q_{\Box}\in {\bf I}_n(M;\mathbb{Z})$ such that
$$\partial Q_{\Box}=S_{\Box}\quad \text{and} \quad \mathbb{M}(Q_{\Box})<\frac{\vol(M)}{2}.$$
Summing over all $2$-cells $\Box\in I(2,j)_2$ then gives an $n$-cycle
$$Q={\textstyle\sum_{\Box\in I(2,j)_2}Q_{\Box}}\in \mathcal{Z}_n(M;\mathbb{Z})$$
whose homology class $\Psi(\phi):=[Q]\in H_n(M;\mathbb{Z})$ is independent of the choice of small-mass fill-ins $S_e$. 

Now, for $\eta>0$ and a discrete family
$$\phi: I(m,j)_0\to \mathcal{Z}_{n-2}(M;\mathbb{Z}),$$
define the quantity
$${\bf m}(\phi,\eta):=\sup \{\|\phi(x)\|(B_\eta(p)) \mid x\in I(m,j)_0,\ p\in M\},$$
giving the maximum amount of mass of a cycle in the family inside a ball of radius $\eta$. For $\delta \in (0,\delta_M)$ and $\lambda\in H_{n-2+m}(M;\mathbb{Z})$, and a constant $C_0=C_0(M,\lambda)<\infty$ to be chosen later, denote by $\mathcal{A}_{\delta}(\lambda)$ the collection of families
$$\phi: I(m,j)_0\to \mathcal{Z}_{n-2}(M;\mathbb{Z})$$
such that
\begin{equation}\label{delta.reqs}
{\bf f}(\phi)<\delta,\quad\sup_{r>\delta}\frac{{\bf m}(\phi,r)}{r^{n-2}}\leq C_0,
\end{equation}
and
$$\Psi(\phi)=\lambda\in H_{n-2+m}(M;\mathbb{Z}).$$
Then consider the approximate min-max widths
\begin{equation}\label{delta.width}
{\bf W}_{\delta}(\lambda):=\inf \Big\{ \max_{y\in I(m,j)_0}\mathbb{M}(\phi(y))\mid \phi \in \mathcal{A}_{\delta}(\lambda)\Big\},
\end{equation}
and define the min-max width
\begin{equation}\label{main.width}
{\bf W}(\lambda):=\inf\Big\{\liminf_{k\to\infty}\max_{y\in I(m,j_k)_0}\mathbb{M}(\phi_k(y))\Big\},
\end{equation}
where the infimum is taken over all sequences $\phi_k:I(m,j_k)_0\to \mathcal{Z}_{n-2}(M;\mathbb{Z})$ such that $\delta_M>{\bf f}(\phi_k)\to 0$, $\limsup_{k\to\infty}{\bf m}(\phi_k,r)\to 0$ as $r\to 0$, and $\Psi(\phi_k)=\lambda$.
Clearly,
\begin{equation}\label{W.Wdelta}
	{\bf W}(\lambda)\le\lim_{\delta\to 0}{\bf W}_{\delta}(\lambda)=\sup_{\delta>0}{\bf W}_{\delta}(\lambda).
\end{equation}
Since we are ruling out concentration of mass in the limit, we can appeal to the interpolation arguments of \cite[Section 13]{WillmoreConj} and \cite[Theorem 2.10]{LMN} to deduce that the widths ${\bf W}(\lambda)$ coincide with Almgren's min-max widths, and are therefore realized as the masses of stationary integral $(n-2)$-varifolds in $M$.

We can now state a more precise version of Theorem \ref{WidthCompThm}.

\begin{theorem}\label{min-max.comp} The min-max energies $\mathcal{E}_{\epsilon}(\mathcal{C}_2)$ and $\mathcal{E}_{\epsilon}(\lambda)$ for $\lambda\in H_{n-1}(M;\mathbb{Z})$ satisfy
\begin{equation}
\liminf_{\epsilon\to 0}\mathcal{E}_{\epsilon}(\mathcal{C}_2)\geq 2\pi {\bf W}([M])
\end{equation}
and
\begin{equation}
\liminf_{\epsilon\to 0}\mathcal{E}_{\epsilon}(\lambda)\geq 2\pi {\bf W}(\lambda).
\end{equation}
\end{theorem}

The remainder of the section is devoted to its proof.

\subsection{Taming min-max families to avoid energy concentration}
To ensure that the min-max energies $\mathcal{E}_{\epsilon}$ are bounded below by the masses of cycles satisfying \eqref{delta.reqs}, we first argue that the energies $\mathcal{E}_{\epsilon}$ are almost achieved as the maximum energy in families $(u_y,\nabla_y)$ satisfying a uniform energy density bound
$$\int_{B_r(p)}e_{\epsilon}(u_y,\nabla_y)\leq Cr^{n-2}$$
for $\epsilon(M,\delta)>0$ sufficiently small and $r\geq \delta$. 

\begin{lemma}\label{tamed.fam} Given $\delta>0$ and $\Lambda<\infty$, there exists $C(M,\Lambda)<\infty$ such that the following holds. If $\epsilon<\delta$, for any family $F\in \mathcal{C}_2\subset C^0(\bar D^2,\widehat{X})$ (or $F\in \mathcal{C}_{\lambda}\subset C^0([0,1],\widehat{X})$ for $\lambda\in H_{n-1}(M;\mathbb{Z})$) satisfying
\begin{equation}\label{e.max.bd}
\max_y E_{\epsilon}(F_y)<\Lambda,
\end{equation}
there exists another family $F'=(u',\nabla')\in \mathcal{C}_2$ (resp.\ $\mathcal{C}_{\lambda}$) of smooth couples such that
$$\max_y E_{\epsilon}(F'_y)<\Lambda$$
and
$$\max_{y,\,r\geq \delta,\,p\in M}\frac{\int_{B_r(p)}e_{\epsilon}(u_y',\nabla_y')}{r^{n-2}}\leq C(M,\Lambda).$$
\end{lemma}

\begin{proof} First, given a family $F\in \mathcal{C}_2$ or $F\in \mathcal{C}_{\lambda}$ satisfying \eqref{e.max.bd}, we can apply a uniform mollification to obtain a new family $\widetilde{F}$ also satisfying \eqref{e.max.bd} that defines a continuous map into the space of smooth pairs $(u_y,\nabla_y)$, equipped with the $C^{\infty}$ topology. Thus, we may assume without loss of generality that the original family $F$ defines a continuous map into the space of smooth pairs.

In Section \ref{grad.sec} below, we investigate a natural $L^2$ gradient flow system for the energies $E_{\epsilon}$, given by a flow of pairs $(u_t,\nabla_t=d-i\alpha_t)$ satisfying
\begin{equation}\label{gflow.1}
\partial_tu_t=-\nabla_t^*\nabla_t+\frac{1}{2\epsilon^2}(1-|u_t|^2)u_t
\end{equation}
and
\begin{equation}\label{gflow.2}
\partial_t\alpha_t=-d^*d\alpha_t+\epsilon^{-2}\langle iu_t,\nabla_tu_t\rangle.
\end{equation}
As discussed in Section \ref{grad.sec}, it is not difficult to establish long-time existence for the flow, and continuous dependence on smooth initial data. Moreover, it is obvious that minimizers of $E_{\epsilon}$ are stationary under the flow; as a consequence, given a family $y\mapsto F_y=(u_y,\nabla_y)$ in $\mathcal{C}_2$ (resp.\ $\mathcal{C}_{\lambda}$) mapping continuously into the space of smooth pairs as above, we may define a new family $F'\in \mathcal{C}_2$ (resp.\ $\mathcal{C}_{\lambda}$) by letting $F_y'=(u_y',\nabla_y')$ be the solution of \eqref{gflow.1}--\eqref{gflow.2} at time $t=2$ with initial data $(u_y,\nabla_y)=F_y$. Since the gradient flow decreases energy, it is obvious that
$$\max_yE_{\epsilon}(F_y')\leq \max_y E_{\epsilon}(F_y)<\Lambda.$$
Finally, by Proposition \ref{grad.dens.bd} below (the main result of Section \ref{grad.sec}), we have the density estimate
$$\int_{B_r(p)}e_{\epsilon}(u_y',\nabla_y')\leq C(M,\Lambda)r^{n-2}$$
for all $r\geq \epsilon$, so that the family $F'$ satisfies the desired properties.
\end{proof}

\begin{remark}\label{trunc}
Note, moreover, that we may always deform an initial family $(u_y,\nabla_y)$ to one $(v_y,\nabla_y)$ with $|v_y|\leq 1$ pointwise, without increasing the energy, by setting $v_y:=\frac{u_y}{\max\{1,|u_y|\}}$. In particular, for the purposes of estimating the min-max energies, we may always assume that our families $(u_y,\nabla_y)$ satisfy $|u_y|\leq 1$ pointwise, without loss of generality.
\end{remark}

To prove Theorem \ref{min-max.comp}, we will use this lemma in concert with the following technical lemma, which follows in a straightforward way from the results of Section 3.

\begin{lemma}\label{lbd.v2} Given $\Lambda, C_0\in (0,\infty)$, for any $\delta>0$ there exists $\epsilon_0(M,\Lambda,\delta,C_0)$ such that, if $\epsilon\in (0,\epsilon_0)$ and $(u,\nabla)$ is a smooth pair satisfying $|u|\leq 1$,
$$E_{\epsilon}(u,\nabla)\leq \Lambda,$$
and
$$\max_{r\geq \delta,\,p\in M} r^{2-n}\int_{B_r(p)}e_{\epsilon}(u,\nabla)\leq C_0,$$
then there exist a smooth $\phi:M\to S^1$ and a unit section $v\in \mathcal{U}_p(L)$ (i.e., $v\in W^{1,p}(M,S^1)$) for all $p \in (1, \frac{n}{n-1})$, satisfying
\begin{equation}\label{lem.con.1}\|u-v\|_{L^1(M)}\leq \delta,
\end{equation}
\begin{equation}\label{lem.con.2}
\|d(\phi^{-1}v)\|_{L^p(M)}\leq C(p,M,\Lambda),
\end{equation}
\begin{equation}\label{lem.con.3}
\mathbb{M}(J(v))\leq \Lambda,
\end{equation} 
and
\begin{equation}\label{lem.con.4}
\|J(v)\|(B_r(p))\leq 2C_0 r^{n-2}
\end{equation}
for all $p\in M$ and $r\geq \delta.$ Moreover, the map $\phi$ is chosen such that
$$\|\phi^*(d\theta)-\Pi(\alpha)\|_{L^2(M)}\leq C(M),$$
where $\nabla=d-i\alpha$ and $\Pi(\alpha)$ is the closed component of the Hodge decomposition of $\alpha$.
\end{lemma}

\begin{proof} The proof follows a straightforward argument by contradiction, using the analysis of Section \ref{SecLimInf}. If the statement were false, then we could find some fixed $\delta>0$, a sequence $\epsilon_j\to 0$, and pairs $(u_j,\nabla_j=d-i\alpha_j)$ such that
\begin{equation}\label{glob.bd}
E_{\epsilon_j}(u_j,\nabla_j)\leq \Lambda,
\end{equation}
and
\begin{equation}\label{loc.bd}
\max_{r\geq \delta,\,p\in M} r^{2-n}\int_{B_r(p)}e_{\epsilon_j}(u_j,\nabla_j)\leq C_0,
\end{equation}
for which there are no $\phi_j:M\to S^1$ and $v_j\in \mathcal{U}_p(L)$ satisfying \eqref{lem.con.1}--\eqref{lem.con.4}. By Lemma \ref{w1p.est} (and its proof), we can find maps $\phi_j: M\to S^1$ such that
$$\|d(\phi_j^{-1}u_j)\|_{L^p(M)}\leq C(p,M,\Lambda) \quad \text{and} \quad \|\alpha_j-\phi_j^*(d\theta)\|_{L^p(M)}\leq C(p,M,\Lambda)$$
for every $p\in (1,\frac{n}{n-1})$, while
$$\|\phi_j^*(d\theta)-\Pi(\alpha_j)\|_{L^2(M)}\leq C(M).$$
In particular, the maps $\phi_j^{-1}u_j$ are uniformly bounded in $W^{1,p}$ for $p\in (1,\frac{n}{n-1})$, and---as discussed in the proof of Theorem 1.2(i)---a subsequence therefore converges strongly in $L^1$ and weakly in $W^{1,p}$ to a singular unit section $v\in \mathcal{U}_p(L)$ (i.e., $v\in W^{1,p}(M,S^1)$, since $L$ is now trivial), while the gauge-invariant $(n-2)$-currents $J(u_j,\nabla_j)$ converge weakly to $J(v)$. Moreover, by \eqref{glob.bd}, \eqref{loc.bd}, and the lower semicontinuity of mass under weak convergence, we see that 
$$\mathbb{M}(J(v))\leq \liminf_{j\to\infty} \mathbb{M}(J(u_j,\nabla_j))\leq E_{\epsilon_j}(u_j,\nabla_j)\leq \Lambda$$
and
$$\|J(v)\|(B_r(p))\leq \liminf_{j\to\infty}\|J(u_j,\nabla_j)\|(B_r(p))\leq \liminf_{j\to\infty}\int_{B_r(p)}e_{\epsilon_j}(u_j,\nabla_j)\leq C_0r^{n-2}$$
for all $r\geq\delta$ and $p\in M$. In particular, for $j$ sufficiently large, we see that $\phi_j$ and $\phi_j v$ satisfy \eqref{lem.con.1}--\eqref{lem.con.4} (in place of $\phi$ and $v$) with respect to $u_j$, giving the desired contradiction.
\end{proof}

\begin{remark} In particular, recall from Corollary \ref{same.class} that for any $v\in \mathcal{U}_p(L)$ with $\mathbb{M}(J(v))<\infty$, we have $J(v)=2\pi\Gamma$ for an integral $(n-2)$-cycle $\Gamma\in \mathcal{Z}_{n-2}(M;\mathbb{Z})$.
\end{remark}

\subsection{Filling in cycles by filling maps}
The results of the preceding subsection will allow us to relate min-max families $F\in \mathcal{C}_2$ or $F\in \mathcal{C}_{\lambda}$ for the energies $E_{\epsilon}$ to certain discrete families of $(n-2)$-cycles with the desired mass bounds. In what follows, we collect some technical lemmas which will allow us to identify the images of those families of $(n-2)$-cycles under the Almgren isomorphism.

\begin{lemma}\label{interp.lem} 
Given $u,v\in W^{1,p}(M,S^1)$, for $p \in (1,2)$, there exists $w \in W^{1,p}(M\times [0,1],S^1)$ satisfying the boundary condition 
$$w(x,0)=u(x,0), \quad \text{and} \quad w(x,1)=v(x,1),$$
in the trace sense, for which the estimate 
$$\|\partial_tw\|_{L^p(M\times [0,1])}\leq C(p)\|u-v\|_{L^p(M)}$$
holds, and such that the pushforward $\pi_*[J(w)]$ of the distributional Jacobian $J(w)$ under the projection $\pi:M\times [0,1]\to M$ satisfies
$$\mathbb{M}(\pi_*[J(w)])\leq C\int_M |u-v|(|du|+|dv|).$$
\end{lemma}
\begin{proof} 
The proof combines ideas from \cite[Section 3]{CanevariOrlandi} and \cite{HardtKinderlehrerLin}. First, we mollify $u$ and $v$ to obtain maps $u_{\delta},v_{\delta}\in C^{\infty}(M,D^2)$ with
$$\|u_{\delta}-u\|_{W^{1,p}(M)}+\|v_{\delta}-v\|_{W^{1,p}(M)}<\delta.$$
Let $w_{\delta}: M\times [0,1]\to D^2$ be the linear interpolation
$$w_{\delta}(x,t):=(1-t)u_{\delta}(x)+tv_{\delta}(x).$$
Consider then the $(n-1)$-currents
$$\Gamma_y^{\delta}:=\pi_*[w_{\delta}^{-1}\{y\}]$$
given by pushing forward the $(n-1)$-dimensional submanifold $w_{\delta}^{-1}\{y\}$ for every regular value $y\in D$. Then for any $\zeta\in \Omega^{n-1}(M)$, the coarea formula gives
\begin{align*}
\langle \Gamma_y^{\delta},\zeta\rangle
&=\int_{w_{\delta}^{-1}\{y\}}\pi^*(\zeta)\\
&= \int_{w_{\delta}^{-1}\{y\}}*\Big(\zeta \wedge \frac{J(w_{\delta})}{|J(w_{\delta})|}\Big)\,d\mathcal{H}^{n-1}\\
&= \int_{w_{\delta}^{-1}\{y\}}*(\zeta \wedge dt\wedge \iota_{\partial_t}J(w_{\delta}))|J(w_{\delta})|^{-1}\,d\mathcal{H}^{n-1}.
\end{align*}
In particular, since
$$|\iota_{\partial_t}J(w_{\delta})|\leq 2|\partial_t w_{\delta}||dt\wedge dw_{\delta}|\leq 2|u_{\delta}-v_{\delta}|(|du_{\delta}|+|dv_{\delta}|),$$
it follows that
$$\mathbb{M}(\Gamma_y^{\delta})\leq \int_{w_{\delta}^{-1}\{y\}}\frac{|u_{\delta}-v_{\delta}|(|du_{\delta}|+|dv_{\delta}|)}{\frac{1}{2}|J(w_{\delta})|} \, d\mathcal{H}^{n-1},$$
and applying the coarea formula for $w_{\delta}$, we arrive at
\begin{equation}\label{gamma.mbd}
\int_{D}\mathbb{M}(\Gamma_y^{\delta})\leq \int_M |u_{\delta}-v_{\delta}|(|du_{\delta}|+|dv_{\delta}|).
\end{equation} 
Now, for each $y\in D_{1/4}$, fix a map $\Phi_y\in C^{\infty}(D_1\setminus \{y\}, S^1)$ satisfying
\begin{equation}
\Phi_y(z) = 
\begin{cases}
\frac{z-y}{|z-y|} & \text{for $z \in D_{1/4}(y)\subset D_{1/2},$} \\
\frac{z}{|z|} & \text{for $\vert z \vert \geq 3/4,$}
\end{cases}
\end{equation}
and
$$|d\Phi_y(z)|\leq \frac{C}{|z-y|}\text{ on }D_1$$
for some fixed constant $C$. Then, writing
$$w_{\delta,y}:=\Phi_y\circ w_{\delta},$$
if $y\in D_{1/4}$ is a regular value of $w_{\delta}$, we see that $w_{\delta,y}$ belongs to $W^{1,p}(M\times[0,1],S^1)$ and satisfies $J(w_{\delta,y})=2\pi w_{\delta}^{-1}\{y\}$, as well as
$$\|dw_{\delta,y}\|_{L^p(M\times[0,1])}^p\leq C\int_{M\times[0,1]} |dw_{\delta}|(x,t)^p|w_{\delta}(x,t)-y|^{-p} \, dx\,dt$$
and
$$\|\partial_tw_{\delta,y}\|_{L^p(M\times[0,1])}^p\leq C\int_{M\times[0,1]} |u_{\delta}-v_{\delta}|^p(x)|w_{\delta}(x,t)-y|^{-p} \, dx\,dt.$$

Integrating the latter two estimates over $y\in D_{1/4}$ and applying Fubini's theorem, we see that
\begin{align*}
\int_{D_{1/4}}\|dw_{\delta,y}\|_{L^p(M\times[0,1])}^p\,dy
&\leq \int_{M\times[0,1]}|dw_{\delta}(x,t)|^p\Big( \int_{D_{1/4}}|w_{\delta}(x,t)-y|^{-p}\,dy\Big)\,dx\,dt\\
&\leq C(p)\|dw_{\delta}\|_{L^p(M\times[0,1])}^p,
\end{align*}
and similarly
$$\int_{D_{1/4}}\|\partial_tw_{\delta,y}\|_{L^p(M\times[0,1])}^p\leq C(p)\|u_{\delta}-v_{\delta}\|_{L^p(M)}^p.$$
Combining these estimates together with \eqref{gamma.mbd}, we can find $y=y_{\delta}\in D_{1/4}$ such that
$$\|dw_{\delta,y}\|_{L^p(M\times[0,1])}\leq C(p)\|dw_{\delta}\|_{L^p(M\times[0,1])}$$
and
$$\|\partial_tw_{\delta,y}\|_{L^p(M\times[0,1])}\leq C(p)\|u_{\delta}-v_{\delta}\|_{L^p(M)}^p,$$
together with
$$\mathbb{M}(\pi_*[J(w_{\delta,y})])=2\pi\mathbb{M}(\Gamma_y^{\delta})\leq C\int_M|u_{\delta}-v_{\delta}|(|du_{\delta}|+|dv_{\delta}|).$$
Since $w_{\delta,y_{\delta}}$ is bounded in $W^{1,p}(M\times[0,1], S^1)$, we may take a subsequential limit 
$$w=\lim_{\delta\to 0}w_{\delta,y_{\delta}}$$
as $\delta\to 0$, to obtain a map $w\in W^{1,p}(M\times[0,1],S^1)$ with the desired properties.
\end{proof}

\begin{remark} On a manifold with Lipschitz boundary $(N, \partial N)$ of dimension $m$ (e.g.\ $N=M\times [0,1]$ or $N=M\times [0,1]^2$ where $M$ is our underlying manifold), given a map $w\in W^{1,p}(N,S^1)\cap W^{1,p}(\partial N,S^1)$, recall that the \textit{(interior) distributional Jacobian} $J(w)$ is the $(m-2)$-current given by
\begin{equation} \label{BoundaryJac}
\langle J(w),\zeta\rangle:=\int_N w^*(d\theta)\wedge d\zeta+\int_{\partial N}w^*(d\theta)\wedge \zeta.
\end{equation}
\end{remark}

In the sequel, we endow $M\times[0,1]$ with the orientation such that $M\times\{1\}$ is oriented as $M$.
Using the product orientation on $M\times [0,1]^2$ and the induced one on the boundary $M\times\de[0,1]^2$, note that $\tau\wedge v$ is positively oriented on the latter manifold when $v$ is a positively oriented $n$-vector of $M$ and $\tau$ is tangent to $\de[0,1]^2$, pointing counter-clockwise.

\begin{remark}\label{line.concat} The distributional Jacobian interacts well with concatenation of maps. Indeed, for any two $w_1,w_2\in W^{1,p}(M\times [0,1],S^1)\cap W^{1,p}(M\times \{0,1\},S^1)$, if $w_1*w_2:M\times [0,1]\to S^1$ is the usual concatenation, we have that 
$$\pi_*[J(w_1*w_2)]=\pi_*[J(w_1)]+\pi_*[J(w_2)].$$
Reasoning by induction one can then prove that the above identity holds for an arbitrary finite concatenation.
\end{remark}

\begin{lemma}\label{square.fill} Let $I^2=[0,1]^2$, and let $F\in W^{1,p}(M\times I^2,S^1)\cap W^{1,p}(M\times \partial I^2, S^1)$. Letting $\pi: M\times I^2\to M$ be the canonical projection, the $n$-current
$$\Xi:=\pi_*[J(F)]\in \mathcal{D}_n(M)$$
depends only on $F|_{M\times \partial I^2}$, is given by
$$\langle \Xi, \varphi\,\dvol_g\rangle=2\pi\int_M \varphi(x)\deg \Big( F \vert_{\{ x \}\times \partial I^2} \Big) \, dx,$$
and satisfies
$$\mathbb{M}(\Xi)\leq \|\partial_tF\|_{L^1(M\times \partial I^2)},$$
where $\partial_tF$ denotes the partial derivative of $F$ along the $\partial I^2$ direction.
\end{lemma}

\begin{proof} Since any $n$-form $\zeta\in \Omega^n(M^n)$ is closed, \eqref{BoundaryJac} implies 
$$\langle J(F),\pi^*\zeta\rangle=\int_{M\times \partial I^2}F^*(d\theta)\wedge\pi^*\zeta =\int_{M}\zeta(x)\Big(\int_{\{x\} \times \partial I^2}F^*(d\theta)\Big) \, dx,$$ from which the desired results follow.
\end{proof}

Hence, if $F_1,F_2\in W^{1,p}(M\times I^2,S^1)\cap W^{1,p}(M\times \partial I^2,S^1)$ are two such maps, satisfying
$$F_1(x,1,t)=F_2(x,0,t),$$
and $\Phi=F_1*F_2$ is the map given by concatenating along one face of the square, i.e.,
\begin{equation}
\Phi(x,s,t) := \begin{cases}
F_1(x,2s,t) & \text{on }M\times [0,1/2]\times I, \\
F_2(x,2s-1,t) & \text{on }M\times [1/2,1]\times I, 
\end{cases}
\end{equation}
we have 
\begin{equation}\label{box.concat}
\pi_*[J(F_1)]+\pi_*[J(F_2)]=\pi_*[J(F_1*F_2)].
\end{equation}
Of course, the same statement holds if we define $F_1*F_2$ by concatenation along any other face of $I^2$. 

\subsection{One-parameter families corresponding to $\bm{\pi_1(\mathcal{Z}_{n-2}(M;\mathbb{Z}),0)}$.}
We come now to the proof of the second inequality in Theorem \ref{min-max.comp}, comparing the one-parameter min-max constructions for the $U(1)$-Higgs energies and the $(n-2)$-mass. That is, for any $\lambda\in H_{n-1}(M;\mathbb{Z})$, our goal in this section is to prove that
\begin{equation}\label{one.par.comp}
\liminf_{\epsilon\to 0}\mathcal{E}_{\epsilon}(\lambda)\geq 2\pi {\bf W}(\lambda).
\end{equation}

To this end, fix $0\neq \lambda\in H_{n-1}(M;\mathbb{Z})$ and a small constant $\delta>0$.
Let $\psi\in C^{\infty}(M,S^1)$ be a map whose (regular) fibers lie in $\lambda\in H_{n-1}(M;\mathbb{Z})$.
Recall that, by definition of $\mathcal{C}_{\lambda}$, the endpoints $(u_0,\nabla_0)$ and $(u_1,\nabla_1)$ of a family $(u_t,\nabla_t)_{t\in[0,1]}$ in $\mathcal{C}_\lambda$ are given by
$$(u_0,\nabla_0)=(1,d) \quad \text{and} \quad (u_1,\nabla_1)=(\psi, d-i\psi^*(d\theta)).$$
We claim that
\begin{equation}\label{one.finite}
	\Lambda:=\liminf_{\epsilon\to 0}\mathcal{E}_{\epsilon}(\lambda)<\infty.
\end{equation}

\begin{proof}[Proof of \eqref{one.finite}] Since the proof is very similar to the one for two-parameter families, given in \cite[Section~7]{PigatiStern}, we just sketch it.
Identifying $M$ with a simplicial complex $\tilde M$ in some Euclidean space $\R^L$, by means of a triangulation of $M$, we can find a piecewise affine map $\tilde\psi:\R^L\to\C$ such that $\tilde\psi=1$ far from $\tilde M$ and $|\tilde\psi-\psi|<\frac{1}{2}$ on $\tilde M$ (provided the triangulation was chosen fine enough). Let $y$ be a small regular value of $\tilde\psi$.

By composing $\psi$ with a piecewise affine homeomorphism of $\C$, we can assume that $y=0$ and that $\tilde\psi^{-1}(\bar D_{1/2})$ is an $O(\epsilon)$-neighborhood of $\tilde\psi^{-1}(0)$, with the bound $|d\tilde\psi|=O(\epsilon^{-1})$.
In particular, the fiber $\tilde\psi^{-1}(0)$ is contained in finitely many affine $(L-2)$-planes $P_j$. With a slight perturbation of $\tilde M$, which does not intersect $\tilde\psi^{-1}(\bar D_{1/2})$, we can assume that all the simplices in $\tilde M$ are transverse to each $P_j$ (when both are translated to the origin).

Now we can apply \cite[Proposition~7.13]{PigatiStern} to (a regularization of)
the maps $\tilde\psi(\cdot-C(1-t))$, with $t\in[0,1]$ and $C$ big enough.
The preimage of $\bar D_{1/2}$ under these maps intersects $\tilde M\cong M$ in an $O(\epsilon)$-neighborhood of $[\tilde\psi^{-1}(0)+C(1-t)]\cap\tilde M$, which has volume $O(\epsilon^2)$.
Also, for $t=0$ the initial map is constant and equal to $1$ (for $C$ big enough).

The aforementioned proposition gives then a family $(u_t,\nabla_t)_{t\in[0,1]}$ with uniformly bounded energy
from $(1,d)$ to $(\bar\psi,d-i\bar\psi^*(d\theta))$, for some $\bar\psi:M\to S^1$ homotopic to $\psi$. Concatenating this family with $(\psi_t,d-i\psi_t^*(d\theta))$, for a homotopy $\psi_t$ from $\bar\psi$ to $\psi$, we get a family in $C_\lambda$ with the same energy, as desired.
\end{proof}

Now, consider a small $\epsilon\in (0,\delta)$ such that
\begin{equation}
\mathcal{E}_{\epsilon}(\lambda)\leq \Lambda+\delta<\Lambda+1.
\end{equation}
By Lemma \ref{tamed.fam} and Remark \ref{trunc}, we can find a family $[0,1]\ni t\mapsto (u_t,\nabla_t=d-i\alpha_t)$ in $\mathcal{C}_{\lambda}\subset C^0([0,1],\widehat{X})$ such that $|u_t|\leq 1$,
\begin{equation}
\max_{t\in [0,1]}E_{\epsilon}(u_t,\nabla_t)\leq \mathcal{E}_{\epsilon}(\lambda)+\epsilon\leq \Lambda+2,
\end{equation}
and
\begin{equation}
\max_{t\in[0,1],\,r\geq \epsilon,\,p\in M}r^{2-n}\int_{B_r(p)}e_{\epsilon}(u_t,\nabla_t)\leq C_0(M,\Lambda).
\end{equation}

Now, by the continuity of the path $t\mapsto (u_t,\nabla_t=d-i\alpha_t)$ in $\widehat{X}$, we may select a finite sequence of times
$$0=t_0<t_1<\cdots<t_{N=3^k}=1$$
such that
$$\|u_{t_{i+1}}-u_{t_i}\|_{W^{1,2}(M)}+\|\alpha_{t_{i+1}}-\alpha_{t_i}\|_{W^{1,2}(M)}<\delta.$$
In what follows, we write $u_i=u_{t_i}$ and $\alpha_i=\alpha_{t_i}$. Suppose now that $\epsilon<\epsilon_0(M,\Lambda+2,\delta,C_0)$ as in Lemma \ref{lbd.v2},
and for each $i=1,\ldots, N=3^k$, let
$$v_i\in W^{1,p}(M,S^1) \quad \text{and} \quad \phi_i:M\to S^1$$
be as in the conclusion of Lemma \ref{lbd.v2}, so that
$$\|u_i-v_i\|_{L^1(M)}\leq \delta,$$
and 
\begin{equation}\label{phiv.sob.bd}
\|d(\phi_i^{-1}v_i)\|_{L^p(M)}\leq C(p,M,\Lambda)
\end{equation}
for $p\in(1,\frac{n}{n-1})$, while
$$\mathbb{M}(J(v_i))\leq \Lambda+2\delta,$$
together with
$$\max_{r\geq \delta,\,p\in M}\frac{\|J(v_i)\|(B_r(p))}{r^{n-2}}\leq 2C_0,$$
and
\begin{equation}\label{phi.char}
\|\phi_i^*(d\theta)-\Pi(\alpha_i)\|_{L^2(M)}\leq C(M).
\end{equation}

In this way, we get a sequence
$$1=v_0,v_1,\ldots,v_N=\psi\text{ in }W^{1,p}(M,S^1)$$
such that 
$$\|v_{i+1}-v_i\|_{L^1(M)}\leq C\delta$$
and the integral $(n-2)$-cycles $T_i:=\frac{1}{2\pi}J(v_i)$ satisfy
$$2\pi\mathbb{M}(T_i)\leq \Lambda+2\delta$$
and
$$\max_{r\geq \delta,\,p\in M}\frac{\|T_i\|(B_r(p))}{r^{n-2}}\leq C_0.$$
Moreover, for each $i=0,\ldots,N-1$, the following holds.

\begin{lemma}\label{fill-in} For $p\in [1,\frac{n}{n-1})$, there exists $w_i\in W^{1,p}(M\times [0,1],S^1)$ with boundary values 
$$w_i(x,0)=v_i(x), \quad w_i(x,1)=v_{i+1}(x),$$ 
satisfying 
$$\|\partial_tw_i\|_{L^p(M\times [0,1])}\leq C(p)\|v_{i+1}-v_i\|_{L^p(M)}\leq C(p) \delta^{1/p}$$
and
$$\mathbb{M}(\pi_*[J(w_i)])\leq C(p,M,\Lambda)\delta^{1-1/p}.$$
\end{lemma}
\begin{proof} To begin, apply Lemma \ref{interp.lem} with $u=\phi_i^{-1}v_i$ and $v=\phi_i^{-1}v_{i+1}$, to obtain a map $\tilde{w}\in W^{1,p}(M\times [0,1])$ which restricts to $\phi_i^{-1}v_i$ and $\phi_i^{-1}v_{i+1}$ on $M\times \{0,1\}$, 
and satisfies
$$\|\partial_t\tilde{w}\|_{L^p(M\times [0,1])}\leq C(p)\|\phi_i^{-1}(v_i-v_{i+1})\|_{L^p(M)}=C(p)\|v_{i+1}-v_i\|_{L^p(M)}$$
and
\begin{align*}
\mathbb{M}(\pi_*[J(\tilde{w})]) 
& \leq C\int_M |\phi_i^{-1}(v_i-v_{i+1})|(|d(\phi_i^{-1}v_i)|+|d(\phi_i^{-1}v_{i+1})|)\\
& \leq \|v_i-v_{i+1}\|_{L^{p'}(M)}(\|d(\phi_i^{-1}v_i)\|_{L^p(M)}+\|d(\phi_i^{-1}v_{i+1})\|_{L^p(M)}).
\end{align*}
Now, we know that
$$\|v_i-v_{i+1}\|_{L^{p'}(M)}\leq C(p)\|v_{i+1}-v_i\|_{L^1(M)}^{1-1/p}\leq C(p) \delta^{1-1/p}$$
and
$$\|d(\phi_i^{-1}v_i)\|_{L^p(M)}\leq C(p,M,\Lambda),$$
while
\begin{align*}
\|d(\phi_i^{-1}v_{i+1})\|_{L^p(M)} & = \|v_{i+1}^*(d\theta)-\phi_i^*(d\theta)\|_{L^p(M)}\\
&\leq \|d(\phi_{i+1}^{-1}v_{i+1})\|_{L^p(M)}+\|\phi_{i+1}^*(d\theta)-\phi_i^*(d\theta)\|_{L^p(M)}\\
&\leq C(p,M,\Lambda)+\|\phi_{i+1}^*(d\theta)-\Pi(\alpha_{i+1})\|_{L^p(M)}+\|\phi_i^*(d\theta)-\Pi(\alpha_i)\|_{L^p(M)} \\
& \quad + \|\Pi(\alpha_i-\alpha_{i+1})\|_{L^p(M)}\\
&\leq C(p,M,\Lambda),
\end{align*} 
which together with the preceding estimates gives
$$\mathbb{M}(\pi_*[J(\tilde{w})])\leq C(p,M,\Lambda)\delta^{1-1/p}.$$
Taking $w_i:=\phi_i\tilde{w}$, one sees that $w_i$ satisfies the conclusions of the claim, since $J(w_i)=J(\tilde{w})$ and $\partial_t w_i=\phi_i \partial_t \tilde{w}$.
\end{proof}

In particular, by \eqref{BoundaryJac}, we see that the $(n-1)$-currents $\Gamma_i:=\frac{1}{2\pi}\pi_*[J(w_i)]\in \mathcal{I}_{n-1}(M;\mathbb{Z})$ give fill-ins
$$\partial \Gamma_i=T_{i+1}-T_i$$
of small mass (taking $p=\frac{n+1}{n}$)
$$\mathbb{M}(\Gamma_i)\leq C(M,\Lambda)\delta^{1/(n+1)}.$$
Thus, the sequence $T_0,T_1,\ldots, T_{N=3^k}$ defines a discrete family
$$\beta: I(1,k)_0\to \mathcal{Z}_{n-2}(M;\mathbb{Z})$$
with
$${\bf m}(\beta,r)\leq C_0r^{n-2} \quad \text{for }r\geq \delta,$$
together with
$$\max_i \mathbb{M}(T_i)\leq \frac{1}{2\pi}(\Lambda+2\delta),$$
and
$${\bf f}(\beta)\leq C(M,\Lambda)\delta^{1/(n+1)}.$$
Moreover, for $\delta<\delta_0(M,\Lambda)$ sufficiently small, the homology class $\Psi(\beta)\in H_{n-1}(M;\mathbb{Z})$ associated to $\beta$ by Almgren's isomorphism is given by
$$\Psi(\beta):=[\Gamma],$$
where
$$\Gamma:={\textstyle\sum_{i=0}^{N-1}}\Gamma_i.$$
Now, by Remark \ref{line.concat}, we can identify $\Gamma$ with the projected Jacobian
$$2\pi\Gamma=\pi_*[J(w_0*w_1*\cdots*w_{N-1})]=\pi_*[J(w)]$$
of the concatenated map $w:=w_0*\cdots *w_{N-1}: M\times [0,1]\to S^1$, which satisfies
$$w(x,0)=1 \quad \text{and} \quad w(x,1)=\psi(x).$$
In particular, for any $\zeta\in\Omega^{n-1}(M^n)$, it follows that
$$2\pi\langle \Gamma,\zeta\rangle=\int_{M\times [0,1]}w^*(d\theta)\wedge d\zeta +\int_M\psi^*(d\theta)\wedge \zeta.$$
Hence, the action of $\Gamma$ on \emph{closed} $(n-1)$-forms agrees with that of $\frac{1}{2\pi}\int_M \psi^*(d\theta)\wedge \cdot$. In particular, since there is no torsion in $H_{n-1}(M;\mathbb{Z})$, it follows that 
$$[\Gamma]=[\psi^{-1}\{\theta\}]=\lambda\in H_{n-1}(M;\mathbb{Z}),$$
as desired.

That is, letting $\eta(\delta):=\max\{\delta, C\delta^{1/(n+1)}\}$, we see that $\beta\in \mathcal{A}_{\eta(\delta)}(\lambda)$, so that
$${\bf W}_{\eta(\delta)}(\lambda)\leq \max_i \mathbb{M}(T_i)\leq \frac{1}{2\pi}(\Lambda+2\delta)=\frac{1}{2\pi}\liminf_{\epsilon\to 0}\mathcal{E}_{\epsilon}(\lambda)+\frac{1}{\pi}\delta.$$
Finally, taking the limit as $\delta \to 0$ and using \eqref{W.Wdelta}, we get the desired estimate \eqref{one.par.comp}.

\subsection{Two-parameter families and the generator of $\bm{\pi_2(\mathcal{Z}_{n-2}(M;\mathbb{Z}),0)}$.}

In this subsection, we complete the proof of Theorem \ref{min-max.comp}, establishing the inequality for the two-parameter families
\begin{equation}\label{2.par.comp}
\liminf_{\epsilon\to 0}\mathcal{E}_{\epsilon}(\mathcal{C}_2)\geq 2\pi {\bf W}([M]).
\end{equation}

To begin, set
$$\Lambda:=\liminf_{\epsilon\to 0}\mathcal{E}_{\epsilon}(\mathcal{C}_2),$$
which is finite (see \cite[Section~7]{PigatiStern}),
and fix some small $\delta>0$. Again let $L\to M$ be the trivial line bundle, and consider a two-parameter family
$$\bar D^2\ni y\mapsto (u_y,\nabla_y=d-i\alpha_y)$$
belonging to $\mathcal{C}_2\subset C^0(\bar D^2,\widehat{X})$, so that
$$(u_{\theta},\nabla_{\theta})\equiv (\theta,d) \quad \text{for all }\theta\in \partial D=S^1.$$
Choose a small $\epsilon\in (0,\delta)$ such that
$$\mathcal{E}_{\epsilon}(\mathcal{C}_2)\leq \Lambda+\delta;$$
by Lemma \ref{tamed.fam} and the subsequent remark, we can select our family $\bar D\ni y\mapsto (u_y,\nabla_y)$ in $\mathcal{C}_2$ such that $|u_y|\leq 1$, 
$$\max_{y\in \bar D}E_{\epsilon}(u_y,\nabla_y)\leq \Lambda+2\delta,$$
and
$$\max_{y\in \bar D,\,r\geq \delta,\,p\in M}r^{2-n}\int_{B_r(p)}e_{\epsilon}(u_y,\nabla_y)\leq C_0(M,\Lambda).$$

Now, identifying $\bar D$ with the square $I^2=[0,1]^2$ in the usual bi-Lipschitz way, by the continuity of the family $I^2\cong \bar D\ni y\mapsto (u_y,\nabla_y)\in \widehat{X}$, we can choose $k$ sufficiently large that the discrete assignment
$$I(2,k)_0\ni a \mapsto (u_a,\nabla_a)=(u_a,d-i\alpha_a)\in \widehat{X}$$
satisfies
$$\|u_a-u_b\|_{W^{1,2}(M)}+\|\alpha_a-\alpha_b\|_{W^{1,2}(M)}<\delta$$
for any adjacent vertices $a,b\in I(2,k)_0$. By Lemma \ref{lbd.v2}, for each vertex $a\in I(2,k)_0$, there exist
$$v_a\in W^{1,p}(M,S^1)\quad\text{and}\quad\phi_a:M\to S^1$$
such that
$$\|u_a-v_a\|_{L^1(M)}\leq \delta$$
and
$$\|d(\phi_a^{-1}v_a)\|_{L^p(M)}\leq C(p,M,\Lambda)$$
for $p\in [1,\frac{n}{n-1})$, while
$$\mathbb{M}(J(v_a))\leq \Lambda+2\delta,$$
together with 
$$\max_{r\geq \delta,\,p\in M}\frac{\|J(v_a)\|(B_r(p))}{r^{n-2}}\leq 2C_0,$$
and
$$\|\phi_a^*(d\theta)-\Pi(\alpha_a)\|_{L^2(M)}\leq C(M).$$
The following lemma, and its proof, is identical to Lemma \ref{fill-in}.

\begin{lemma} For each pair of adjacent vertices $a,b\in I(2,k)_0$, there exists $w_{a,b}\in W^{1,p}(M\times [0,1],S^1)$ satisfying the boundary conditions 
$$w_{a,b}(x,0)=v_a(x) \quad \text{and} \quad w_{a,b}(x,1)=v_b(x),$$ 
while for every $p\in [1,\frac{n}{n-1})$,
$$\|\partial_t w_{a,b}\|_{L^p(M\times [0,1])}\leq C(p)\|v_b-v_a\|_{L^p(M)}\leq C(p)\delta^{1/p},$$
and
$$\mathbb{M}(\pi_*[J(w_{a,b})])\leq C(p,M,\Lambda)\delta^{1-1/p}.$$
\end{lemma}

\begin{remark} If the vertices $a,b$ lie on the boundary $\partial I^2$, so that $u_a$ and $u_b$ are constant maps to $S^1$, then we take $v_a=u_a$, $v_b=u_b$, and simply let $w_{a,b}$ be the geodesic interpolation in $S^1$ between the two constants.
\end{remark}

In particular, for each pair of adjacent vertices $a,b\in I(2,k)_0$, the $(n-1)$-current 
$$\Gamma_{a,b}:=\frac{1}{2\pi}\pi_*[J(w_{a,b})]\in \mathcal{I}_{n-1}(M;\mathbb{Z})$$
provides a small-mass fill-in
$$\partial \Gamma_{a,b}=T_b-T_a$$
for the difference of the integral $(n-2)$-cycles $T_a:=\frac{1}{2\pi}J(v_a)$; namely, taking $p=\frac{n+1}{n}$ in the preceding lemma, we have
$$\mathbb{M}(\Gamma_{a,b})\leq C(M,\Lambda)\delta^{\frac{1}{n+1}}.$$
Thus, setting $\beta(a):=T_a$ gives a discrete family
$$\beta: I(2,k)_0\to \mathcal{Z}_{n-2}(M;\mathbb{Z})$$
satisfying
$${\bf m}(\beta,r)\leq C_0r^{n-2}\quad\text{for }r\geq \delta,$$
together with
$$\max_{a\in I(2,k)_0}\mathbb{M}(T_a)\leq \frac{1}{2\pi}(\Lambda+2\delta),$$
and
$${\bf f}(\beta)\leq C(M,\Lambda)\delta^{\frac{1}{n+1}}.$$
It remains to show that the homology class $\Psi(\beta)\in H_n(M;\mathbb{Z})$ associated to $\beta$ by Almgren's isomorphism is the fundamental class $[M]$.

For each $2$-cell $\Box\in I(2,k)_2$ with vertices $a,b,c,d$ (ordered counter-clockwise),
let $F:M\times \partial I^2\to S^1$ be the concatenation given by $w_{a,b}$ along the edge $[a,b]$ of $\partial I^2$, $w_{b,c}$ on $[b,c]$, and so on.
We apply Lemma \ref{interp.lem} to interpolate between $F$ and $1$, obtaining an extension $F_{\Box}\in W^{1,p}(M\times I^2,S^1)\cap W^{1,p}(M\times \partial I^2,S^1)$ of the map $F$, so that 
$$\Xi_{\Box}:=\frac{1}{2\pi}\pi_*[J(F_\Box)]\in \mathcal{I}_n(M;\mathbb{Z})$$
has boundary
$$\partial \Xi_{\Box}=\frac{1}{2\pi}\pi_*[J(F)]=\Gamma_{a,b}+\Gamma_{b,c}+\Gamma_{c,d}+\Gamma_{d,a}.$$
In particular, since $\|\partial_tw_{a,b}\|_{L^p(M\times[a,b])}\leq C(p)\delta^{1/p}=C(n)\delta^{n/(n+1)}$, it follows from Lemma \ref{square.fill} that $\Xi_{\Box}$ is the (unique) small-mass fill-in of $\Gamma_{a,b}+\cdots+\Gamma_{d,a}$, provided $\delta<\delta_0(M,\Lambda)$ is sufficiently small. In particular, we see that
$$\Psi(\beta)=[{\textstyle\sum_{\Box \in I(2,k)_2}}\Xi_{\Box}]\in H_n(M;\mathbb{Z}).$$

By concatenating the maps $F_1$ and $F_2$ associated to adjacent boxes $\Box_1, \Box_2$ along the shared edge, we obtain a map $\Phi=F_1*F_2$ which satisfies
$$\pi_*[J(\Phi)]=\Xi_{\Box_1}+\Xi_{\Box_2}.$$
In particular, concatenating all maps along each row of the grid, we obtain a column of maps, which we may again concatenate to obtain finally a map
$$F\in W^{1,p}(M\times I^2,S^1)\cap W^{1,p}(M\times \partial I^2,S^1)$$
for which
$$\pi_*[J(F)]=2\pi{\textstyle\sum_\Box} \Xi_{\Box}.$$
On the other hand, it is clear from the construction that the restriction of $F$ to $M\times \partial I^2$ has the form
$$F(x,t)=h(t)$$
for a fixed homeomorphism $h: \partial I^2\to S^1$. In particular, it follows that
$$\deg \Big( F|_{\{x\} \times \partial I^2}\Big)=1$$
for all $x\in M$, so that
$$2\pi{\textstyle\sum_\Box}\Xi_{\Box}=\pi_*[J(F)]=2\pi [M],$$
by Lemma \ref{square.fill}. 

Thus, $\Psi(\beta)=[M]$, as desired, and again setting $\eta(\delta):=\max\{\delta, C\delta^{1/(n+1)}\}$, we see that $\beta\in \mathcal{A}_{\eta(\delta)}([M])$, and consequently
$${\bf W}_{\eta(\delta)}([M])\leq \max_{a\in I(2,k)_0}\mathbb{M}(T_a)\leq \frac{1}{2\pi}(\Lambda+2\delta)=\frac{1}{2\pi}\liminf_{\epsilon\to 0}\mathcal{E}_{\epsilon}(\mathcal{C}_2)+\frac{1}{\pi}\delta.$$
Taking the limit as $\delta\to 0$ and using \eqref{W.Wdelta}, we then get the desired estimate \eqref{2.par.comp}, completing the proof of Theorem \ref{min-max.comp}.

\section{Huisken-type monotonicity along the gradient flow}\label{grad.sec}

In Lemma \ref{tamed.fam} of the previous section, we made use of the fact that a continuous family of pairs $y\mapsto (u_y,\nabla_y)$ may be deformed to a family $(u_y',\nabla_y')$ with $E_{\epsilon}(u'_y,\nabla'_y)\leq E_{\epsilon}(u_y,\nabla_y)$ satisfying uniform bounds on the $(n-2)$-energy densities $r^{2-n}\int_{B_r(p)}e_{\epsilon}(u_y',\nabla_y')$ in terms of the initial energies $E_{\epsilon}(u_y,\nabla_y)$. We achieve this by showing that the natural $L^2$ gradient flow for these energies satisfies a variant of Huisken's monotonicity formula \cite{Huisk} for the codimension-two mean curvature flow. In addition to its applications above, the result may be of independent interest, in that it provides strong evidence that these gradient flows provide a regularization of the codimension-two Brakke flow---a relationship which we plan to explore further in future work. We also show that this $E_{\epsilon}$-gradient flow satisfies long-time existence and continuous dependence on initial data (the fact that we are working with the abelian gauge group $U(1)$ is of course crucial here).

\subsection{Definition, Bochner identities, and estimates along the gradient flow}
Let $L\to M$ be the trivial line bundle over a closed, oriented Riemannian manifold $(M^n,g)$.
We will assume $n\ge 3$ throughout this section.

We will say that the smooth couples $(u_t,\nabla_t=d-i\alpha_t)_{t\in[0,\infty)}$ solve the gradient flow equations for $E_{\epsilon}$ if they satisfy the coupled nonlinear heat equations
\begin{align} \label{gf}
\left\{
\begin{aligned} 
\partial_t u_t&=-\nabla_t^*\nabla_tu_t+{\textstyle\frac{1}{2\epsilon^2}}(1-|u_t|^2)u_t, \\
\partial_t\alpha_t&=-d^*d\alpha_t+\epsilon^{-2}\langle iu_t,\nabla_tu_t\rangle.
\end{aligned}
\right.
\end{align}
Note that they are formally the gradient flow of $\frac{1}{2}E_\epsilon$ with respect to the $L^2$-scalar product 
\begin{equation*}
\ang{(u,\alpha),(v,\beta)}=\int_M(\ang{u,v}+\epsilon^2\ang{\alpha,\beta}), 
\end{equation*}
where $u$ and $v$ are sections, and $\alpha$ and $\beta$ are one-forms. We defer the proof of long-time existence, uniqueness and continuous dependence on initial data to the end of the section. In what follows, we will also assume that the initial section $u_0\in \Gamma(L)$ satisfies $|u_0|\leq 1$ pointwise.

Assuming the initial data $(u_0,\nabla_0)$ satisfies the energy bound
\begin{equation}
E_{\epsilon}(u_0,\nabla_0)\leq \Lambda, 
\end{equation}
it is easy to see that we have
$$E_{\epsilon}(u_t,\nabla_t)\leq \Lambda$$
for all $t>0$, as the energy is decreasing along the flow. Similar to results for the stationary case in \cite{PigatiStern} (and analogous work of Ilmanen for the parabolic Allen--Cahn equation in codimension one \cite{Ilmanen}), a key ingredient in establishing the desired monotonicity result will be bounding the discrepancy function
\begin{equation}\label{discr}
\xi_t:=\epsilon|d\alpha_t|-\frac{1-|u_t|^2}{2\epsilon}
\end{equation}
along the flow.

As in the stationary case \cite[Section~3]{PigatiStern}, it is straightforward to check that solutions of \eqref{gf} satisfy the following identities: letting 
$$\omega_t:=d\alpha_t$$
and 
$$\psi(u_t,\nabla_t)(e_j,e_k):=2\langle i\nabla_{e_j}u,\nabla_{e_k}u\rangle,$$
we have
\begin{equation}
\epsilon^2(\partial_t+\Delta_H)\omega_t=\psi(u_t,\nabla_t)-|u_t|^2\omega_t,
\end{equation}
from which one obtains the parabolic Bochner identity
\begin{equation}\label{omega.boch}
-\epsilon^2(\partial_t+d^*d)\frac{1}{2}|\omega_t|^2=|u_t|^2|\omega_t|^2+\epsilon^2|D\omega_t|^2-\langle \psi(u_t,\nabla_t),\omega_t\rangle+\epsilon^2\mathcal{R}_2(\omega_t,\omega_t),
\end{equation}
where $\mathcal{R}_2$ denotes the Weitzenb\"ock curvature operator for two-forms. Also,
\begin{equation}\label{u.boch}
-(\partial_t+d^*d)\frac{1}{2}|u_t|^2=|\nabla_t u_t|^2-\frac{1}{2\epsilon^2}(1-|u_t|^2)|u_t|^2.
\end{equation}

\begin{remark}\label{ut.bd} It is an easy consequence of \eqref{u.boch} and the parabolic maximum principle that $|u_t|\leq 1$ for all $t>0$, for initial sections $u_0$ satisfying $|u_0|\leq 1$.
\end{remark}

By a combination of \eqref{omega.boch} and \eqref{u.boch}, similarly to \cite{PigatiStern}, we find that the discrepancy function in \eqref{discr} satisfies the weak differential inequality 
\begin{equation}
-(\partial_t+d^*d+\epsilon^{-2}|u_t|^2)\xi_t\geq-C_0(M)\epsilon|\omega_t|.
\end{equation}
Equivalently, writing
$$\bar{\xi}_t:=e^{-C_0t}\xi_t,$$
we have
\begin{equation}\label{xi.sub}
-(\partial_t+d^*d+\epsilon^{-2}|u_t|^2)\bar{\xi}_t\geq -C_0e^{-C_0t}\frac{1-|u_t|^2}{2\epsilon}
\geq -\frac{C_0}{2\epsilon}(1-|u_t|^2).
\end{equation}

Now, let $K(t,x,y)$ be the heat kernel of $M$, so that
$$(\partial_t+d^*d)K(t,\cdot,y)=0 \quad \text{and} \quad \lim_{t\to 0}K(t,\cdot,y)=\delta_y.$$
Define then 
$$\varphi(t,x):=\int_MK(t,x,y)|\xi_0|(y)\,dy$$
and
$$\psi(t,x):=\int_0^t\int_M K(t-s,x,y)\frac{C_0}{2\epsilon}(1-|u_s|^2)(y)\,dy\,ds.$$
Thus, $\varphi$ is the nonnegative solution of the heat equation
$-(\partial_t+d^*d)\varphi=0 $, with initial condition $\varphi(0,x)=|\xi_0(x)|$. 
By Duhamel's principle, $\psi$ is the nonnegative solution of the inhomogeneous heat equation
\begin{equation*}
-(\partial_t+d^*d)\psi=-\frac{C_0}{2\epsilon}(1-|u_t|^2),
\end{equation*}
with boundary data $\psi(0,x)=0.$ In particular, it follows from \eqref{xi.sub} that
\begin{equation}\label{xi.subeq}
- (\partial_t+d^*d+\epsilon^{-2}|u_t|^{2})(\bar{\xi}_t-\psi_t-\varphi_t)\geq \frac{|u_t|^2}{\epsilon^2}(\varphi_t+\psi_t)\geq 0,
\end{equation}
while $\xi_0-\psi_0-\varphi_0=\xi_0-|\xi_0|\leq 0$. Hence, the parabolic maximum principle (for continuous weak solutions) implies the pointwise bound
\begin{equation}\label{xibd1}
\bar{\xi}_t\leq \varphi_t+\psi_t.
\end{equation}
We now use the following well-known asymptotics for the heat kernel on a compact manifold (see, e.g., \cite[Chapter~VI]{Chavel}).

\begin{lemma}\label{heat.limit}
	We have $(4\pi t)^{n/2}e^{d(x,y)^2/4t}K(t,x,y)\to 1$ as $t\to 0^+$, uniformly in $x,y\in M$.
\end{lemma}

In particular, since $|K(t,x,y)|\le C(\tau,M)$ for any $t\ge\tau>0$, one has
$$\int_M K(t,x,y)^p\,dy\leq C(p)\max\{t^{(1-p)n/2},1\}.$$
Since $|u_t|\le 1$ by Remark \ref{ut.bd}, we have automatically
$$\Big\|\frac{1}{\epsilon}(1-|u_t|^2)\Big\|_{L^{\infty}(M)}\leq \frac{1}{\epsilon} \quad \text{and} \quad \Big\|\frac{1}{\epsilon}(1-|u_t|^2)\Big\|_{L^2(M)}\leq 2\sqrt{\Lambda}$$
for every $t$, and interpolating we see that
$$\Big\|\frac{1}{\epsilon}(1-|u_t|^2)\Big\|_{L^q(M)}\leq C(M,\Lambda)\epsilon^{(2-q)/q}$$
for $2\le q\le\infty$.
It follows that, for $p\in (1,\frac{n}{n-2})$ with H\"{o}lder conjugate $q$,
\begin{align*}
	\psi(t,x)&\leq \int_0^t\|K(t-s,x,y)\|_{L^p(M)}C\epsilon^{(2-q)/q}\,ds\\
	&\leq C \epsilon^{(2-q)/q}\int_0^t (t-s)^{\frac{n(1-p)}{2p}}\,ds\\
	&\leq C(p,M,\Lambda)\epsilon^{(2-q)/q}\Big(\frac{n}{2p}-\frac{n-2}{2}\Big)^{-1}t^{\frac{n}{2p}-\frac{n-2}{2}},
\end{align*}
provided that $q\ge 2$.
In particular, taking
$p:=\frac{n-1}{n-2}\text{ and }q:=n-1$,
we arrive at an estimate of the form
$$\psi(t,x)\leq C_1(M,\Lambda) \epsilon^{\frac{3-n}{n-1}}t^{\frac{n-2}{2(n-1)}}.$$

Now, let
$$\eta_t:=\bar{\xi}_t-\varphi_t\leq \psi_t\leq C_1\epsilon^{\frac{3-n}{n-1}}t^{\frac{n-2}{2(n-1)}},$$
and setting
$$f_t:=\eta_t-C_1\epsilon^{\frac{3-n}{n-1}}t^{\frac{n-2}{2(n-1)}}(1-|u_t|^2),$$
note that
$$f_t\leq C_1\epsilon^{\frac{3-n}{n-1}}t^{\frac{n-2}{2(n-1)}}|u_t|^2$$
pointwise. On the other hand, recalling \eqref{xi.sub}, note that $f_t$ satisfies
\begin{align*}
	-(\partial_t+d^*d)f_t&\geq \frac{|u_t|^2}{\epsilon^2}\bar{\xi}_t-\frac{C_0}{2\epsilon}(1-|u_t|^2)
	+C_1\epsilon^{\frac{3-n}{n-1}}t^{\frac{n-2}{2(n-1)}}(2|\nabla_t u_t|^2-\epsilon^{-2}(1-|u_t|^2)|u_t|^2)\\
	&\geq \frac{|u_t|^2}{\epsilon^2}f_t-\frac{C_0}{2\epsilon}(1-|u_t|^2),
\end{align*}
and since $|u_t|^2\geq c\epsilon^{\frac{n-3}{n-1}}t^{\frac{2-n}{2(n-1)}}f_t$, it follows that on $\{f>0\}$ we have
$$-(\partial_t+d^*d)f_t\geq \epsilon^{-2}(c\epsilon^{\frac{n-3}{n-1}}t^{\frac{2-n}{2(n-1)}}f_t^2-C_0\epsilon).$$

Note that $f_0=\xi_0-|\xi_0|\le 0$. For any $\tau>0$, if $f$ has a positive maximum on $[0,\tau]\times M$ at some point $(t,x)$ with $t>0$, then the last weak subequation implies that here
\begin{align*}
	&c\epsilon^{\frac{n-3}{n-1}}t^{\frac{2-n}{2(n-1)}}f_t^2-C_0\epsilon\le 0,
\end{align*}
or equivalently
\begin{align*}
	&f_t
	\le C\epsilon^{\frac{1}{n-1}}t^{\frac{n-2}{4(n-1)}}
	\le C\epsilon^{\frac{1}{n-1}}\tau^{\frac{n-2}{4(n-1)}}.
\end{align*}
The same inequality holds then on all of $[0,\tau]\times M$. Since $\tau$ was arbitrary, we obtain
\begin{align*}
	&f_t\le C\epsilon^{\frac{1}{n-1}}t^{\frac{n-2}{4(n-1)}}
\end{align*}
for all $t\ge 0$. Recalling the definitions of $f$, $\eta$, $\bar{\xi}$ and $\varphi$, the preceding estimate tells us that
\begin{equation}
\xi_t\leq Ce^{Ct}\Big(\varphi_t+\epsilon^{\frac{1}{n-1}}+\epsilon^{\frac{2}{n-1}}\frac{1-|u|^2}{\epsilon}\Big),
\end{equation}
where $\varphi$ is the solution of the heat equation with initial data $\varphi_0=|\xi_0|$, for a constant $C=C(M,\Lambda)$. Finally, noting that
$$\varphi_t\leq C\|\xi_0\|_{L^1(M)}\leq C(M,\Lambda) \quad \text{for $t\geq 1,$} $$
it follows from the above that
\begin{equation}\label{main.xi.bd}
\xi_t\leq Ce^{Ct}(1+\epsilon^{\frac{2}{n-1}}\sqrt{e_{\epsilon}(u_t,\nabla_t)})\quad \text{for $t\geq 1$} .
\end{equation}

\subsection{Huisken-type monotonicity and $\bm{(n-2)}$-energy-density bounds along the flow}
As above, let $(u_t,\nabla_t)$ be a solution of the gradient flow with $E_{\epsilon}(u_0,\nabla_0)\leq\Lambda$ and $|u_0|\leq 1$. Mimicking the computations leading to Huisken's monotonicity for the mean curvature flow \cite{Huisk}, let us introduce $h(t,x)$, a positive solution of the \emph{backward} heat equation 
\begin{equation*}
\partial_th=d^*dh
\end{equation*}
on $[0,T)\times M$, with $\int_Mh=1$. Write $e_t:=e_{\epsilon}(u_t,\nabla_t)$ to lighten the notation and set 
\begin{equation*}
\Phi_h(t):=\int_M h e_t. 
\end{equation*}
Integration by parts combined with the gradient flow equations allows us to deduce that
\begin{align*}
	\Phi_h^\prime(t) &= \int_M(\partial_t h e_t+h \partial_t e_t)\\
	&=\int_M [(d^*dh) e_t+ h(2\langle \nabla \dot{u}-i\dot{\alpha}u,\nabla u\rangle+2\epsilon^2\langle d\dot{\alpha},d\alpha\rangle-\epsilon^{-2}(1-|u|^2)\langle u,\dot{u}\rangle)]\\
	&=\int_M [\langle dh,d e_t\rangle-2h(|\dot{u}|^2+\epsilon^2|\dot{\alpha}|^2)-2(\langle \nabla_{dh}u,\dot{u}\rangle+\epsilon^2d\alpha(dh,\dot{\alpha}))]
\end{align*}
(where we dropped the subscript $t$ from $u_t$, $\alpha_t$, $\dot{u}_t$, and $\dot{\alpha}_t$). Next, recall from \cite[Section~4]{PigatiStern} the stress-energy tensor
$$T_{\epsilon}(u,\nabla):=e_{\epsilon}(u,\nabla)g-2\nabla u^*\nabla u-2\epsilon^2d  \alpha^*d\alpha,$$ and note (cf.\ \cite[Section~4]{PigatiStern}) that we have the identities
\begin{align*}
	\operatorname{div}(T_{\epsilon})&=2\langle \nabla u,\nabla^*\nabla u\rangle+d\frac{W(u)}{\epsilon^2}
	+2\omega(\langle iu,\nabla u\rangle,\cdot)-2\epsilon^2\omega(d^*\omega,\cdot)\\
	&=-2\langle \nabla u,\dot{u}\rangle-2\epsilon^2 d\alpha(\cdot,\dot{\alpha}),
\end{align*}
where the second equality follows from \eqref{gf}. We can now rewrite the term $\langle dh, de_t\rangle$ in our computation of $\Phi_h'(t)$ as
$$\langle dh,d e_t\rangle=\langle dh, \operatorname{div}(T_{\epsilon})+2\operatorname{div}(\nabla u^*\nabla u+\epsilon^2d\alpha^*d\alpha)\rangle,$$
and apply the formula for $\operatorname{div}(T_{\epsilon})$ to see that
\begin{align*}
\Phi_h'(t)  & =  2\int_M \langle dh, \operatorname{div}(\nabla u^*\nabla u+\epsilon^2d\alpha^*d\alpha)\rangle \\
	& \quad -2 \int_M h ( |\dot{u}|^2+\epsilon^2|\dot{\alpha}|^2 ) - 4\int_M ( \langle \nabla_{dh}u,\dot{u}\rangle+\epsilon^2d\alpha(dh,\dot{\alpha})) \\
	& = - 2 \int \langle D^2h,\nabla u^*\nabla u+\epsilon^2d\alpha^*d\alpha\rangle \\
	& \quad - 2\int_M ( h|\dot{u}+h^{-1}\nabla_{dh}u|^2+\epsilon^2 h |\dot{\alpha}+h^{-1}\iota_{dh}d\alpha|^2 ) \\
	& \quad + 2\int_M h^{-1}( |\nabla_{dh}u|^2+\epsilon^2|\iota_{dh}d\alpha|^2 )  \\
	& \leq -2\int \langle D^2h,\nabla u^*\nabla u+\epsilon^2d\alpha^*d\alpha\rangle + 2\int_M h^{-1}(|\nabla_{dh}u|^2+\epsilon^2|\iota_{dh}d\alpha|^2). 
\end{align*}
Now, setting
$$P_t:=\nabla u^*\nabla u+\epsilon^2d\alpha^*d\alpha,$$ 
so that the stress-energy tensor $T_\epsilon(u, \nabla)$ becomes simply $e_\epsilon(u, \nabla)g - 2 P_t$, we can rewrite the preceding inequality as
\begin{equation}\label{phi.dot}
\Phi_h'(t)\leq -2\int_M \langle P_t,D^2h-h^{-1}dh\otimes dh\rangle.
\end{equation}
On the other hand, by Hamilton's \emph{matrix Harnack estimate} for the heat equation, see \cite[p.\ 132]{Hamilton}, there exist constants $C(M)$ and $B(M)$ such that, for $t\in [T-1,T)$,
$$D^2h-\frac{dh\otimes dh}{h}+\frac{1}{2(T-t)}hg\geq -C[(1+h\log(B/(T-t)^{n/2})]g.$$
Applying this in \eqref{phi.dot}, we see that for $t\in [T-1,T)$ the following inequality holds:
$$\Phi_h'(t)\leq \int_M\Big(\frac{h}{T-t}+C+Ch\log(B/(T-t)^{n/2})\Big)\langle P_t,g \rangle.$$
Now, recalling \eqref{main.xi.bd}, observe that
\begin{align*}
	\langle P_t,g\rangle&= |\nabla u|^2+2\epsilon^2|d\alpha|^2\\
	&=e_t+\epsilon^2|d\alpha|^2-\frac{(1-|u|^2)^2}{4\epsilon^2}\\
	&=e_t+\xi_t\Big(\epsilon |d\alpha|+\frac{1}{2\epsilon}(1-|u|^2)\Big)\\
	&\leq (1+Ce^{Ct}\epsilon^{\frac{2}{n-1}})e_t+Ce^{Ct}\sqrt{e_t}
\end{align*}
for $t\geq 1$.
In particular, setting $\alpha_n:=\frac{2}{n-1}$, for $T\in [2,3]$ and $t\in [T-1,T)$ it then follows that
\begin{align*}
	\Phi_h'(t)&\leq \frac{1+C\epsilon^{\alpha_n}}{T-t}\Phi_h(t)+\frac{C}{T-t}\int_M h\sqrt{e_t} + C\int_Me_t+C\log(B/(T-t)^{n/2})\Phi_h(t)\\
	&\leq \frac{1+C_2\epsilon^{\alpha_n}}{T-t}\Phi_h(t)+\frac{C_2}{T-t}\Phi_h(t)^{1/2} +C_2+C_2\log(B/(T-t)^{n/2})\Phi_h(t)
\end{align*}
for some $C_2(M,\Lambda)$,
where we also used the trivial inequality $\langle P_t,g\rangle \le 2e_t$.
Thus, setting
$$\Psi_h(t):=(T-t)^{1+C_2\epsilon^{\alpha_n}}e^{\zeta(t)}\Phi_h(t),$$
where $|\zeta(t)|\leq C(M,\Lambda)$ is the bounded function on $[T-1,T)$ given by
$$\zeta(t):=-\int_1^t C_2\log(B/(T-s)^{n/2})\,ds,$$
we see that
$$\Psi_h'(t) \leq C(T-t)^{-1/2}\Psi_h(t)^{1/2}+C$$
for $t\in [T-1,T)\subseteq [1,3)$.
From this differential inequality, we can conclude that 
\begin{equation}\label{weak.mono}
\Psi_h(t)\leq C(M,\Lambda)(\Psi_h(T-1)+1), 
\end{equation}
for any $t\in [T-1,T)\subseteq [1,3)$. 

Specializing, fix $T\in [2,3]$ and $x_0\in M$, and let 
$$h(t,x)=h_{T,x_0}(t,x):=K(T-t,x,x_0),$$
where $K$ is the heat kernel on $M$. Then, for $t\in [T-1,T)$, the inequality in \eqref{weak.mono} leads to an estimate of the form
\begin{align*}
	(T-t)^{1+C\epsilon^{\alpha_n}} & \int_M K(T-t,x,x_0)e_{\epsilon}(u_t,\nabla_t)\,dx\\ 
	&\leq C\int_M K(1,x,x_0)e_{\epsilon}(u_{T-1},\nabla_{T-1})+C\\
	&\leq CE_{\epsilon}(u_{T-1},\nabla_{T-1})+C\\
	&\leq C(M,\Lambda).
\end{align*}
In particular, taking $t:=2$ and $T:=2+\delta^2$ for $\delta\in (0,1]$, we see that
\begin{equation}
\delta^{2+2C\epsilon^{\alpha_n}}\int_M K(\delta^2,x,x_0)e_{\epsilon}(u_2,\nabla_2)\,dx\leq C(M,\Lambda).
\end{equation}
Since
$$\inf_{\epsilon\in (0,1]}\epsilon^{2C\epsilon^{\alpha_n}}=c(M,\Lambda)>0,$$
it follows that
\begin{equation}
\max_{\epsilon\leq \delta\leq 1}\Big(\delta^2\int_MK(\delta^2,x,x_0)e_{\epsilon}(u_2,\nabla_2)\,dx\Big)\leq C(M,\Lambda).
\end{equation}
Finally, using again Lemma \ref{heat.limit}, it follows that
$$\delta^{2-n}\int_{B_{\delta}(x_0)}e_{\epsilon}(u_2,\nabla_2)\leq C(M,\Lambda) \quad \text{for } \epsilon \leq \delta \leq 1.$$
Thus, we have arrived at the following bound.

\begin{proposition}\label{grad.dens.bd} If $(u_t,\nabla_t)$ is a solution of the gradient flow \eqref{gf} for $E_{\epsilon}$ with initial energy bound $E_{\epsilon}(u_0,\nabla_0)\leq \Lambda$, then at time $2$ the pair $(u_2,\nabla_2)$ satisfies
\begin{equation}
\int_{B_r(x_0)}e_{\epsilon}(u_2,\nabla_2)\leq C(M,\Lambda)r^{n-2}, 
\end{equation}
for all $r \in [\epsilon, 1]$ and $\ x_0\in M$.
\end{proposition}
Since $(u_2,\nabla_2)$ depends continuously on the initial couple $(u_0,\nabla_0)$, this provides in particular the regularization that we needed in the previous section.
\begin{remark}
Note that, in analogy with the monotonicity formula for critical couples, if we just used the trivial bound $\langle P_t,g\rangle\le 2e_t$ we would have obtained 
\begin{equation*}
\Phi_h'(t)\le\frac{2}{T-t}\Phi_h(t)+C+C\log(B/(T-t)^{n/2})\Phi_h(t),
\end{equation*}
leading to 
\begin{equation*}
(T-t)^2\int_M h_{T,x_0}(t,\cdot)e_t\le C(M,\Lambda)
\end{equation*}
and hence a non-sharp bound $C\delta^{n-4}$ for the energy of $(u_2,\nabla_2)$ on a ball $B_\delta(x_0)$. This would have sufficed for our present purposes (of ruling out concentration of mass in the min-max families) only when $n>4$.
\end{remark}


\subsection{Long-time existence of the gradient flow}
In this last part we show long-term existence, uniqueness and continuous dependence on initial conditions for the gradient flow of $E_\epsilon$, on the trivial line bundle. To do so, it is convenient to pass to the Coulomb gauge.
Namely, given a smooth couple $(u,\alpha)$, we can always find a change of gauge
\begin{equation} \label{gaugeParabolic}
(v,\beta)=(e^{i\theta}u,\alpha+d\theta) \quad \text{with $d^*\beta=0$}. 
\end{equation}
Indeed, it is enough to take a solution $\theta:M\to\R$ of
$d^*\alpha+d^*d\theta=0$, i.e., $\Delta_H \theta=-d^*\alpha$. The solution is unique once we impose $\int_M\theta=0$.

In the sequel, we denote $Q:=-\Delta_H ^{-1}d^*:\Omega^1(M)\to\Omega^0(M)$ the corresponding operator, with values into mean-zero functions. By standard elliptic regularity, this operator maps $H^k(M)$ continuously into $H^{k+1}(M)$, for any $k\in\N$.

Given a smooth solution $(u_t,\alpha_t)$ to the gradient flow equations, let $\theta_t=Q\alpha_t$. Omitting the time dependence and passing to the Coulomb gauge as in \eqref{gaugeParabolic} 
we get $\dot\theta=\epsilon^{-2}Q\ang{iu,\nabla u}$. Thus, setting $\tilde\nabla:=d-i\beta=\nabla-id\theta$, we obtain 
\begin{align*}
\dot{\beta}
&=\dot{\alpha}+d\dot\theta \\
&=-d^*d\alpha+\epsilon^{-2}(\ang{iu,\nabla u}+dQ\ang{iu,\nabla u}) \\
&=-\Delta_H\beta+\epsilon^{-2}(\ang{iv,\tilde\nabla v}+dQ\ang{iv,\tilde\nabla v}),
\end{align*}
since by gauge invariance $d^*d\alpha=d^*d\beta=\Delta_H\beta$ and $\ang{iu,\nabla u}=\ang{iv,\tilde\nabla v}$. Similarly,
\begin{align*}
\dot{v}
&=e^{i\theta}\dot{u}+ie^{i\theta}\dot{\theta}u \\
&=-\tilde\nabla^*\tilde\nabla v+\frac{1}{2\epsilon^2}(1-|v|^2)v+\epsilon^{-2}(Q\ang{iv,\tilde\nabla v})iv.
\end{align*}
Let $P:\Omega^1(M)\to\Omega^1(M)$ denote the Hodge projection on the co-closed part of a one-form.
Since $-dQ\lambda$ equals the exact part of $\lambda$, we have $\lambda+dQ\lambda=P\lambda$ for any $\lambda\in\Omega^1(M)$.
Thus, expanding $\tilde\nabla^*\tilde\nabla$ in terms of $\beta$, the equations \eqref{gf} give the new system
\begin{equation} \label{parab.coul}
\left\{
\begin{aligned}
\dot{v}+d^*dv &= -2i\ang{\beta,dv}-|\beta|^2v+{\textstyle\frac{1}{2\epsilon^2}}(1-|v|^2)v+\epsilon^{-2}(Q\ang{iv,dv-i\beta v})iv, \\
\dot{\beta}+\Delta_H\beta &= \epsilon^{-2}P\ang{iv,dv-i\beta v}.
\end{aligned}
\right.
\end{equation}

Conversely, given a couple $(u_0,\alpha_0)$ and setting $\theta_0:=Q\alpha_0$, from a smooth solution $(v_t,\beta_t)$ of \eqref{parab.coul} with initial condition $(e^{i\theta_0}u_0,\alpha_0+d\theta_0)$ one recovers a smooth solution $(u_t,\alpha_t)$ to the original system \eqref{gf}, by letting $\theta=\theta_t$ solve $\dot\theta_t=\epsilon^{-2}Q\ang{iv_t,(d-i\beta_t)v_t}$, and setting $(u,\alpha):=(e^{-i\theta}v,\beta-d\theta)$.

Thus, we reduce ourselves to establishing the desired long-term existence, uniqueness and continuous dependence for \eqref{parab.coul}. We will use the following classical fact from the theory of linear parabolic equations.

\begin{lemma}\label{ibp.bound}
Given $f_t\in\Omega^\ell(M)$ smooth on $[0,T]\times M$, with $0<T\le 1$, the (unique) solution $w_t$ to $\de_tw_t+\Delta_Hw_t=f_t$ with initial condition $w_0=0$ satisfies 
\begin{equation*}
\|w\|_{C^0([0,T],H^{k+1}(M))}\le C(k,\ell,M)\|f\|_{L^2([0,T],H^k(M))}, 
\end{equation*}
where the norms are shorthand for $\max_{t\in[0,T]}\|w_t\|_{H^{k+1}(M)}$ and $(\int_0^T\|f_t\|_{H^k(M)}^2\,dt)^{1/2}$.
\end{lemma}
As a consequence, we get a well-defined operator 
\begin{equation*}
T_{\ell,k}:L^2([0,T],H^k(M))\to C^0([0,T],H^{k+1}(M))
\end{equation*}
mapping $f$ to $w$.

Using this lemma, short-time existence and uniqueness easily follow using the Banach fixed-point theorem. Namely, fix an integer $k>\frac{n}{2}$ and, given a smooth initial condition $(v_0,\beta_0)$,
let $w^0$ denote the constant couple $w^0_t=(v_0,\beta_0)$.
For $R>0$, the subset $S$ of
\begin{align*}
&Y_T:=C^0([0,T],H^{k+1}(M)\times H^k(M)),
\end{align*}
given by the couples $w_t$ with initial value $w_0=(v_0,\beta_0)$ and $\|w-w^0\|_{Y_T}\le R$, forms a complete metric space with the distance induced by $Y_T$. To any $w=(v,\beta)\in S$ we can associate the solution $F(w)=(v',\beta')$ of
\begin{align*}
\left\{
\begin{aligned}
\dot{v'}+d^*dv'&=-2i\ang{\beta,dv}-|\beta|^2v+{\textstyle\frac{1}{2\epsilon^2}}(1-|v|^2)v+\epsilon^{-2}(Q\ang{iv,dv-i\beta v})iv, \\
\dot{\beta'}+\Delta_H\beta'&=\epsilon^{-2}P\ang{iv,dv-i\beta v}.
\end{aligned}
\right.
\end{align*}
Denoting $G(w_t)$ and $H(w_t)$ the right-hand sides of the two equations,
note that they belong to $C^0([0,T],H^k(M))$, since $H^k(M)$ is an algebra and $P$ and $Q$ map $H^k(M)$ into itself.
Hence, $F(w)\in Y_T$ is well-defined. 
For the same reason, letting $R':=R+\|v_0\|_{H^{k+1}(M)}+\|\beta_0\|_{H^k(M)}$, note that for a fixed $t \in [0,T]$ we have 
\begin{equation*}
\|G(w_t^1)-G(w_t^2)\|_{H^{k}(M)} \leq C(M,R') \|w_t^1-w_t^2\|_{H^{k+1}(M) \times H^k(M)}
\end{equation*}
and similarly
\begin{equation*}
\|H(w_t^1)-H(w_t^2)\|_{H^{k}(M)} \leq C(M,R') \|w_t^1-w_t^2\|_{H^{k+1}(M) \times H^k(M)},
\end{equation*}
whenever $w^1,w^2\in S$. As a consequence, Lemma \ref{ibp.bound} gives
\begin{equation*}
\|F(w^1)-F(w^2)\|_{Y_T}\le C(M,R')\sqrt{T}\|w^1-w^2\|_{Y_T}.
\end{equation*} 
Hence, for $T$ small enough, we have $\|F(w^1)-F(w^2)\|_{Y_T}\le\frac{1}{2}\|w^1-w^2\|_{Y_T}$ and, by continuity, $\|F(w^0)-w^0\|_{Y_T}\le R/2$; in particular, 
\begin{equation*}
\|F(w)-w^0\|_{Y_T}\le\|F(w)-F(w^0)\|_{Y_T}+\|F(w^0)-w^0\|_{Y_T}\le R
\end{equation*}
for $w\in S$, and thus $F(w)\in S$ as well. The Banach fixed-point theorem applies and gives a unique $w\in S$ with $F(w)=w$, as desired. Since $R$ was arbitrary, this also establishes uniqueness in this regularity class.

Let $[0,\bar T)$ be the maximal time of existence in the same class. From standard $L^2$ regularity theory for linear parabolic equations, it then follows that the solution $(v,\beta)$ is smooth on $[0,\bar T)\times M$.

We shall now prove long-time existence of the flow. Assume by contradiction that $\bar T<\infty$. As we already saw above, the corresponding solution $(u,\alpha)$ to the original system \eqref{gf} satisfies 
\begin{equation*}
\sup_{[0,\bar T)\times M}|d\alpha|<\infty. 
\end{equation*}
In a similar fashion, we can derive a bound for $|\nabla u|$. Indeed, as in \cite[Section~3]{PigatiStern}, we have the Bochner identity
\begin{align*}
&-(\de_t+dd^*)\frac{1}{2}|\nabla u|^2
=|\nabla^2u|^2+\frac{3|u|^2-1}{2\epsilon^2}|\nabla u|^2-2\ang{\omega,\psi(u,\nabla)}+\mathcal{R}_1(\nabla u,\nabla u)
\end{align*}
and, in particular, using the bound $|\psi(u,\nabla)|\le|\nabla u|^2$, we easily deduce the weak subequation
\begin{align*}
-(\de_t+d^*d)|\nabla u|
&\ge\frac{3|u|^2-1}{2\epsilon^2}|\nabla u|-2|\omega||\nabla u|-C(M)|\nabla u|. 
\end{align*}
Recalling that 
\begin{equation*}
-(\de_t+d^*d)\frac{1-|u|^2}{\epsilon}=\frac{|u|^2}{\epsilon^2}\frac{1-|u|^2}{\epsilon}-\frac{2}{\epsilon}|\nabla u|^2,
\end{equation*}
we obtain for the difference $w:=|\nabla u|-\frac{1-|u|^2}{\epsilon}$ that
\begin{align*}
&-(\de_t+d^*d)w
\ge\frac{|u|^2}{\epsilon^2}w+|\nabla u|\Big(\frac{2}{\epsilon}|\nabla u|-\frac{1-|u|^2}{2\epsilon^2}-2|\omega|-C(M)\Big).
\end{align*}
For any $0<\tau<\bar T$, if $w$ attains a positive maximum on $[0,\tau]\times M$ at some point $(t,x)$ with $t>0$,
it then follows that here
\begin{align*}
&\frac{2}{\epsilon}|\nabla u|\le\frac{1-|u|^2}{2\epsilon^2}+2|\omega|+C(M).
\end{align*}
Hence, 
\begin{equation*}
|\nabla u|
\le \frac{1}{\epsilon}+\sup_{[0,\bar T)\times M} w
\le \frac{2}{\epsilon}+\epsilon\sup_{[0,\bar T)\times M}|\omega|+\frac{\epsilon}{2}C(M)+\|\nabla_0u_0\|_{L^{\infty}(M)}
\end{equation*}
on all of $[0,\bar T)\times M$. By gauge invariance, we then get
\begin{align*}
&\sup_{[0,\bar T)\times M}|d\beta| < \infty \quad \text{and} \quad \sup_{[0,\bar T)\times M}|dv-i\beta v|<\infty.
\end{align*}
In particular, the co-exact part of $\beta_t$ is also bounded.
From \eqref{gf} it follows that
\begin{align*}
&\int_M(|\dot u_t|^2+\epsilon^2|\dot\alpha_t|^2)=-\frac{1}{2}\frac{d}{dt}E_\epsilon(u_t,\alpha_t),
\end{align*}
from which we deduce the bound $\int_0^{\bar T}\int_M|\dot\alpha|^2<\infty$ just by integrating the above expression. In particular, $\dot\alpha\in L^1([0,\bar T],L^2(M))$, giving $\alpha\in C^0([0,\bar T],L^2(M))$. Thus, the harmonic part $\alpha^h_t$ in the Hodge decomposition of $\alpha_t$ stays bounded. Since $\beta_t^h=\alpha_t^h$ and $\beta$ has no exact part, this implies that
\begin{align*}
&\sup_{[0,\bar T)\times M}|\beta|<\infty.
\end{align*}
Also, note that $|v|=|u|\le 1$ as a simple application of the maximum principle to the equation satisfied by $|u|^2$, provided $|u_0|\le 1$, implying 
\begin{equation*}
\sup_{[0,\bar T)\times M}|dv|<\infty.
\end{equation*}
From $L^p$ regularity theory (see, e.g., \cite{Schlag}), it follows that $v,\beta\in L^p([0,\bar T],W^{k,p}(M))$ for all $k\in\N$, $1<p<\infty$ and, hence, $v$ and $\beta$ extend smoothly to $[0,\bar T]\times M$. Since we can extend the solution past $\bar T$, we arrive at a contradiction. This shows that $\bar T=\infty$.
Finally, continuous dependence (in the smooth topology) on the initial condition for the system \eqref{gf} follows from the same property for \eqref{parab.coul}.

\frenchspacing

\nocite{*}
\printbibliography

\end{document}